\documentclass[sn-mathphys-num]{sn-jnl-arxiv}

\usepackage{multicol}
\usepackage{graphicx}%
\usepackage{multirow}%
\usepackage{amsmath,amssymb,amsfonts}%
\usepackage{amsthm}%
\usepackage{mathrsfs}%
\usepackage[title]{appendix}%
\usepackage{xcolor}%
\usepackage{textcomp}%
\usepackage{manyfoot}%
\usepackage{booktabs}%
\usepackage{algorithm}%
\usepackage{algorithmicx}%
\usepackage{algpseudocode}%
\usepackage{listings}%
\usepackage{enumitem}
\usepackage{cases}
\usepackage{comment}

\theoremstyle{thmstyleone}%
\newtheorem{theorem}{Theorem}[section]
\newtheorem{proposition}[theorem]{Proposition}%

\newtheorem{lemma}[theorem]{Lemma}%
\newtheorem{corollary}[theorem]{Corollary}%

\theoremstyle{thmstyletwo}%
\newtheorem{example}[theorem]{Example}%
\newtheorem{remark}{Remark}%
\newtheorem{problem}{Problem}

\theoremstyle{thmstylethree}%
\newtheorem{definition}[theorem]{Definition}%

\usepackage{pgf,tikz}

\usetikzlibrary{calc,positioning,shapes,arrows.meta,fit}

\tikzset{%
 shaded/.style={draw, shape=circle, fill=black!35, inner sep=1.4pt},
 unshaded/.style={draw, shape=circle, fill=white, inner sep=1.4pt},
 quasi/.style={draw, shape=rectangle, rounded corners=3pt, fill=white, inner sep=2.5pt, minimum height=14.5pt},
 blob/.style={draw, shape=rectangle, rounded corners=12pt, thin, densely dotted},
 arrow/.style={->, thin, >=latex, shorten >=2.5pt, shorten <=2.5pt},
 order/.style={thin},
 curvy/.style={thin, looseness=1.2, bend angle=70},
 fatcurvy/.style={thin, looseness=1.7, bend angle=75},
 label/.style={shape=rectangle, inner sep=6pt},
 auto}

\newcommand{\koflbf}{\mathbf{K}[\mathbf{L}]}

\newcommand{\Lalg}{\mathbf{L}}
\newcommand{\UpE}{\mathsf{Up}(\mathbf{E})}

\newcommand{\K}{\mathbf{K}}

\font\bmi=cmmi8 scaled 1440
\newcommand{\powerset}{\raise.6ex\hbox{\bmi\char'175 }}

\renewcommand{\smile}{\smallsmile}

\begin{document}

\title[Representability for DqRAs]{Representability for distributive quasi relation algebras via nested sums}


\author*[1,2]{\fnm{Andrew} \sur{Craig}}\email{acraig@uj.ac.za}

\author[1]{\fnm{Wilmari} \sur{Morton}}\email{wmorton@uj.ac.za}

\author[1]{\fnm{Claudette} \sur{Robinson}}\email{claudetter@uj.ac.za}

\affil*[1]{\orgdiv{Department of Mathematics and Applied Mathematics}, \orgname{University of Johannesburg}, \orgaddress{\street{PO Box 524}, \city{Auckland Park}, \postcode{2006}, \country{South Africa}}}

\affil[2]{\orgdiv{National Institute for Theoretical and Computational Sciences (NITheCS)}, \orgaddress{\city{Johannesburg}, \country{South Africa}}}

\abstract{We extend the work of Galatos (2004) on nested sums, originally called generalised ordinal sums, of residuated lattices. We show that the nested sum of an odd quasi relation algebra (qRA) satisfying certain conditions and an arbitrary qRA is again a qRA. In a recent paper by 
Craig and Robinson (2025)
the notion of representability for distributive quasi relation algebras (DqRAs) was developed. For certain
pairs of representable DqRAs, we prove that their nested sum is again representable. An important consequence of this result is that finite
Sugihara chains are finitely representable.}




\maketitle

Quasi relation algebras (qRAs) were first described by Galatos and Jipsen~\cite{GJ13}. On the one hand they can be viewed as a generalisation of relation algebras, or as involutive FL-algebras with an additional negation-like unary operation. Unlike relation algebras, the variety of qRAs has a decidable equational theory.

In this paper, we first extend the so-called {\em nested sum}, originally called the \emph{generalised ordinal sum} by Galatos~\cite{G04}, from residuated lattices to quasi relation algebras.  
Since the 
construction for residuated lattices has appeared in tandem with ``ordinal sums of posets", recent papers on the topic (see eg.~\cite{Sant24}) have chosen to use the term nested sum. 
When thinking about the order structure of the lattices, the construction can be considered in the following way: a residuated lattice $\Lalg$ is inserted into a residuated lattice $\K$ where it replaces the unit of the monoid operation. The requirements placed on $\Lalg$ and $\K$ are discussed in Section~\ref{sec:KL-construction}.
We essentially extend the 
conditions for residuated lattices to the additional unary operations in the signature of qRAs. 
A consequence of the conditions on $\mathbf{K}$ is that its FL-algebra reduct will be odd (i.e. $0=1$). 

Having shown that the $\mathbf{K[L]}$ construction gives us a method for constructing new qRAs from existing ones, we then turn to the question of representability. 

Abstract relation algebras, first described by Tarski in 1941~\cite{Tar41}, were designed to provide an abstract algebraic setting for studying binary relations. J\'{o}nsson and Tarski~\cite{JT48} asked whether all relation algebras were isomorphic to algebras of binary relations (with the monoid operation given by relational composition). 
This question was answered in the negative by Lyndon~\cite{Lyn50}. 
The task of representing relation algebras as algebras of binary relations has been an intriguing area of research ever since. 

In~\cite{CR-RDqRA}, two of the current authors gave a definition of \emph{representability}  for distributive quasi relation algebras 
(see Section~\ref{sec:RDqRA}). Partially ordered sets equipped with an equivalence relation (and satisfying certain symmetry requirements) 
can be used to build distributive quasi relation algebras (DqRAs) as algebras of binary relations. This method was further exploited to give relational representations of Sugihara monoids~\cite{CR-Sugihara}. 

In Section~\ref{sec:KL-rep} we show that if $\mathbf{K}=\mathbf{S}_3$, and if $\mathbf{L}$ is representable, then their nested sum $\mathbf{K[L]}$ will again be representable (Theorem~\ref{Thm:K[L]represenatble}). 
This is extended in Corollary~\ref{cor:Sn[L]-rep} to the case of $\mathbf{K}$ being any finite odd Sugihara chain.
These results are similar
in spirit
to  the result  for integral relation algebras where the relation algebra obtained from $\mathbf{A}$ and $\mathbf{B}$
by Comer's construction \cite{Com83}
is representable if and only if  $\mathbf{A}$ and $\mathbf{B}$ are representable.  



A further application of 
Theorem~\ref{Thm:K[L]represenatble} is given in Section~\ref{sec:applications}, where we show that all finite Sugihara chains can be represented using a set of binary relations on a \emph{finite} poset (Theorem~\ref{thm:fin-rep-Sn}).
This improves the results in \cite{Mad2010} and~\cite{CR-Sugihara} where infinite posets (using copies of $\mathbb{Q}$, the rational numbers) were required to obtain a relational representation of $\mathbf{S}_n$ for $n \geqslant 4$.


In Problem~\ref{prob:K=Snalways?} we pose an intriguing question about the structures of possible algebras $\mathbf{K}$ for which we can guarantee that $\koflbf$ will be representable. Specifically we ask  whether $\K$ must be a finite Sugihara chain.

\section{Preliminaries}\label{sec:Prelim}
In this section we recall the basic definitions of the algebras
used in this paper.  The reader is referred to~\cite{GJKO} 
for more information. 

\subsection{Quasi relation algebras}\label{subsec:qRAs}

A quasi relation algebra~\cite{GJ13} is an expansion of a residuated lattice.  Recall
that a {\em residuated lattice} (RL) is an algebra 
$\mathbf{A}=\langle A,\wedge,\vee, \cdot,\backslash,/,1\rangle$ such that 
$\langle A,\cdot,1\rangle$ is a monoid, $\langle A,\wedge,\vee\rangle$
is
a lattice and the monoid operation $\cdot$ is residuated with
residuals $\backslash$ and $/$, i.e., for all $a,b,c\in A$,
\[a\cdot b\leqslant c\quad\iff\quad 
a\leqslant c/b\quad\iff\quad
b\leqslant a\backslash c.\]
A \emph{distributive} residuated lattice $\mathbf{A}$  is a residuated lattice whose underlying lattice is distributive.
Recall also that a residuated lattice $\mathbf A$  is \emph{commutative} if $a \cdot b = b \cdot a$ for all $a, b \in A$. An element $a \in A$ is said to be \emph{idempotent} if $a\cdot a = a^2 = a$. 
If every
element of $\mathbf{A}$ is idempotent, then
$\mathbf{A}$ is said to be {\em idempotent}.

If $a\leqslant 1$ or $1\leqslant a$ for all $a\in A$, 
then $\mathbf{A}$ is called {\em conic}.
We note that when the term conic first appeared 
in~\cite[Definition 2.3]{HR09} it was applied to 
\emph{commutative} residuated pomonoids, 
but since then the term has been used more generally.

A residuated lattice expanded with a constant $0$, that is, 
$\mathbf{A}=\langle A,\wedge,\vee,\cdot,\backslash,/,1,0\rangle$
is called a {\em Full Lambek {\upshape (}FL-{\upshape )}algebra} (cf.~\cite[Chapter~2.2]{GJKO}).  
No additional properties are assumed about the constant $0$. An FL-algebra satisfying $1=0$ is called {\em odd}.

On an FL-algebra, two 
unary operations ${\sim}:A\to A$ and $-:A\to A$, called {\em linear negations},
are defined in terms of the residuals and $0$ as follows.  For $a\in A$,
\[{\sim} a=a\backslash 0 \quad \text{and}\quad {-}a=0/a.\]
It follows that ${\sim}1 = 1\backslash 0 = 0= 0/1=-1$.
By residuation, 
${\sim} (a\vee b)={\sim}a \wedge {\sim}b$ and ${-}(a\vee b)=
{-a}\wedge{-}b$ for all $a,b\in A$.


We say that an FL-algebra $\mathbf{A}$ is \emph{cyclic} if $-a = {\sim} a$ for all $a \in A$. It is easy to see that if an FL-algebra $\mathbf{A}$ is commutative, then it is cyclic. 

If an FL-algbra $\mathbf{A}$ satisfies the condition 
\[\textsf{(In)}:\quad   {\sim}{-}a=a={-}{\sim}a,\text{ for all }a\in A,\]
then it is called an {\em involutive Full Lambek {\upshape(}InFL-{\upshape)}algebra}.
In an InFL-algebra $\mathbf{A}$ we have that 
$a\leqslant b\iff a\,\cdot({\sim} b)\leqslant -1\iff (-b)\cdot a\leqslant-1$,
for all $a,b\in A$. Moreover, since both $-$ and $\sim$ are order-reversing, 
we get that 
$-$ and $\sim$ are dual lattice isomorphisms in an InFL-algebra. 

The dual of $\cdot$ is the binary operation $+:A\times A\to A$
defined by $a+b={\sim}(-a\cdot -b)$ for
all $a,b\in A$. In an InFL-algebra we have that 
$a+b=-({\sim} b\cdot {\sim} a)$ holds. 

In~\cite[Lemma 2.2]{GJ13} Galatos and Jipsen show that 
an InFL-algebra is term-equivalent to an algebra $\mathbf{A}=\langle A, \wedge,\vee,\cdot,{\sim},{-},1\rangle$ such that $\langle A, \wedge,\vee\rangle$ is a lattice, $\langle A, \cdot, 1\rangle$ is a monoid, and for all $a, b, c \in A$, we have
\[a\cdot b\leqslant c\quad\iff\quad 
a\leqslant -\left(b\cdot {\sim}c\right)\quad\iff\quad
b\leqslant {\sim}\left(-c\cdot a\right).\]
We will usually write $\mathbf A = \langle A,\wedge, \vee, \cdot, {\sim}, {-}, 1, 0\rangle$ for an InFL-algebra, 
where $0 = -1 ={\sim}1$. The residuals can be expressed in terms of $\cdot$ and the linear negations as follows:
\begin{equation}
    c/b = -\left(b\cdot {\sim}c\right)
    \quad\text{and}\quad
    a\backslash c = {\sim}\left(-c\cdot a\right).\label{Eqn:ResidualsFromNegations}
\end{equation}

An InFL-algebra $\mathbf{A}$ can be expanded with an additional
unary operation ${\neg}:A\to A$ 
to form an {\em InFL$'$-algebra} 
$\mathbf{A}=\langle A, \wedge,\vee,\cdot,\sim,{-},{\neg},1,0\rangle$
such that $\neg\neg a=a$ for all $a\in A$. 
We note that the involution is usually denoted by a prime, i.e., $'$, but 
for readability when using subscripts we will use $\neg$.

If an InFL$'$-algebra $\mathbf{A}$
additionally satisfies the De Morgan law
\[\textsf{(Dm)}:\quad\neg (a\vee b)=\neg a\wedge \neg b,\text{ for all }a,b,\in A,\]
then $\mathbf{A}$ is called a {\em DmInFL$'$-algebra}.  
A {\em quasi relation algebra} (qRA) is a DmInFL$'$-algebra 
$\mathbf{A}=\langle A,\wedge,\vee,\cdot,{\sim},{-},{\neg},1,0\rangle$
that satisfies, for all $a,b\in A$,
\[\textsf{(Di)}:\quad \neg({\sim} a)=-(\neg a)
\quad\text{and}\quad
\textsf{(Dp)}:\quad\neg (a\cdot b)=\neg a+\neg b.\]
The abbreviations \textsf{(Di)} and \textsf{(Dp)} stand for De Morgan involution and
De Morgan product, respectively.

We note that if $\mathbf{A}$ is a qRA, then it can be shown that $\neg 1 = - 1= {\sim}1 = 0$. 
An equivalent definition of a qRA is an FL$'$-algebra (i.e. an FL-algebra with an involutive unary operation $'$) that satisfies \textsf{(Dm)}, \textsf{(Di)} and \textsf{(Dp)}. Notice that if a DmFL$'$-algebra satisfies \textsf{(In)} and \textsf{(Dp)}, then it can be shown that it also satisfies \textsf{(Di)} and therefore is a qRA (cf.~\cite[Lemma 1]{CJR24}).

Finally, a {\em distributive quasi relation
algebra} (DqRA) is a quasi relation algebra $\mathbf{A}=\langle A,\wedge,\vee,\cdot,{\sim},{-},{\neg},1,0\rangle$
such that the underlying lattice $\langle A,\wedge,\vee\rangle$
is distributive. 

In Figure~\ref{fig:qRAs-examples} we give examples of small cyclic qRAs. Black nodes denote idempotent elements. The first diamond, $\mathbf{L}_2$, in Figure~\ref{fig:qRAs-examples} is term equivalent to the four element Boolean algebra (where the monoid
operation is given by the meet). The second algebra, $\mathbf{K}_2$ is an example of an odd qRA (hence $-1={\sim}1=\neg 1=0$).  Note that the lattice reduct need not be distributive
as can be seen from the third algebra in Figure~\ref{fig:qRAs-examples}.  In~\cite[Figure~1]{CR-RDqRA} two examples of non-cyclic qRAs are given. All of these algebras were found
with the help of Prover9/Mace4~\cite{P9M4}. 

\begin{figure}[h!]

\centering
\begin{tikzpicture}

\begin{scope}
\node[draw,circle,inner sep=1.5pt,fill] (0) at (0,0) {};
\node[draw,circle,inner sep=1.5pt,fill] (a) at (-1,1) {};
\node[draw,circle,inner sep=1.5pt,fill] (b) at (1,1) {};
\node[draw,circle,inner sep=1.5pt,fill] (1) at (0,2) {};
\draw[order] (0)--(a)--(1);
\draw[order] (0)--(b)--(1);
\node[label,anchor=south] at (1) {$1$};
\node[label,anchor=east,xshift=1pt] at (a) {$a$};
\node[label,anchor=west,xshift=-1pt] at (b) {$b$};
\node[label,anchor=north] at (0) {$0$};
\node[label,anchor=south,xshift=2pt,yshift=-94pt] at (1) {$\mathbf{L}_1$};
\end{scope}

\begin{scope}[yshift=-4cm]
\node (somenode) at (0,2){
\begin{tabular}{|c||c|c|c|c|c|}
\hline $\cdot$ & $a$  & $b$ & $\neg $ & $\sim$ & $-$ \\\hline
\hline $a$ & $a$ & $0$ & $b$ & $b$ & $b$ \\
\hline $b$ & $0$ & $b$ & $a$ & $a$ & $a$ \\
\hline
\end{tabular}}; 
\end{scope}

\begin{scope}[xshift=4cm]
\node[draw,circle,inner sep=1.5pt,fill] (0) at (0,0) {};
\node[draw,circle,inner sep=1.5pt] (a) at (-1,1) {};
\node[draw,circle,inner sep=1.5pt,fill] (b) at (1,1) {};
\node[draw,circle,inner sep=1.5pt,fill] (1) at (0,2) {};
\draw[order] (0)--(a)--(1);
\draw[order] (0)--(b)--(1);
\node[label,anchor=south] at (1) {$\top$};
\node[label,anchor=east,xshift=1pt] at (a) {$a$};
\node[label,anchor=west,xshift=-1pt] at (b) {$1=0$};
\node[label,anchor=north] at (0) {$\bot$};
\node[label,anchor=south,xshift=2pt,yshift=-94pt] at (1) {$\mathbf{K}_2$};
\end{scope}

\begin{scope}[xshift=4cm,yshift=-4cm]
\node (somenode) at (0,2){
\begin{tabular}{|c||c|c|}
\hline $\cdot$ & $a$  & $\top$  \\\hline
\hline $a$ & $\bot$ & $a$  \\
\hline $\top$ & $a$ & $\top$  \\
\hline
\end{tabular}}; 
\end{scope}

\begin{scope}[xshift=7cm]
\node[draw,circle,inner sep=1.5pt,fill] (c0) at (1,0) {};
\node[draw,circle,inner sep=1.5pt,fill] (c1) at (1,3) {};
\node[draw,circle,inner sep=1.5pt,fill] (a0) at (2,1) {};
\node[draw,circle,inner sep=1.5pt,fill] (a1) at (2,2) {};
\node[draw,circle,inner sep=1.5pt] (c2) at (0,1.5) {};
\draw[order] (c0)--(c2);
\draw[order] (c0)--(a0);
\draw[order] (c2)--(c1);
\draw[order] (a0)--(a1);
\draw[order] (a1)--(c1);
\node[label,anchor=east,xshift=1pt] at (c0) {$\bot$};
\node[label,anchor=east,xshift=1pt] at (c1) {$\top$};
\node[label,anchor=east,xshift=1pt] at (c2) {$a$};
\node[label,anchor=west,xshift=-1pt] at (a0) {$0$};
\node[label,anchor=west,xshift=-1pt] at (a1) {$1$};
\end{scope}

\begin{scope}[yshift=-4cm, xshift=8cm]
\node (somenode) at (0,2){
\begin{tabular}{|c||c|c|c|c|c|}
\hline $\cdot$ & $0$  & $a$ & $\neg $ & $\sim$ & $-$ \\\hline
\hline $0$ & $0$ & $a$ & $1$ & $1$ & $1$ \\
\hline $a$ & $a$ & $\bot$ & $a$ & $a$ & $a$ \\
\hline
\end{tabular}}; 
\end{scope}

%
	

\end{tikzpicture}
\caption{Examples of small qRAs.}
\label{fig:qRAs-examples}
\end{figure}

An important class of examples of (distributive) quasi relation algebras is covered in more detail in the next subsection. 

\subsection{Sugihara chains}\label{sec:Sugihara} 

In Section~\ref{sec:KL-construction} we will describe a construction,
called the nested sum, that 
`inserts' one algebra into another.
In Section~\ref{sec:KL-rep} we show that if the algebras
satisfy certain conditions and are representable, then 
their nested sum is representable too.
Specifically, we will show that this is the case
if the algebra into which the other is `inserted' 
is a finite odd Sugihara chain.  We therefore recall
the basic definitions related to finite odd Sugihara chains.

It is well known that a residuated lattice is commutative iff it satisfies the identity 
$x\backslash y\approx x/y$. Therefore,
keeping with convention, for a commutative residuated lattice $\mathbf{S}$ we use the signature $\mathbf{S}=\langle S, \wedge,\vee, \cdot,\to,1\rangle$
where ${\to}:= \backslash$.
We will also consider expansions of 
commutative residuated lattices by a single linear negation,
$\mathbf{S}=\langle S, \wedge,\vee, \cdot,\to,{\sim},1\rangle$, such that 
${\sim}{\sim} a=a$ and $a\to({\sim}b)=
b\to ({\sim}a)$ for all $a,b\in S$, 
to form an involutive commutative residuated lattice $\mathbf{S}$.
A {\em Sugihara monoid} 
$\mathbf{S}=\langle S, \wedge,\vee, \cdot,\to,{\sim},1\rangle$
is an involutive commutative distributive idempotent residuated lattice.
A Sugihara chain is a Sugihara monoid such that the
underlying lattice $\langle S,\wedge,\vee\rangle$ is linearly ordered.

In the sequel we will only consider finite Sugihara chains. 
These are algebras 
$\mathbf{S}_n=\langle S_n, \wedge, \vee, \cdot, \to, {\sim},1\rangle$ for $n \in \omega$
and $n \geqslant 2$. If $n=2k$ for $k>0$ then 
$S_n=\{a_{-k}, \dots, a_{-1}, a_1, \dots, a_k\}$, and if $n=2k+1$ for $k>0$ then 
$S_n=\{a_{-k}, \dots, a_{-1}, a_0, a_1, \dots a_k\}$. 
The meet and join are defined as expected: 
$a_i \wedge a_j = a_{\text{min}\{i,j\}}$ and $a_i \vee a_j = a_{\text{max}\{i,j\}}$. 
The unary operation ${\sim}$ is defined by ${\sim}a_j = a_{-j}$. 
The monoid operation is defined by
$$a_i \cdot a_j = \begin{cases}
a_i & \text{if }\, |j|<|i| \\ 
a_j & \text{if }\, |i|< |j| \\ 
a_{\text{min}\{i,j\}} & \text{if }\, |j|=|i|.
\end{cases}$$
The implication $a_i \to a_j$ is defined by $
{\sim}a_i \vee a_j$ if $i \leqslant j$ and $
{\sim}a_i \wedge a_j$ if $i >j$. 
If $n$ is odd then, the monoid identity is given by $1=a_0$,
and if $n$ is even, then the monoid identity is $1=a_1$. See Figure~\ref{fig:Sugihara-examples} for examples.

Given a finite Sugihara chain $\mathbf{S}_n$, we can set
$0={\sim} 1$ and $-a={\sim}a$ and $\neg a={\sim}a$, which gives
rise to a (distributive) quasi relation algebra. Expanding
the signature of the reduct $\langle S_n, \wedge, \vee, \cdot, {\sim}, 1\rangle$, we can view $\mathbf{S}_n$ as the DqRA 
$\langle S_n, \wedge, \vee, \cdot, {\sim}, {\sim}, {\sim},1, 0\rangle$.\\


\begin{figure}[ht]
\centering
\begin{tikzpicture}[scale=0.9]
\begin{scope}
\node[draw,circle,inner sep=1.5pt,fill] (0) at (0,0) {};
\node[draw,circle,inner sep=1.5pt,fill] (1) at (0,1) {};
\draw[order] (0)--(1);
\node[label,anchor=west] at (0) {$a_{-1}=0$};
\node[label,anchor=west] at (1) {$a_1=1$};
\node[label,anchor=south,xshift=-20pt,yshift=-10pt] at (1) {$\mathbf{S}_2$:};
\end{scope}

\begin{scope}[xshift=3.2cm]
\node[draw,circle,inner sep=1.5pt,fill] (a) at (0,-1) {};
\node[draw,circle,inner sep=1.5pt,fill] (b) at (0,0) {};
\node[draw,circle,inner sep=1.5pt,fill] (c) at (0,1) {};
\draw[order] (a)--(b)--(c);
\node[label,anchor=west] at (a) {$a_{-1}$};
\node[label,anchor=west] at (b) {$a_{0}=1=0$};
\node[label,anchor=west] at (c) {$a_{1}$};
\node[label,anchor=south,xshift=-20pt,,yshift=-10pt] at (c) {$\mathbf{S}_3$:};
\end{scope}

\begin{scope}[xshift=6.4cm]
\node[draw,circle,inner sep=1.5pt,fill] (d) at (0,-2) {};
\node[draw,circle,inner sep=1.5pt,fill] (e) at (0,-1) {};
\node[draw,circle,inner sep=1.5pt,fill] (f) at (0,0) {};
\node[draw,circle,inner sep=1.5pt,fill] (g) at (0,1) {};
\draw[order] (d)--(e)--(f)--(g);
\node[label,anchor=west] at (d) {$a_{-2}$};
\node[label,anchor=west] at (e) {$a_{-1}=0$};
\node[label,anchor=west] at (f) {$a_{1}=1$};
\node[label,anchor=west] at (g) {$a_{2}$};
\node[label,anchor=south,xshift=-20pt,yshift=-10pt] at (g) {$\mathbf{S}_4$:};
\end{scope}

\begin{scope}[xshift=9.6cm]
\node[draw,circle,inner sep=1.5pt,fill] (h) at (0,-3) {};
\node[draw,circle,inner sep=1.5pt,fill] (i) at (0,-2) {};
\node[draw,circle,inner sep=1.5pt,fill] (j) at (0,-1) {};
\node[draw,circle,inner sep=1.5pt,fill] (k) at (0,0) {};
\node[draw,circle,inner sep=1.5pt,fill] (l) at (0,1) {};
\draw[order] (h)--(i)--(j)--(k)--(l);
\node[label,anchor=west] at (h) {$a_{-2}$};
\node[label,anchor=west] at (i) {$a_{-1}$};
\node[label,anchor=west] at (j) {$a_{0}=1=0$};
\node[label,anchor=west] at (k) {$a_{1}$};
\node[label,anchor=west] at (l) {$a_{2}$};
\node[label,anchor=south,xshift=-20pt,yshift=-10pt] at (l) {$\mathbf{S}_5$:};
\end{scope}

\end{tikzpicture}
\caption{Finite Sugihara chains $\mathbf{S}_2$ to $\mathbf{S}_5$.}
\label{fig:Sugihara-examples}
\end{figure}

\section{Nested sums of qRAs}\label{sec:KL-construction}


We first recall some definitions which are used in the original nested sum result~\cite[Lemma 6.1]{G04} (see also \cite[Lemma 9.6.2]{GJKO}). We note that our description makes a small correction to the original requirements for residuated lattices.

An element $b$ in an algebra $\mathbf{A}$ is said to be 
\emph{irreducible} with respect to an $n$-ary operation (where $n$ is a positive integer)
$f$ of $\mathbf{A}$
if $f(b_1,\ldots,b_n)=b$ implies $b_i=b$ for some $i\in\{1,\ldots,n\}$, for
all $b_1,\ldots,b_n\in A$. If $b\in A$ is irreducible with respect to all 
of the non-nullary 
operations of 
$\mathbf{A}$, then we call $b$ \emph{totally irreducible}.

Let $\K=\langle K,\wedge_\K,\vee_\K, \cdot_\K,\backslash_\K,/_\K,1_\K\rangle$
and $\Lalg=\langle L,\wedge_\Lalg,\vee_\Lalg,\cdot_\Lalg, \backslash_\Lalg,/_\Lalg,1_\Lalg\rangle$ be
residuated lattices such that $K \cap L=\varnothing$
and 
$1_\K$  is totally irreducible.
Let $K[L] = \left(K\backslash\{1_\K\}\right)\cup L$ and $1_{\K[\Lalg]} = 1_\Lalg$. 
For $m_0,m_1\in K[L]$
and $\star\in\{\cdot,\backslash,/\}$, 
we define $\star_{\K[\Lalg]}$  by
\begin{numcases}{m_0\star_{\K[\Lalg]} m_1 =}
	m_0\star_{\K} m_1 &  if $m_0,m_1\in K\backslash\{1_\K\}$\nonumber\\
	m_0\star_{\Lalg} m_1 &  if $m_0,m_1\in L$\nonumber\\
	m_0\star_{\K} 1_\K &  if $m_0\in K\backslash\{1_\K\}$ and $m_1\in L$\nonumber\\
	1_\K\star_{\K} m_1 &  if $m_0\in L$ and $m_1\in K\backslash\{1_\K\}$;
	\label{OperationDef}
\end{numcases}
and $\wedge_{\K[\Lalg]}$ and $\vee_{\K[\Lalg]}$  are defined by, 
\begin{numcases}{m_0\wedge_{\K[\Lalg]} m_1 =}
	m_0\wedge_\K m_1 &  if $m_0,m_1\in K\backslash\{1_\K\}$\nonumber\\
	m_0\wedge_\Lalg m_1 & if $m_0,m_1\in L$\nonumber\\
	m_i &  if $m_i\in L, m_{1-i}\in K\backslash\{1_\K\}$  and $m_{1-i}\geqslant_\K 1_\K$\nonumber\\
	m_{1-i}\wedge_\K 1_\K&  if $m_{i}\in L, m_{1-i}\in K\backslash\{1_\K\}$ and $m_{1-i}\ngeqslant_\K 1_\K$	\label{MeetDefinition}
\end{numcases}
and
\begin{numcases}{m_0\vee_{\K[\Lalg]} m_1 =}
	m_0\vee_\K m_1 &  if $m_0,m_1\in K\backslash\{1_\K\}$\nonumber\\
	m_0\vee_\Lalg m_1 &  if $m_0,m_1\in L$\nonumber\\
	m_i &  if $m_i\in L, m_{1-i}\in K\backslash\{1_\K\}$ and  $m_{1-i}\leqslant_\K 1_\K$\nonumber\\
	m_{1-i}\vee_\K 1_\K&  if $m_{i}\in L, m_{1-i}\in K\backslash\{1_\K\}$ and $m_{1-i}\nleqslant_\K 1_\K$.\label{JoinDef}
\end{numcases}


Sufficient conditions to be placed upon $\K$ and $\Lalg$ are given in~\cite{G04} that ensure  $\koflbf$ supports the
structure of a residuated lattice with the operations as defined above. There, $\K$ is {\em admissible} by $\Lalg$ if
$1_\K$ is totally irreducible and either $\Lalg$ is integral or $\K$  satisfies
\begin{equation}\label{k-condition}
	k\backslash_\K 1_\K\neq 1_\K \quad\text{ and }\quad 1_\K/_\K k\neq 1_\K\quad \text{ for all }k\in K\backslash\{1_\K\}.
\end{equation}
In
Lemma~\ref{lem:ResIrr-implies-CondK} below 
we prove
that, in fact, the 
condition~(\ref{k-condition})
above is 
actually subsumed by the assumption that 
$1_\K$ be totally irreducible in $\K$. 

In the sequel we omit subscripts in the proofs where no confusion can arise. \\



\begin{lemma}\label{lem:ResIrr-implies-CondK}
Let $\K$ be a residuated lattice with $1_\K$ irreducible with respect to $\backslash_\K$ and $/_\K$. Then $\K$ satisfies {\upshape (\ref{k-condition})}.
\end{lemma}
\begin{proof}
Assume that $1$ is  irreducible with respect to $\backslash$ and $/$, and then suppose that there exists  $k \neq 1$ such that  $1/k=1$ or $k\backslash 1 =1$. 

If $1/k=1$, we get $1 \leqslant 1/k$, and so by residuation $k=1\cdot k \leqslant 1$. Since $/$ is order-preserving in the first argument, we have $k/k\leqslant 1/k=1$. By residuation, we always have $1 \leqslant k/k$ so $k/k=1$. But this is a contradiction since $1$ is irreducible with respect to $/$. 

If $k \backslash 1=1$, we follow a similar argument to reach a contradiction. 
Hence, no such $k$ can exist and so $\K$ satisfies (\ref{k-condition}).  
\end{proof}
Consequently,
$1_\K$ being totally irreducible entails $\K$ is admissible by any residuated lattice $\Lalg$.

We further note that the definitions
of $\wedge_{\koflbf}$ and $\vee_{\koflbf}$ given above differ from the definitions given in~\cite{G04}. This is to ensure closure of 
$K[L]$
under meet and join. 
For example, if $\mathbf{K}=\mathbf{S}_3$, then under the original definition for any $c \in L$ we get  $b \wedge_{\koflbf} c = 1_\K \notin K[L]$. 

The following lemma is an immediate consequence of the construction
described above.

\vskip 0.2cm

\begin{lemma}\label{Lemma:K[L]ordering}
	Let $\K[\Lalg]$ be the nested sum of residuated
	lattices $\K$ and $\Lalg$ 
    such that $1_\K$ is totally irreducible. 
    Let $k_0,k_1\in K\backslash\{1_\K\}$ and $\ell_0,\ell_1\in L$,
	then
	\[k_0\leqslant_{\K[\Lalg]} \ell_0 \iff k_0\leqslant_\K 1_\K,\qquad
	\ell_0\leqslant_{\K[\Lalg]} k_0\iff k_0\geqslant_\K 1_\K,
	\]
	\[k_0\leqslant_{\K[\Lalg]} k_1\iff k_0\leqslant_\K k_1, \qquad
	\ell_0\leqslant_{\K[\Lalg]} \ell_1\iff \ell_0\leqslant_\Lalg \ell_1.\]
\end{lemma}

\vskip 0.2cm 

With Lemma~\ref{Lemma:K[L]ordering} we can now prove that the construction described above does indeed produce a residuated lattice. 

\vskip 0.2cm 

\begin{lemma}[{\cite[Lemma 6.1]{G04}}]
\label{Gal-K-construction}
If $\K$ and $\Lalg$ are  residuated lattices
such that $1_\K$ is totally irreducible,
then the structure ${\K[\Lalg]} = \langle K[L], \wedge_{\K[\Lalg]},\vee_{\K[\Lalg]}, \cdot_{\K[\Lalg]}, \backslash_{\K[\Lalg]},/_{\K[\Lalg]},1_{\K[\Lalg]}\rangle$,
as defined above, is a 
residuated lattice.
\end{lemma}
\begin{proof}
The proof that $K[L]$ forms a lattice and a monoid under the operations defined above is long but straightforward. 
When $m_0$ and $m_1$ are both elements of $K{\setminus}\{1_{\mathbf{K}}\}$, the total irreducibility of $1_{\mathbf{K}}$ is required to ensure that $m_0\star_{\koflbf} m_1$, $m_0 \wedge_{\koflbf} m_1$ and $m_0 \vee_{\koflbf} m_1$ are all  elements of $K{\setminus}\{1_{\mathbf{K}}\}$ and hence elements of $K[L]$. 
Then, when proving the lattice and monoid identities, one must consider all possible combinations of elements from $K\backslash\{1_\K\}$ and $L$. 

The proof of residuation again requires one to consider $m_0$, $m_1$, $m_2$ and all possible cases of membership of $K{\setminus}\{1_\K\}\cup L$. Here we prove just one equivalence for one of the cases. Let $m_1$, $m_2 \in K{\setminus}\{1_\K\}$ and $m_3 \in L$. Then 
\begin{align*}
m_1 \cdot_{\koflbf} m_2 \leqslant_{\koflbf} m_3 \Longleftrightarrow m_1 \cdot_\K m_2 \leqslant_{\koflbf} m_3 &\stackrel{(\ref{Lemma:K[L]ordering})}{\Longleftrightarrow} m_1 \cdot_\K m_2 \leqslant_{\K} 1_{\K} \\ 
&\Longleftrightarrow  m_2 \leqslant_{\K}m_1 \backslash_{\K} 1_{\K}\\
&\Longleftrightarrow m_2 \leqslant_{\koflbf} m_1 \backslash_{\koflbf} m_3. 
\end{align*}
The other cases make similar use of Lemma~\ref{Lemma:K[L]ordering} and residuation in $\K$ and $\Lalg$. 
\end{proof}


The algebra ${\K[\Lalg]}$ is called the \emph{nested sum} of 
$\K$ and $\Lalg$. It is easy to see that if both $\K$ and $\Lalg$ are commutative, then so is 
$\K[\Lalg]$. 
Any idempotent element of $\K$ or $\mathbf{L}$ is also idempotent in $\K[\Lalg]$.

\vskip 0.4cm 
Let 
$\K$ and 
$\Lalg$
be FL-algebras with respective negation constants $0_\K$ and $0_\Lalg$.  	
If $1_\K$ is 
totally irreducible, then 
the nested sum $\K[\Lalg]$ 
of the 0-free reducts of $\K$ and $\Lalg$  form a residuated lattice by Lemma~\ref{Gal-K-construction},
from which an FL-algebra may be obtained by defining a negation constant.
We extend
the definition of the nested sum $\K[\Lalg]$ to FL-algebras $\K$ and $\Lalg$ as the expansion with the negation constant defined by $0_{\K[\Lalg]}:= 0_\Lalg$.
%

Recall that for FL-algebras $\K$, $\Lalg$ and $\koflbf$, the linear negations, ${\sim}_{\K},{-}_{\K}$ and ${\sim}_{\Lalg},{-}_{\Lalg}$ and ${\sim}_{\koflbf},{-}_{\koflbf}$ can be defined on $\K$, $\Lalg$ and $\koflbf$, respectively, in terms of the residuals and the $0$.  In particular, for
$m\in K[L]$,
\[
{\sim}_{\koflbf} m=m\backslash_{\koflbf}0_{\koflbf}\quad\text{and}
\quad {-}_{\koflbf} m=0_{\koflbf}/_{\koflbf}m
\]

\begin{lemma}\label{Lemma:IrreducibilityFL}
Let $\K$
be an FL-algebra such that $1_{\K}$ is irreducible with respect to the residuals.
Then $1_{\K}$ is irreducible with respect to
${\sim}_{\K}$ and ${-}_{\K}$.
\end{lemma}
\begin{proof}
   We consider the case for $-$. Assume $- k = 1$. Then  $0\backslash k=1$, and so, since $1_\K$ is irreducible with respect to $\backslash_\K$, we have $0 = 1$ or $k = 1$. If $k = 1$, we are done, so assume $0= 1$. Then $k = 1\backslash k = 0 \backslash k = 1$, as required. 
\end{proof} 

The following equivalent statements for an
InFL-algebra show that
the total irreducibility of the monoid identity is
independent of our choice of signature for
an InFL-algebra and can be characterised 
as precisely those InFL-algebras that are odd.\\

\begin{lemma}\label{Lemma:IrreducibilityImplied}
Let $\K$
be an InFL-algebra such that $1_{\K}$ is irreducible with respect to $\cdot_{\K}$.  Then the  following statements are equivalent:
    \begin{enumerate}[label={\upshape(\roman*\upshape)}]
        \item  $1_\K$ is irreducible with respect to ${\sim}_\K$ and ${-}_\K$.
        \item $1_\K$ is irreducible with respect to $\backslash_{\K}$
        and $/_{\K}$.
        \item $\K$ is odd, i.e., $1_{\K}=0_{\K}$.
    \end{enumerate}
\end{lemma}
\begin{proof}
(i) $\Rightarrow$ (ii):  Assume that $\K$ is an InFL-algebra such that $1$ is irreducible with respect to $\cdot$, ${\sim}$ and ${-}$ and let $k_1,k_2\in K$
such that $k_1\backslash k_2=1$. Then by~(\ref{Eqn:ResidualsFromNegations}) we have that
    $k_1\backslash k_2={\sim}\left({-}k_2\cdot k_1\right)=1$ which implies that ${-}k_2\cdot k_1=1$, hence ${-} k_2=1$ or $k_1=1$ 
    and therefore $k_2=1$ or $k_1=1$ since $1$ is irreducible with respect to $\cdot$, ${\sim}$ and ${-}$, respectively. That is, $1$ is irreducible with respect to $\backslash$.
    The proof that $1$ is irreducible with respect 
    to $/$ follows similarly.
    
(ii) $\Rightarrow$ (i): This implication holds even in the case of FL-algebras as shown in Lemma~\ref{Lemma:IrreducibilityFL}.

    (i) $\Rightarrow$ (iii):  Assume that $\K$ is an InFL-algebra such that $1$ is irreducible with respect to $\cdot$, ${\sim}$ and ${-}$.  In any FL-algebra we have ${\sim} 1 =0$, and so by \textsf{(In)},
$1={-}{\sim} 1=-0$ and finally since $1$ is irreducible with respect to ${-}$ we have that $0=1$.

(iii) $\Rightarrow$ (i):  Assume that $\K$ is an odd InFL-algebra such that $1$ is irreducible with respect to $\cdot$.  Let $k\in K$ such that
    ${\sim}k=1$.  Then
    $k={-}{\sim} k={-}1=0=1$.
    If ${-}k=1$, then 
    $k={\sim}{-} k={\sim}1=0=1$.
\end{proof}

Now consider InFL-algebras  $\K$ and $\Lalg$ such that $1_{\K}$ is totally irreducible (and hence, $\K$ is odd). 
For $K[L]$ with
$\wedge_{\K[\Lalg]}$, $\vee_{\K[\Lalg]}$, $\cdot_{\K[\Lalg]}$, ${\sim}_{\koflbf}$,
${-}_{\koflbf}$, $1_{\K[\Lalg]}$ and $0_{\koflbf}$
defined as before, we call  
$\koflbf=\langle K[L],\wedge_{\koflbf},\vee_{\koflbf},\cdot_{\koflbf},{\sim}_{\koflbf},{-}_{\koflbf},1_{\koflbf},0_{\koflbf}\rangle$
the nested sum of $\K$ and $\Lalg$.\\

\begin{lemma}\label{Lemma:NegationBehaviour}
Let  $\K[\Lalg]$ be the nested sum of InFL-algebras $\K$ and $\Lalg$, as defined above
with $1_\K$ totally irreducible {\upshape (}implying that $\K$ is odd{\upshape)}.
Then, for $m\in K[L]$, we have 
\[
{\sim}_{\K[\Lalg]}m=
\begin{cases}
{\sim}_{\Lalg} m & \text{ if }m\in L\\
{\sim}_{\K} m & \text{ if }m\in K\backslash\{1_\K\},
\end{cases}
\quad \text{and} 
\quad 
{-}_{\K[\Lalg]}m=
\begin{cases}
{-}_{\Lalg} m & \text{ if }m\in L\\
{-}_{\K} m & \text{ if }m\in K\backslash\{1_\K\}.
\end{cases}
\]	
\end{lemma}
\begin{proof}
Let $m\in K[L]$. If $m \in L$, then ${\sim}_{\K[\Lalg]}m=m\backslash_{\K[\Lalg]}0_{\K[\Lalg]}= 
m\backslash_\Lalg 0_\Lalg = {\sim}_\Lalg m$ and
${-}_{\K[\Lalg]}m=0_{\K[\Lalg]}/_{\K[\Lalg]} m =0_\Lalg/_\Lalg m = {-}_\Lalg m$.

Recall that since $1_\K$ is totally 
irreducible, by Lemma~\ref{Lemma:IrreducibilityImplied}, $\K$ is odd. 
If $m\in K\backslash\{1_\K\}$, then
${\sim}_{\K[\Lalg]}m=m\backslash_{\K[\Lalg]}0_{\K[\Lalg]}=m\backslash_{\K[\Lalg]}0_{\Lalg}
=m\backslash_\K1_\K=m\backslash_\K0_\K={\sim}_\K m$ 
and
${-}_{\K[\Lalg]}m=0_{\K[\Lalg]}/_{\K[\Lalg]}m=0_{\Lalg}/_{\K[\Lalg]}m
=1_\K/_\K m=0_\K/_\K m={-}_\K m$.
\end{proof}

The following proposition follows immediately from Lemmas~\ref{Gal-K-construction},
~\ref{Lemma:IrreducibilityImplied} and \ref{Lemma:NegationBehaviour} 
and the fact that $\K$ and $\Lalg$ satisfy \textsf{(In)}. 

\vskip 0.2cm 

\begin{proposition}
	If $\K$ and $\Lalg$ are {\upshape(}cyclic{\upshape)} InFL-algebras such that
	$1_\K$ is totally irreducible {\upshape (}implying that $\K$ is odd{\upshape)},
    then their nested sum $\K[\Lalg]$
	is a {\upshape (}cyclic{\upshape)} InFL-algebra.\\
\end{proposition}
\begin{proof}
Let $m \in K[L]$. If $m \in K{\setminus}\{1_\K\}$ then ${\sim}_{\koflbf} -_{\koflbf} m={\sim}_\K -_\K m =m$ since $\K$ is an InFL-algebra. Similarly for $m\in L$.
\end{proof}

Next let $\K$ and $\Lalg$ be InFL$'$-algebras with $1_\K$ totally irreducible,
then their nested sum is the algebra 
$\K[\Lalg]=\langle K[L], \wedge_{\K[\Lalg]}, \vee_{\K[\Lalg]}, \cdot_{\K[\Lalg]}, {\sim}_{\K[\Lalg]}, {-}_{\K[\Lalg]},{\neg}_{\K[\Lalg]},1_{\K[\Lalg]},0_{\K[\Lalg]}\rangle$
where the InFL-reduct of $\K[\Lalg]$
is as defined above, and for $m\in K[L]$,
\[
\neg_{\K[\Lalg]}m=
\begin{cases}
	{\neg}_\Lalg m & \text{ if }m\in L\\
	{\neg}_\K m & \text{ if }m\in K\backslash\{1_\K\}.
\end{cases}
\]
Given this definition of ${\neg}_{\K[\Lalg]}$ it is
immediate that $\K[\Lalg]$ is an InFL$'$-algebra whenever
both
$\K$ and $\Lalg$ are.

Finally, assume that $\K$ are $\Lalg$ are qRAs.  Since
\textsf{(Di)} implies \textsf{(In)}, it follows as above that
if $1_\K$ is totally irreducible, then $\K$ is odd.\\

\begin{proposition}\label{Proposition:GOSqRA}
	If $\K$ and $\Lalg$ are qRAs such that
	$1_\K$ is totally irreducible {\upshape(}and hence $\K$ is odd{\upshape)},
    then their 
	nested sum $\K[\Lalg]$
	is a qRA.
\end{proposition}
\begin{proof}
	To prove that $\K[\Lalg]$ is a qRA, we need to show that it satisfies
	\textsf{(Dm)}, \textsf{(Di)} and \textsf{(Dp)}.  Let $m_0,m_1\in K[L]$. We first show that $\K[\Lalg]$ satisfies
	\textsf{(Dm)}.
	
	\vskip 0.2cm 
	
If $m_0,m_1\in K\backslash\{1_\K\}$ or $m_0,m_1\in L$, then \textsf{(Dm)} is inherited from $\K$ or $\Lalg$, respectively.  
	For the remaining cases, we can assume without loss of generality
	that $m_0\in K\backslash\{1_\K\}$ and $m_1\in L$ since
	$\wedge_{\K[\Lalg]}$ and $\vee_{\K[\Lalg]}$ are commutative.  Now consider two cases.
If $m_0\geqslant_\K 1_\K$, then $\neg_\K m_0\leqslant_\K \neg_\K 1_\K=0_\K=1_\K$. 
Using 
	this and Lemma~\ref{Lemma:K[L]ordering} we have that
	\[\neg_{\K[\Lalg]}(m_0\wedge_{\K[\Lalg]} m_1)=\neg_{\K[\Lalg]} m_1=\neg_\Lalg m_1
	=\neg_\K m_0\vee_{\K[\Lalg]}\neg_\Lalg m_1=
	\neg_{\K[\Lalg]}m_0\vee_{\K[\Lalg]}\neg_{\K[\Lalg]} m_1.\]
	
On the other hand, assume that $m_0\ngeqslant_\K 1_\K$. Since $\K$ is odd and $\K$  satisfies \textsf{(Dm)}, we get
	\begin{multline*}
	\neg_{\K[\Lalg]}(m_0\wedge_{\K[\Lalg]} m_1)=\neg_{\K[\Lalg]}(m_0\wedge_\K 1_\K)
	=\neg_\K(m_0\wedge_\K 1_\K)=\neg_\K m_0\vee_\K\neg_\K1_\K =\\
	\neg_\K m_0\vee_\K 0_\K=\neg_\K m_0\vee_\K 1_\K= \neg_\K m_0\vee_{\K[\Lalg]}\neg_\Lalg m_1
	=\neg_{\K[\Lalg]} m_0\vee_{\K[\Lalg]}\neg_{\K[\Lalg]} m_1.
	\end{multline*}
		
\textsf{(Di)} follows directly from the definition
	of 
    $\neg_{\K[\Lalg]}$, Lemma~\ref{Lemma:NegationBehaviour} and the fact that $\K$ and $\Lalg$ satisfy \textsf{(Di)}.
	
	\vskip 0.2cm
	
Finally, we show that $\K[\Lalg]$ satisfies \textsf{(Dp)}.  As with \textsf{(Dm)}, 
	if $m_0,m_1\in K\backslash\{1_\K\}$ or $m_0,m_1\in L$,
	then \textsf{(Dp)} is inherited from $\K$ or $\Lalg$, respectively.
Next assume that $m_0\in K\backslash\{1_\K\}$ and 
	$m_1\in L$.  Then,
	\[\neg_{\K[\Lalg]}(m_0\cdot_{\K[\Lalg]} m_1)=
	\neg_{\K[\Lalg]}(m_0\cdot_\K 1_\K)=\neg_{\K[\Lalg]} m_0= \neg_\K m_0.\]	
	Since \textsf{(Di)} implies \textsf{(In)} we also have
\begin{align*}
{\sim}_{\K[\Lalg]}({-}_{\K[\Lalg]}(\neg_{\K[\Lalg]} m_0) \cdot_{\K[\Lalg]}
	{-}_{\K[\Lalg]} (\neg_{\K[\Lalg]} m_1))
 &= {\sim}_{\K[\Lalg]}({-}_{\K}(\neg_{\K} m_0) \cdot_{\K[\Lalg]}
	{-}_{\Lalg} (\neg_{\Lalg} m_1))\\
 &={\sim}_{\K[\Lalg]}({-}_{\K}(\neg_{\K} m_0) \cdot_{\K}
	1_\K)\\
 & ={\sim}_{\K[\Lalg]}{-}_{\K}(\neg_{\K} m_0)\\
 & ={\sim}_{\K}{-}_{\K}(\neg_{\K} m_0)\\
	&=\neg_\K m_0.
 \end{align*}
	The case where $m_0\in L$ and $m_1\in K\backslash\{1_\K\}$ follows similarly.	
\end{proof}

Now we give an example of the nested sum $\koflbf$ for qRAs $\K$ and $\Lalg$. \\

\begin{example}\label{Example:KofLConstruction}
Let $\K_1$ be the three-element Sugihara chain $\mathbf{S}_3$
{\upshape(}depicted in
Figure~\ref{Figure:GOS1}{\upshape)}. 
Then 
$1_{\K_1}$ is totally irreducible.		
Let $\Lalg_1$ be the four-element algebra also depicted in
Figure~\ref{Figure:GOS1}
{\upshape(}also see Figure~\ref{fig:qRAs-examples} for
a full description{\upshape)}.
Since $\K_1$ and $\Lalg_1$ are qRAs such that $1_{\K_1}$ is totally irreducible, 
it follows from Proposition~\ref{Proposition:GOSqRA} that  their
nested sum $\K_1[\Lalg_1]$ {\upshape(}the final algebra depicted in
Figure~\ref{Figure:GOS1}{\upshape)}
is a qRA. \\
\end{example}

\begin{figure}[ht]
\centering
\begin{tikzpicture}[scale=0.8]
\begin{scope}
\node[draw,circle,inner sep=1.5pt,fill] (bot) at (0,0) {};
\node[draw,circle,inner sep=1.5pt,fill] (id) at (0,1) {};
\node[draw,circle,inner sep=1.5pt,fill] (top) at (0,2) {};
\draw[order] (bot)--(id)--(top);
\node[label,anchor=east,xshift=1pt] at (top) {$\top$};
\node[label,anchor=east,xshift=1pt] at (id) {$1_{\mathbf{K}_1}$};
\node[label,anchor=east,xshift=1pt] at (bot) {$\bot$};
\node[label,anchor=north,yshift=-20pt] at (bot) {$\K_1$};
\end{scope}

\begin{scope}[xshift=2.75cm]
\node[draw,circle,inner sep=1.5pt,fill] (bot) at (1,0) {};
\node[draw,circle,inner sep=1.5pt,fill] (top) at (1,2) {};
\node[draw,circle,inner sep=1.5pt,fill] (b2) at (0,1) {};
\node[draw,circle,inner sep=1.5pt,fill] (b3) at (2,1) {};
\draw[order] (bot)--(b2);
\draw[order] (bot)--(b3);
\draw[order] (b2)--(top);
\draw[order] (b3)--(top);
\node[label,anchor=south] at (top) {$1_{\mathbf{L}_1}$};
\node[label,anchor=north,xshift=1pt] at (bot) {$0$};
\node[label,anchor=east,xshift=1pt] at (b2) {$a$};
\node[label,anchor=west,xshift=-1pt] at (b3) {$b$};
\node[label,anchor=north,yshift=-20pt] at (bot) {$\Lalg_1$};
\end{scope}

\begin{scope}[xshift=7.5cm]
\node[draw,circle,inner sep=1.5pt,fill] (a1) at (1,0) {};
\node[draw,circle,inner sep=1.5pt,fill] (a2) at (1,4) {};
\node[draw,circle,inner sep=1.5pt,fill] (b0) at (1,1) {};
\node[draw,circle,inner sep=1.5pt,fill] (b1) at (1,3) {};
\node[draw,circle,inner sep=1.5pt,fill] (b2) at (0,2) {};
\node[draw,circle,inner sep=1.5pt,fill] (b3) at (2,2) {};
\draw[order] (b0)--(b2);
\draw[order] (b0)--(b3);
\draw[order] (b2)--(b1);
\draw[order] (b3)--(b1);
\draw[order] (a1)--(b0);
\draw[order] (b1)--(a2);
\node[label,anchor=east,xshift=1pt,yshift=2pt] at (b1) {$1_{\mathbf{K}_1[\mathbf{L}_1]}$};
\node[label,anchor=east,yshift=-1pt] at (b0) {$0$};
\node[label,anchor=east,xshift=1pt] at (b2) {$a$};
\node[label,anchor=west,xshift=-1pt] at (b3) {$b$};
\node[label,anchor=east,xshift=1pt] at (a1) {$\bot$};
\node[label,anchor=east,xshift=1pt] at (a2) {$\top$};
\node[label,anchor=north,yshift=-20pt] at (a1) {$\K_1[\Lalg_1]$};
\end{scope}
\end{tikzpicture}	
\caption{qRAs $\K_1$, $\Lalg_1$ and their nested sum, the
qRA $\K_1[\Lalg_1]$.}
\label{Figure:GOS1}
\end{figure}

For our representation results in Section~\ref{sec:KL-rep} we require that $\mathbf{K[L]}$ be distributive. 
Hence we are interested in when distributivity will be preserved. Recall that a DqRA $\K$ is called {\em conic} if every $a\in K$ satisfies $a\leqslant_\K 1$ or $1\leqslant_\K a$.\\

\begin{theorem}\label{thm:ConicKDqRA}
	Let $\K$ and $\Lalg$ be DqRAs such that $\K$ is conic and $1_\K$ is 
	totally irreducible,
	then
    their nested sum $\K[\Lalg]$ is a DqRA.
\end{theorem}
\begin{proof}
It follows from Proposition~\ref{Proposition:GOSqRA} that
$\K[\Lalg]$ is a qRA so we need only confirm that the underlying lattice is
distributive.
Since $1_\mathbf{K}$ is both join- and meet-irreducible, the set $A=\{\,m \in K \mid m >1_\mathbf{K}\,\}$ forms a distributive sublattice of $\mathbf{K[L]}$, as does $B=\{\, m \in K \mid m < 1_\mathbf{K}\,\}$. For $a \in A$, $b \in B$ and $\ell \in L$ we always have $b < \ell < a$.
Now suppose that $P \subseteq K[L]$ forms a sublattice isomorphic to $\mathbf{M}_3$ or $\mathbf{N}_5$. 

If the lattice formed by $P$ is isomorphic to $\mathbf{M}_3$, and $p$ and $q$ are two incomparable join-irreducible elements of $P$ such that $p\in A$ and $q \in B$, then $q <p$ by the above. This contradicts $p$ and $q$ being incomparable. Similarly we get a contradiction if $p$ and $q$ are members of two different sets among $A$, $B$ and $L$. Hence the three incomparable join-irreducibles 
of $P$ 
must all be in one of $A$ or $B$ or $L$. But by the definition of $\vee_{\koflbf}$ and $\wedge_{\koflbf}$, this would lead to $P \subseteq A$, $P \subseteq B$ or $P \subseteq L$, contradicting the distributivity of $\K$ and $\Lalg$.  

If the lattice formed by $P$ is isomorphic to $\mathbf{N}_5$, consider the three join-irreducibles $n_1$ unrelated to $n_2$, and $n_2<n_3$. As above, $n_1$ and $n_2$ must both be in one of $A$ or $B$ or $L$. Then we get $n_1 \vee_{\koflbf} n_2$  in the same set, and since $n_2 < n_3 < n_1\vee_{\koflbf} n_2$, we have $n_3$ in the same set as well. As $A,B$ and $L$ are closed under $\wedge_{\koflbf}$ we get $P$ a subset of $A$, $B$ or $L$. Again this contradicts the distributivity of $\K$ and $\Lalg$. 
\end{proof}

The following proposition is a partial converse of Theorem~\ref{thm:ConicKDqRA} above. \\

\begin{proposition}\label{prop:NotDqRA}
Let $\K$ and $\Lalg$ be qRAs such that $|L|>1$ and $1_\K$ is totally irreducible.
If $\K[\Lalg]$ is distributive, then $\mathbf K$ is conic. 
\end{proposition}	

\begin{proof}
If $\mathbf K$ is not conic then 
there exists some $k\in K$ such that $k\nleqslant_\K 1_\K$ and
	$1_\K\nleqslant_\K k$.  Set $m_0=k\wedge_\K 1_\K$ and $m_1=k\vee_\K 1_\K$.
	Hence $m_0,m_1\neq 1_\K$ and $m_0,m_1\neq k$, and $\{k,m_0,m_1,1_\K\}$
	forms a sublattice of $\langle K,\wedge_\K,\vee_\K\rangle$.  Since $L$ is a lattice and $|L|>1$, there
	exist $\ell_0,\ell_1\in L$ such that $\ell_0 <_\Lalg \ell_1$.  
	Clearly $\ell_0\wedge_{\K[\Lalg]} \ell_1=\ell_0$ and  $\ell_0\vee_{\K[\Lalg]} \ell_1=\ell_1$.
	Also, $m_1\wedge_{\K[\Lalg]}\ell_i=\ell_i$ and $m_0\vee_{\K[\Lalg]} \ell_i=\ell_i$ for 
	$i\in\{0,1\}$. Moreover,
	for $d_0\in\{k,m_0\}$ and $ d_1\in\{k,m_1\}$,
	\[d_0\wedge_{\K[\Lalg]} \ell_0=d_0\wedge_\K 1_\K=m_0\qquad\text{ and }\qquad
	d_0\wedge_{\K[\Lalg]} \ell_1=d_0\wedge_\K 1_\K=m_0\]
	\[d_1\vee_{\K[\Lalg]} \ell_0=d_1\vee_\K 1_\K=m_1\qquad\text{ and }\qquad
	d_1\vee_{\K[\Lalg]} \ell_1=d_1\vee_\K 1_\K=m_1.\]
 \vskip 0.4cm 
\noindent	Thus, $\{k,m_0,m_1,\ell_0,\ell_1\}$ forms  a sublattice of the underlying lattice of $\K[\Lalg]$ 
isomorphic to $\mathbf{N}_5$.
\end{proof}

\begin{example}
Let $\K_2$ be the four-element algebra depicted in
Figure~\ref{Figure:GOS2} 
{\upshape(}also see the second diamond in
Figure~\ref{fig:qRAs-examples}{\upshape)}. 
Then $1_{\mathbf{K}_2}$ is totally irreducible.
Let $\Lalg_2$ be the two-element Sugihara chain $\mathbf{S}_2$ 
{\upshape(}also depicted in Figure~\ref{Figure:GOS2}{\upshape)}.
Since $\K_2$ and $\Lalg_2$ are both qRAs, it follows from Proposition~\ref{Proposition:GOSqRA}
that $\K_2[\Lalg_2]$ is again a qRA. However, 
since $a\nleqslant 1_{\K_2}$ and $1_{\K_2}\nleqslant a$ it follows from Proposition~\ref{prop:NotDqRA} that $\K_2[\Lalg_2]$ is not
a DqRA.  Indeed, $\K_2[\Lalg_2]$
{\upshape(}also depicted in Figure~\ref{Figure:GOS2}{\upshape)}
has its lattice reduct isomorphic to
$\mathbf{N}_5$. \\
\end{example}

\begin{figure}[ht]
\begin{tikzpicture}[scale=0.8]
\begin{scope}
\node[draw,circle,inner sep=1.5pt,fill] (bot) at (1,0) {};
\node[draw,circle,inner sep=1.5pt,fill] (top) at (1,2) {};
\node[draw,circle,inner sep=1.5pt] (c2) at (0,1) {};
\node[draw,circle,inner sep=1.5pt,fill] (c3) at (2,1) {};
\draw[order] (bot)--(c2);
\draw[order] (bot)--(c3);
\draw[order] (c2)--(top);
\draw[order] (c3)--(top);
\node[label,anchor=south] at (top) {$\top$};
\node[label,anchor=north] at (bot) {$\bot=a^2$}; 
\node[label,anchor=east,xshift=1pt] at (c2) {$a= a\cdot\top$}; 
\node[label,anchor=west,xshift=-1pt] at (c3) {$1_{\mathbf{K}_2}$}; 
\node[label,anchor=north,yshift=-20pt] at (bot) {$\K_2$};
\end{scope}
\begin{scope}[xshift=6.2cm]
\node[draw,circle,inner sep=1.5pt,fill] (bot) at (0,0.5) {};
\node[draw,circle,inner sep=1.5pt,fill] (top) at (0,1.5) {};
\draw[order] (bot)--(top);
\node[label,anchor=east,xshift=1pt] at (top) {$1_{\mathbf{L}_2}$};				
\node[label,anchor=east,xshift=1pt] at (bot) {$0$};
\node[label,anchor=north,yshift=-20pt] at (bot) {$\Lalg_2$};
\end{scope}	

\begin{scope}[xshift=9.2cm]
\node[draw,circle,inner sep=1.5pt,fill] (c0) at (1,0) {};
\node[draw,circle,inner sep=1.5pt,fill] (c1) at (1,3) {};
\node[draw,circle,inner sep=1.5pt,fill] (a0) at (2,1) {};
\node[draw,circle,inner sep=1.5pt,fill] (a1) at (2,2) {};
\node[draw,circle,inner sep=1.5pt] (c2) at (0,1.5) {};
\draw[order] (c0)--(c2);
\draw[order] (c0)--(a0);
\draw[order] (c2)--(c1);
\draw[order] (a0)--(a1);
\draw[order] (a1)--(c1);
\node[label,anchor=east,xshift=1pt] at (c0) {$\bot$};
\node[label,anchor=east,xshift=1pt] at (c1) {$\top$};
\node[label,anchor=east,xshift=1pt] at (c2) {$a$};
\node[label,anchor=west,xshift=-1pt] at (a0) {$0$};
\node[label,anchor=west,xshift=-1pt] at (a1) {$1_{\mathbf{K}_2[\mathbf{L}_2]}$};
\node[label,anchor=north,yshift=-20pt] at (c0) {$\K_2[\Lalg_2]$};
\end{scope}
\end{tikzpicture}
\caption{DqRAs $\K_2$, $\Lalg_2$ and their non-distributive nested sum $\K_2[\Lalg_2]$.}
\label{Figure:GOS2}
\end{figure}


We note that if $\mathbf{S}_n$ is a finite odd Sugihara chain,
then $\mathbf{S}_n$ can be viewed as an odd DqRA such that its 
monoid identity $1(=a_0)$ is totally irreducible.
Moreover, since $\mathbf{S}_n$
is conic, it follows from Theorem~\ref{thm:ConicKDqRA} that
if $\mathbf{L}$ is a DqRA, then the nested sum
$\mathbf{S}_n[\mathbf{L}]$ of $\mathbf{S}_n$ and $\mathbf{L}$
is a DqRA.\\

\begin{proposition}\label{prop:K=Sn+L=Sm}
Let $\mathbf{K}=\mathbf{S}_n$ for $n$ odd, and $\mathbf{L}=\mathbf{S}_m$ for $m \geqslant 2$, then $\mathbf{K}[\mathbf{L}]\cong \mathbf{S}_{n+m-1}$. \\
\end{proposition}

\begin{proof}
Let $\K = \mathbf{S}_{n}$ be a Sugihara chain with $n = 2k+1$ for $k\geqslant 1$, and let $\Lalg = \mathbf{S}_{m}$ with $m = 2\ell$ or $m = 2\ell + 1$ for $\ell \geqslant 1$. Elements of $S_n$ will be denoted by $a_{-k}^\K, \ldots, a_{k}^\K$, elements of $S_m$ by $a^\Lalg_{-\ell}, \ldots, a^\Lalg_{\ell}$, and elements of $S_{n+m-1}$ by $a_{-k-\ell}, \ldots, a_{k+\ell}$. Define a map $\psi: K[L] \to S_{n+m-1}$ by setting, for each $b\in K[L]$,
\[
\psi(b)=
\begin{cases}
a_{j-\ell} & \text{ if } b = a^\K_{j} \text{ and } -k \leqslant j \leqslant -1\\
a_{j +\ell} & \text{ if } b = a^\K_{j} \text{ and } 1 \leqslant j \leqslant k\\
a_i & \text{ if } b = a^\Lalg_i.
\end{cases}
\]
It is straightforward to show that $\psi$ is an isomorphism from $\K[\Lalg]$ to $\mathbf{S}_{n+m-1}$.
\end{proof}

\vskip 0.4cm 

There is evidence to suggest that a 
conic DqRA with its monoid identity $1$ totally irreducible will always be a chain, and further will always be a Sugihara chain. We have not however been able to prove this. When searching for conic DqRAs with $1$ totally irreducible, Prover9/Mace4 returns only finite Sugihara chains for models of 
size
less than 20. \\

\begin{problem}\label{prob:K=Snalways?}
Let $\mathbf{K}$ be a DqRA 
such that $1_\mathbf{K}$ is totally irreducible.
If $\mathbf{K}$ is conic, is it always linearly ordered? 
If $\mathbf{K}$ is finite and conic, is it linearly ordered? 
When $\mathbf{K}$ is finite, is it always isomorphic to $\mathbf{S}_n$ for some odd $n$? 
\end{problem}

\section{Representable distributive quasi relation algebras}\label{sec:RDqRA}

In this section we recall the construction of distributive quasi relation algebras from partially ordered equivalence relations. This recent work by two of the current authors~\cite{CR-RDqRA} included a definition of \emph{representability} of a DqRA. The basic idea is to replicate the way in which a relation algebra can be concretely constructed as an algebra of binary relations, starting from an equivalence relation.  In the setting of DqRAs, the algebras are constructed as up-sets of a partially ordered equivalence relation (with the partial order satisfying some additional symmetry conditions). 

As observed in~\cite{CR-RDqRA}, the lattice structure in the construction is the same as that used by Galatos and Jipsen~\cite{GJ20-AU, GJ20-ramics}, and  by Jipsen and \v{S}emrl~\cite{JS23}. An essential difference is 
the additional requirement of the existence of the maps $\alpha$ and $\beta$ (see below). 

Before describing the construction from~\cite{CR-RDqRA}, we recall some basic facts about binary relations. 
For 
$R \subseteq X^2$ 
its converse is 
$R^\smile = \left\{\left(x, y\right) \mid \left(y, x\right) \in R\right\}$. For two binary relations $R$ and $S$ the composition, $R \mathbin{;} S$ is given by 
$R \mathbin{;} S = \left\{\left(x, y\right) \mid \left(\exists z \in X\right)\left(\left(x, z\right) \in R \textnormal{ and } \left(z, y\right) \in S\right)\right\}$.
When considering a binary relation, $R$, it will usually be contained in an equivalence relation $E$ and therefore
the complement of $R$, denoted  $R^c$, will mean $R^c=\{\, (x,y) \in E \mid (x,y) \notin R\,\}$. If an equivalence relation is not specified, it should be taken as $X^2$. 

We define $\mathrm{id}_X=\{\,(x,x)\mid x \in X\,\}$.
The following familiar equivalences will be used frequently throughout the paper. 
If $E$ is an equivalence relation on $X$ and $R,S,T \subseteq E$ then we have $(R^{\smile})^{\smile}=R$, $(R^{\smile})^c=(R^c)^{\smile}$, and $\mathrm{id}_X\mathbin{;} R = R\mathbin{;} \mathrm{id}_X = R$. Further, $\left(R\, ; S\right)\mathbin{;} T = R\mathbin{;} \left(S\mathbin{;} T\right)$ and 
$\left(R\mathbin{;} S\right)^\smile = S^\smile\mathbin{;} R^\smile$.




For a function $\gamma: X \to X$ we will use  $\gamma$ to denote either the function or the binary relation that is its graph. 
Note that the lemma below applies when $\gamma$ is a bijective function $\gamma : X \to X$. Also recall that $R^c$ is the complement of $R$ in $E$.\\


\begin{lemma}
[{\cite[Lemma 3.4]{CR-RDqRA}}]
\label{lem:important_eq_injective_map}
Let $E$ be an equivalence relation on a set $X$, and let $R, S, \gamma \subseteq E$. If 
$\gamma$
satisfies 
$\gamma^{\smile}\mathbin{;} \gamma = \mathrm{id}_X$ and $\gamma \mathbin{;} \gamma^{\smile}=\mathrm{id}_X$
then the following hold:
\begin{enumerate}[label=(\roman*)]
\item $\left(\gamma\mathbin{;} R\right)^c = \gamma \mathbin{;} R^c$
\item $\left(R\mathbin{;} \gamma\right)^c = R^c \mathbin{;} \gamma$.\\
\end{enumerate}	
\end{lemma}

We now recall from~\cite{CR-RDqRA} how to construct a DqRA of binary relations.  
Consider a poset $\mathbf X = \left(X, \leqslant\right)$ and   $E$ an equivalence relation on $X$ with ${\leqslant} \subseteq E$. The equivalence relation $E$ can be partially ordered for all $(u, v), (x, y) \in E$ as follows: $(u, v) \preccurlyeq (x, y)  \textnormal{ iff } x \leqslant u \textnormal{ and } v\leqslant y$. The structure $\mathbf E = \left(E, \preccurlyeq\right)$ is then a poset,  hence the set of up-sets of $\mathbf E$, denoted $\textsf{Up}\left(\mathbf E\right)$, ordered by inclusion, is a distributive lattice. 



If $R,S \in \UpE$, then $R\mathbin{;}S \in \UpE$. The set of downsets of $(E,\preccurlyeq)$ is denoted $\mathsf{Down}(\mathbf{E})$ and is also closed under composition. Importantly, $R \in \UpE$ iff $R^c \in \mathsf{Down}(\mathbf{E})$ iff $R^{\smile} \in \mathsf{Down}(\mathbf{E})$.
We have 
${\leqslant}\in \mathsf{Up}(\mathbf{E})$ and it is the identity with respect to~$\mathbin{;}$ 
on $\UpE$.
Further, the operation of relational composition $\mathbin{;}$ is residuated with  
$R \backslash_{\UpE} S = (R^{\smile}\mathbin{;}S^c)^c$ and 
$R/_{\UpE}S=(R^c\mathbin{;}S^{\smile})^c$.
Hence $\langle \mathsf{Up}(\mathbf{E}), \cap, \cup, \mathbin{;},\backslash_{\UpE},/_{\UpE}, \leqslant \rangle
$ is a distributive residuated lattice.  

The remaining DqRA operations in the construction (see  Theorem~\ref{thm:Dq(E)} below)  require an order automorphism $\alpha : X \to X$ and a self-inverse dual order automorphism $\beta : X \to X$. We note briefly some facts regarding set-theoretic operations on $\mathsf{Up}(\mathbf{E})$. 
Details can be found in~\cite[Lemma 3.5]{CR-RDqRA}.
If $R, S  \in \mathsf{Up}\left(\mathbf E\right)$, then 
$\alpha\mathbin{;} R$, and  $R\mathbin{;} \alpha$ are elements of $\mathsf{Up}\left(\mathbf E\right)$.
If $R  \in \mathsf{Down}\left(\mathbf E\right)$, then $
\beta\mathbin{;}R\mathbin{;}\beta \in \UpE$.

For
$\alpha : X \to X$ an order automorphism,
we let $0= \alpha \mathbin{;} {\leqslant^{c\smile}}$. With this we can define ${\sim}$ and $-$ 
on $\mathsf{Up}(\mathbf{E})$. 
By \cite[Lemma 3.10]{CR-RDqRA} we can in fact define the linear negations \emph{without} using the residuals. We get 
${\sim} R = R^{c\smile}\mathbin{;} \alpha$ and $-R = \alpha \mathbin{;} R^{c\smile}$ for all $R \in \mathsf{Up}\left(\mathbf{E}\right)$. For calculations it is much easier to use these definitions of ${\sim}R$ and $-R$ than those involving the residuals. 

Lastly, $\neg R$ is defined using the dual order automorphism $\beta$, as stated below.\\ 

\begin{theorem}[{\cite[Theorem 3.15]{CR-RDqRA}}]\label{thm:Dq(E)}
Let $\mathbf{X}=\left(X,\leqslant\right)$ be a poset and $E$ an equivalence relation on $X$ such that ${\leqslant} \subseteq E$.  Let $\alpha: X \to X$ be an order automorphism of $\mathbf X$ and $\beta: X \to X$ a self-inverse dual order automorphism of $\mathbf X$ such that $\alpha, \beta \subseteq E$ and $\beta = \alpha \mathbin{;} \beta\mathbin{;} \alpha$. 
Set $1={\leqslant}$ and $0= \alpha \mathbin{;} {\leqslant^{c\smile}}$. 
For $R \in \mathsf{Up}(\mathbf E)$, define
${\sim} R = R^{c\smile}\mathbin{;} \alpha$, $-R = \alpha \mathbin{;} R^{c\smile}$, and 
$\neg R =  \alpha\mathbin{;} \beta \mathbin{;} R^c \mathbin{;} \beta$.
Then the algebra $\mathbf{Dq}(\mathbf E) = \left\langle \mathsf{Up}\left(\mathbf{E}\right),\cap, \cup, \mathbin{;}, {\sim}, {-}, {\neg}, 1, 0 \right\rangle$ 
is a distributive quasi relation algebra. If $\alpha$ is the identity, then $\mathbf{Dq}(\mathbf E)$ is a cyclic distributive quasi relation algebra.	\\
\end{theorem}	

An algebra of the form described by Theorem~\ref{thm:Dq(E)} is called an \emph{equivalence distributive quasi relation algebra} and the class of such algebras is  denoted $\mathsf{EDqRA}$. If $E=X^2$, then we refer to the algebra $\mathbf{Dq}(\mathbf{E})$ as a \emph{full distributive quasi relation algebra}, with the class of such algebras denoted by $\mathsf{FDqRA}$. Analogous to the case for relation algebras (cf.~\cite[Chapter 3]{Mad06}), it was shown~\cite[Theorem 4.4]{CR-RDqRA} that 
$\mathbb{IP}(\mathsf{FDqRA})=\mathbb{I}(\mathsf{EDqRA})$. This gives rise to the definition below: \\

\begin{definition}[{\cite[Definition 4.5]{CR-RDqRA}}]\label{def:RDqRA}
A DqRA $\mathbf{A} = \left\langle A, \wedge, \vee, \cdot,  {\sim}, {-}, {\neg}, 1,0 \right\rangle$ 
is \emph{representable} if 
$\mathbf{A} \in \mathbb{ISP}\left(\mathsf{FDqRA}\right)$
or, equivalently, $\mathbf{A} \in \mathbb{IS}\left(\mathsf{EDqRA}\right)$. \\
\end{definition}

We say that a DqRA $\mathbf{A}$ is \emph{finitely} representable if the poset $(  X ,\leqslant )$ used in the representation of $\mathbf{A}$ is finite. 

Below we give some examples of representable DqRAs. These examples will become useful later on. \\

\begin{example}\label{ex:S2}
    Consider the two-element Sugihara chain $\mathbf{S}_2$ {\upshape(}viewed as a DqRA{\upshape)}.
    Let $\mathbf{X}_{\mathbf{S}_2}$ be the one-element poset
    with $X_{\mathbf{S}_2}=\{u\}$ and $\leqslant_{X_{\mathbf{S}_2}}=
    \alpha=\beta=\textnormal{id}_{X_{\mathbf{S}_2}}$.
    Then the lattice 
    $\mathsf{Up}( (X_{\mathbf{S}_2})^2,\preccurlyeq)$ is the
    two-element chain $\left(\left\{\varnothing,(X_{\mathbf{S}_2})^2\right\},\subseteq\right)$
    that forms a represenation of
    $\mathbf{S}_2$.  Figure~\ref{fig:S2Representation} contains
    a depiction of $\mathbf{X}_{\mathbf{S}_2}$, $\alpha$,
    $\beta$ and the blocks of $(X_{\mathbf{S}_2})^2$ as well as
    the representation of $\mathbf{S}_2$.\\
\end{example}

\begin{figure}[ht]
\centering
\begin{tikzpicture}[scale=1.5,pics/sample/.style={code={\draw[#1] (0,0) --(0.6,0) ;}},
    Dotted/.style={
    dash pattern=on 0.1\pgflinewidth off #1\pgflinewidth,line cap=round,
    shorten >=#1\pgflinewidth/2,shorten <=#1\pgflinewidth/2},
    Dotted/.default=3]
\begin{scope}[xshift=-0.2cm,box/.style = {draw,dashdotdotted,inner sep=20pt,rounded corners=5pt,thick}]
\node[draw,circle,inner sep=1.5pt] (u) at (-2.5,0.7) {};
\draw [->,dotted,thick] (u) edge[loop right, looseness=40]node{} (u);
\draw [->,dashed, thick] (u) edge[loop above, looseness=40]node{} (u);
\node[label,anchor=north,xshift=-1pt] at (u) {$u$};
\node[box,fit=(u)] {};
\node[label,anchor=north,xshift=8pt,yshift=-25pt] at (u) {$\mathbf{X}_{\mathbf{S}_2}$};
\path (-3.5,1) 
 node[matrix,anchor=north east,draw,nodes={anchor=center},inner sep=2pt, thick]  {
  \pic{sample=dashed}; & \node{$\alpha$}; \\
  \pic{sample=dotted}; & \node{$\beta$}; \\
  \pic{sample=dashdotdotted}; & \node{$E$ blocks}; \\
 };
\end{scope}

\begin{scope}[xshift=0.2cm]
\node[draw,circle,inner sep=1.5pt,fill] (0) at (0,0.3) {};
\node[draw,circle,inner sep=1.5pt,fill] (1) at (0,1) {};
\draw[order] (0)--(1);
\node[label,anchor=west] at (1) {$(X_{\mathbf{S}_2})^2$};
\node[label,anchor=west] at (0) {$\varnothing$};
\node[label,anchor=north,yshift=-7pt] at (0) {$\mathbf{S}_2$};
\end{scope}
\end{tikzpicture}
\caption{The poset $\mathbf{X}_{\mathbf{S}_2}$ used to represent $\mathbf{S}_2$ (left) and the
representation of $\mathbf{S}_2$ (right).}
\label{fig:S2Representation}
\end{figure}

\begin{example}[{\cite[Example 5.1]{CR-RDqRA}}]\label{ex:S3}
Recall that $\mathbf{S}_3$
denotes the three-element Sugihara chain  {\upshape(}viewed as a DqRA{\upshape)} with $S_3=\{a_{-1},a_0,a_1\}$ such that $a_0$ is the identity of the monoid operation. Let $\mathbf{X}_{\mathbf{S}_3}$ be the two-element antichain, i.e., $X_{\mathbf{S}_3} = \left\{x, y\right\}$ with ${\leqslant_{X_{\mathbf{S}_3}}} = \textnormal{id}_{X_{\mathbf{S}_3}}$, and consider the order automorphism $\alpha = \left\{\left(x, y\right), \left(y, x\right)\right\}$ 
and the dual order automorphism $\beta= \textnormal{id}_{X_{\mathbf{S}_3}}$ 
of $\mathbf{X}_{\mathbf{S}_3}$.
See Figure~\ref{fig:S3Representation}
for a depiction of the poset $\mathbf{X}_{\mathbf{S}_3}$, the
order automorphism $\alpha$, the dual
order automorphism $\beta$, and the
blocks of $(X_{\mathbf{S}_3})^2$.
Then the  lattice 
$\mathsf{Up}( (X_{\mathbf{S}_3})^2,\preccurlyeq)$ is the 16-element Boolean lattice. It can
easily be shown that
$\alpha \mathbin{;}(\leqslant_{X_{\mathbf{S}_3}}^c)^\smile =
(\leqslant_{X_{\mathbf{S}_3}}^c)^\smile \mathbin{;} \alpha
= {\sim}{\leqslant_{X_{\mathbf{S}_3}}} 
=
{\leqslant_{X_{\mathbf{S}_3}}}  $ 
and that the set of 
binary 
relations $\{\varnothing, \leqslant_{X_{\mathbf{S}_3}}, (X_{\mathbf{S}_3})^2\}$ forms a subuniverse  of 
$\langle \mathsf{Up}((X_{\mathbf{S}_3})^2,\preccurlyeq),\cap, \cup, \mathbin{;},  {\sim},{-}, \neg, \leqslant_{X_{\mathbf{S}_3}}, (\alpha \mathbin{;} {\leqslant_{X_{\mathbf{S}_3}}^{c\smile}})\rangle$, such that the subalgebra 
$\langle \{\varnothing, \leqslant_{X_{\mathbf{S}_3}}, (X_{\mathbf{S}_3})^2\} ,\cap, \cup, \mathbin{;},  {\sim},{-}, \neg, \leqslant_{X_{\mathbf{S}_3}}, \leqslant_{X_{\mathbf{S}_3}} \rangle$, 
is a representation of 
$\mathbf{S}_3$
{\upshape(}also depicted in 
Figure~\ref{fig:S3Representation}{\upshape)}.\\
\end{example}

\begin{figure}[ht]
\centering
\begin{tikzpicture}[scale=1.5,pics/sample/.style={code={\draw[#1] (0,0) --(0.6,0) ;}},
    Dotted/.style={
    dash pattern=on 0.1\pgflinewidth off #1\pgflinewidth,line cap=round,
    shorten >=#1\pgflinewidth/2,shorten <=#1\pgflinewidth/2},
    Dotted/.default=3]
\begin{scope}[ box/.style = {draw,dashdotdotted,inner sep=20pt,rounded corners=5pt,thick}]
\node[draw,circle,inner sep=1.5pt] (x) at (-2.5,0.7) {};
\node[draw,circle,inner sep=1.5pt] (y) at (-0.5,0.7) {};
\path (x) edge [->, bend left=25, dashed, thick] node {} (y);
\path (y) edge [->, bend left=25, dashed, thick] node {} (x);
\draw [->,dotted, thick] (x) edge[loop left]node{} (x);
\draw [->,dotted, thick] (y) edge[loop right]node{} (y);
\node[label,anchor=north,xshift=-1pt] at (x) {$x$};
\node[label,anchor=north,xshift=-1pt] at (y) {$y$};
\path (-0.7,2.5) 
 node[matrix,anchor=north east,draw,nodes={anchor=center},inner sep=2pt, thick]  {
  \pic{sample=dashed}; & \node{$\alpha$}; \\
  \pic{sample=dotted}; & \node{$\beta$}; \\
  \pic{sample=dashdotdotted}; & \node{$E$ blocks}; \\
 };
 \node[box,fit=(x)(y)] {};
\node[label,anchor=north,xshift=40pt,yshift=-25pt] at (x) {$\mathbf{X}_{\mathbf{S}_3}$};
\end{scope}

\begin{scope}[xshift=4.3cm,scale=0.7]
\node[draw,circle,inner sep=1.5pt,fill] (bot) at (0,0) {};
\node[draw,circle,inner sep=1.5pt,fill] (id) at (0,1) {};
\node[draw,circle,inner sep=1.5pt,fill] (top) at (0,2) {};
\draw[order] (bot)--(id)--(top);
\node[label,anchor=east,xshift=1pt] at (top) {$(X_{\mathbf{S}_3})^2$};
\node[label,anchor=east,xshift=1pt] at (id) {${\leqslant_{X_{\mathbf{S}_3}}} = {\sim}{\leqslant_{X_{\mathbf{S}_3}}}=-{\leqslant_{X_{\mathbf{S}_3}}}= \neg{\leqslant_{X_{\mathbf{S}_3}}}$};
\node[label,anchor=east,xshift=1pt] at (bot) {$\varnothing$};
\end{scope}

\end{tikzpicture}
\caption{The poset $\mathbf{X}_{\mathbf{S}_3}$ used to represent $\mathbf{S}_3$ (left) and the
representation of $\mathbf{S}_3$ (right).}
\label{fig:S3Representation}
\end{figure}

\begin{example} \label{Example:Representation-diamond}
Consider the first diamond $\Lalg_1$ given in Figure~\ref{fig:qRAs-examples}. It can be represented
over the two-element antichain $\mathbf{X}_{\Lalg_1}$, 
i.e., $X_{\Lalg_1} = \{u, v\}$ and ${\leqslant_{X_{\Lalg_1}}} = \textnormal{id}_{X_{\Lalg_1}}$, with $E = \alpha = \beta = \textnormal{id}_{X_{\Lalg_1}}$.
Hence, in this example the algebra $\mathbf{Dq}(\mathbf{E})$ is \emph{not} a full distributive quasi relation algebra. 
The poset $\mathbf{X}_{\Lalg_1}$, the
order automorphism $\alpha$, the dual
order automorphism $\beta$ and the
blocks of $E$ are depicted in Figure~\ref{fig:DiamondRepresentation}.  Clearly the set of
binary relations 
$\left\{\varnothing,\{(u,u)\},\{(v,v)\},\leqslant_{X_{\Lalg_1}}\right\}$ forms
a subuniverse of the algebra
$\langle \mathsf{Up}( E,\preccurlyeq),\cap, \cup, \mathbin{;}, {\sim}, {-}, \neg, \leqslant_{X_{\Lalg_1}}, (\alpha \mathbin{;} {\leqslant_{X_{\Lalg_1}}^{c\smile}})\rangle$.
We
have that the monoid identity $1_{\mathbf{L}_1}$ is mapped to ${\leqslant_{X_{\Lalg_1}}} = E$, $0$ to $\varnothing$, $a$ to $\{(u, u)\}$ and $b$ to $\{(v,v)\}$.
\end{example}

\begin{figure}[ht]
\centering   
\begin{tikzpicture}[scale=1.5,pics/sample/.style={code={\draw[#1] (0,0) --(0.6,0) ;}},
Dotted/.style={
dash pattern=on 0.1\pgflinewidth off #1\pgflinewidth,line cap=round,
shorten >=#1\pgflinewidth/2,shorten <=#1\pgflinewidth/2},
Dotted/.default=3]
\begin{scope}[ box/.style = {draw,dashdotdotted,inner sep=20pt,rounded corners=5pt,thick}]
\node[draw,circle,inner sep=1.5pt] (u) at (-2.5,0.7) {};
\node[draw,circle,inner sep=1.5pt] (v) at (-0.5,0.7) {};
\draw [->,dashed, thick] (u) edge[loop left]node{} (u);
\draw [->,dashed, thick] (v) edge[loop right]node{} (v);
\node[label,anchor=north,xshift=-1pt] at (u) {$u$};
\node[label,anchor=north,xshift=-1pt] at (v) {$v$};
\path (-0.7,2.5) 
node[matrix,anchor=north east,draw,nodes={anchor=center},inner sep=2pt, thick]  {
\pic{sample=dashed}; & \node{$\alpha,\beta$}; \\
\pic{sample=dashdotdotted}; & \node{$E$ blocks}; \\
};
\node[box,fit=(u)] {};
\node[box,fit=(v)] {};
\node[label,anchor=north,xshift=40pt,yshift=-25pt] at (u) {$\mathbf{X}_{\Lalg_1}$};
\end{scope}	

\begin{scope}[xshift=3.2cm,scale=0.7]
\node[draw,circle,inner sep=1.5pt,fill] (0) at (0,0) {};
\node[draw,circle,inner sep=1.5pt,fill] (a) at (-1,1) {};
\node[draw,circle,inner sep=1.5pt,fill] (b) at (1,1) {};
\node[draw,circle,inner sep=1.5pt,fill] (1) at (0,2) {};
\draw[order] (0)--(a)--(1);
\draw[order] (0)--(b)--(1);
\node[label,anchor=south] at (1) {${\leqslant_{X_{\Lalg_1}}}$};
\node[label,anchor=east,xshift=1pt] at (a) {$\left\{\left(u, u\right)\right\}$};
\node[label,anchor=west,xshift=-1pt] at (b) {$\left\{\left(v, v\right)\right\}$};
\node[label,anchor=north] at (0) {$\varnothing$};
\end{scope}

\end{tikzpicture}

\caption{The poset $\mathbf{X}_{\mathbf{L}_1}$ used to represent the diamond
$\Lalg_1$ (left) and the representation of $\mathbf{L}_1$ (right).}
\label{fig:DiamondRepresentation}
\end{figure}

\section{Representable DqRAs via nested sums}\label{sec:KL-rep}

In this section we prove that if $\Lalg$ is a representable distributive quasi relation algebra,
then the nested sum of $\mathbf{S}_3$ and $\Lalg$ is representable.\\

\begin{theorem}\label{Thm:K[L]represenatble}
    Let 
	$\mathbf{L}$ 
	be a {\upshape(}finitely{\upshape)} representable DqRA.
	Then $\mathbf{S}_3\left[\mathbf L\right]$ 
	is a {\upshape(}finitely{\upshape)} representable DqRA. \\
\end{theorem}


In order to prove Theorem~\ref{Thm:K[L]represenatble},
we use the representations
of $\mathbf{S}_3$ and $\mathbf{L}$ to construct a suitable
candidate with which to represent $\mathbf{S}_3\left[\mathbf L\right]$.
Through a series of lemmas we then show that the construction
does indeed produce a representation of the DqRA $\mathbf{S}_3\left[\mathbf L\right]$ (as summarized in the
proof at the end of the section).

Let $\mathbf{K}$ be the odd three-element Sugihara chain
described in Section~\ref{sec:Prelim}, i.e.,  
$\mathbf{K}=\langle \{a_{-1},a_0,a_1\},
\wedge,\vee,\cdot,{\sim},{\sim},{\sim}, a_0,a_0\rangle$ {\upshape(}also
see Figure~\ref{fig:Sugihara-examples}{\upshape)}.
Then $a_0$, the monoid identity of $\mathbf{K}$, is totally irreducible and $\mathbf{K}$
is odd, conic
and representable.  Recall from Example~\ref{ex:S3} that the representation uses the two-element antichain 
$\mathbf X_{\mathbf{K}} = \left(X_{\mathbf{K}}, \leqslant_{X_{\mathbf{K}}}\right)$
with
$X_{\mathbf{K}}=\{x,y\}$ and 
the equivalence relation $E_{\mathbf{K}}=X_{\mathbf{K}}\times X_{\mathbf{K}}$. The order automorphism $\alpha_{\mathbf{K}}:X_{\mathbf{K}}\to X_{\mathbf{K}}$
is defined by $\alpha_\K(x)=y$ and $\alpha_\K(y)=x$, while
the dual order automorphism $\beta_{\K}:X_{\mathbf{K}}\to X_{\mathbf{K}}$ is the identity map and 
$\alpha_{\mathbf{K}}\mathbin{;} \beta_{\mathbf{K}} \mathbin{;} \alpha_{\mathbf{K}}= \beta_{\mathbf{K}}$.
Elements of $X_{\K}$ will be denoted by $w_1,w_2,\ldots$ or just $w$.
Finally, $\mathbf{K}$ can be embedded into the
algebra of up-sets of 
$\left(E_{\mathbf{K}},\preccurlyeq_{E_\K}\right)$ via the embedding
$\varphi_{\mathbf{K}}: 
\K \hookrightarrow \mathbf{Dq}\left(E_\mathbf{K}, \preccurlyeq_{E_\mathbf{K}}\right)$ such that 
$\varphi_{\mathbf{K}}(a_{-1})=\varnothing$,
$\varphi_{\mathbf{K}}(a_{0})={\leqslant_{X_\mathbf{K}}}$
and
$\varphi_{\mathbf{K}}(a_{1})=E_{\mathbf{K}}=X_{\mathbf{K}}\times X_{\mathbf{K}}$.

Now let $\mathbf{L}$ be any representable DqRA. Then 
there exists a poset $\mathbf X_\mathbf{L} = \left(X_\mathbf{L}, \leqslant_{X_\mathbf{L}}\right)$,
an equivalence relation $E_\mathbf{L}\subseteq X_\mathbf{L}\times X_\mathbf{L}$ with
${\leqslant_{X_\mathbf{L}}}\subseteq E_\mathbf{L}$, an
order automorphism $\alpha_\mathbf{L}:X_\mathbf{L}\to X_\mathbf{L}$
and a self-inverse dual order automorphism $\beta_\mathbf{L}:X_\mathbf{L}\to X_\mathbf{L}$ such that 
$\alpha_\mathbf{L},\beta_\mathbf{L}\subseteq E_\mathbf{L}$ and
$\alpha_\mathbf{L}\mathbin{;} \beta_\mathbf{L} \mathbin{;} \alpha_\mathbf{L} = \beta_\mathbf{L}$.  Elements of the set $X_{\mathbf{L}}$ will
be denoted by $z_1,z_2,\ldots$ or just $z$.
Moreover, there exists an embedding
$\varphi_\mathbf{L}: \mathbf{L} \hookrightarrow \mathbf{Dq}\left(E_\mathbf{L}, \preccurlyeq_{E_\mathbf{L}}\right)$ (see Section~\ref{sec:RDqRA} for the details).

We first describe the poset used to represent $\mathbf{K}[\mathbf{L}]$. 
For each $[z]\in X_{\Lalg}/E_{\Lalg}$, 
let ${}^{[z]}\mathbf{X}_{\K}$ and ${}_{[z]}\mathbf{X}_{\K}$ each be a
copy of $\mathbf{X}_{\K}$
with ${}^{[z]}\!X_{\K}=\{x^{[z]},y^{[z]}\}$ and ${}_{[z]}\!X_{\K}=\{x_{[z]},y_{[z]}\}$. 
Thus, both ${}^{[z]}\mathbf{X}_{\K}$ and ${}_{[z]}\mathbf{X}_{\K}$
are two-element antichains with their orderings denoted by
${}^{[z]}{\leqslant_{X_{\K}}}$ and ${}_{[z]}{\leqslant_{X_{\K}}}$, respectively.
Now let 
\begin{equation} 
X_{\koflbf}:=X_{\Lalg}
\cup\, 
\bigcup\left\{{}^{[z]}X_{\K} \;\Big\vert \; [z]\in X_{\Lalg}/E_{\Lalg}\right\}
\cup\,
\bigcup\left\{{}_{[z]}X_{\K}\;\Big\vert \; [z]\in X_{\Lalg}/E_{\Lalg}\right\}\label{Def:XKofL}\end{equation}
and we use $u$'s and $v$'s to denote elements of
$X_{\koflbf}$.

We define 
${\leqslant_{X_{\koflbf}}}\subseteq X_{\koflbf}\times X_{\koflbf}$ by
\begin{align}
{\leqslant_{X_{\koflbf}}}:=&\,
\bigcup\left\{{}^{[z]}{\leqslant_{X_{\K}}}\;\Big\vert \; [z]\in X_{\Lalg}/E_{\Lalg}\right\}
\cup\,
\bigcup\left\{{}_{[z]}{\leqslant_{X_{\K}}}\;\Big\vert \;[z]\in X_{\Lalg}/E_{\Lalg}\right\}
\cup\,\nonumber\\
&\,\bigcup\left\{[z]\times {}^{[z]}\!X_{\K}\;\Big\vert \; [z]\in X_{\Lalg}/E_{\Lalg}\right\}
\cup\,
\bigcup\left\{{}_{[z]}\!X_{\K}\times [z]\;\Big\vert \; [z]\in X_{\Lalg}/E_{\Lalg}\right\}\cup\nonumber\\
&\,\bigcup\left\{{}_{[z]}\!X_{\K}\times{}^{[z]}\!X_{\K}\;\Big\vert \; [z]\in X_{\Lalg}/E_{\Lalg}\right\}\cup \leqslant_{X_{\Lalg}}.
\label{Def:IneqKofL}    
\end{align}
Then $\mathbf{X}_{\koflbf}=(X_{\koflbf},\leqslant_{X_{\koflbf}})$
is a poset.  Intuitively, $\mathbf{X}_{\koflbf}$ is obtained from 
$\mathbf{X}_{\Lalg}$ by placing a copy of $\mathbf{X}_{\K}$
above and below (${}^{[z]}\!\mathbf{X}_{\K}$ and 
${}_{[z]}\!\mathbf{X}_{\K}$, respectively)
each block $[z]$ of $\mathbf{X}_{\Lalg}$ modulo
$E_{\Lalg}$. \\

\begin{example}\label{Example:XofKofL}
Let $\K=\mathbf{S}_3$ from Figure~\ref{fig:Sugihara-examples} and $\Lalg=\Lalg_1$ the first diamond from
Figure~\ref{fig:qRAs-examples}.  Then $\K[\Lalg]$
is the DqRA depicted in Figure~\ref{Figure:GOS1}.
Recall from Examples~\ref{ex:S3} and~\ref{Example:Representation-diamond},
respectively, that both $\K$ and $\Lalg$ are representable,
with the posets used to find their representations depicted
in Figures~\ref{fig:S3Representation} and ~\ref{fig:DiamondRepresentation}, respectively. Using these
posets we can construct the poset $\mathbf{X}_{\koflbf}$,
as described in {\upshape (\ref{Def:XKofL})} and {\upshape (\ref{Def:IneqKofL})}. The structure 
$(X_{\koflbf},\leqslant_{X_\mathbf{K[L]}})$ is drawn 
 in Figure~\ref{fig:XofKofL}.
\\
\end{example}

\begin{figure}[ht]
    \centering
\begin{tikzpicture}
\begin{scope}
\node[draw,circle,inner sep=1.5pt] (y) at (-1,1) {};
\node[draw,circle,inner sep=1.5pt] (z) at (1,1) {};
\node[draw,circle,inner sep=1.5pt] (x) at (0,0) {};
\node[draw,circle,inner sep=1.5pt] (v) at (-1,-1) {};
\node[draw,circle,inner sep=1.5pt] (w) at (1,-1) {};
\draw[order] (y)--(x)--(v);
\draw[order] (z)--(x)--(w);

\node[label,anchor=south,yshift=-1pt] at (y) {$x^{[u]}$};
\node[label,anchor=south,yshift=-2pt] at (z) {$y^{[u]}$};
\node[label,anchor=north] at (v) {$x_{[u]}$};
\node[label,anchor=north] at (w) {$y_{[u]}$};
\node[label,anchor=east] at (x) {$u$};
\end{scope}

\begin{scope}[xshift=3.5cm]
\node[draw,circle,inner sep=1.5pt] (y) at (-1,1) {};
\node[draw,circle,inner sep=1.5pt] (z) at (1,1) {};
\node[draw,circle,inner sep=1.5pt] (x) at (0,0) {};
\node[draw,circle,inner sep=1.5pt] (v) at (-1,-1) {};
\node[draw,circle,inner sep=1.5pt] (w) at (1,-1) {};
\draw[order] (y)--(x)--(v);
\draw[order] (z)--(x)--(w);

\node[label,anchor=south,yshift=-1pt] at (y) {$x^{[v]}$};
\node[label,anchor=south,yshift=-2pt] at (z) {$y^{[v]}$};
\node[label,anchor=north] at (v) {$x_{[v]}$};
\node[label,anchor=north] at (w) {$y_{[v]}$};
\node[label,anchor=west] at (x) {$v$};
\end{scope}
\end{tikzpicture}
    \caption{The poset $\mathbf{X}_{\koflbf}$ from Example~\ref{Example:XofKofL}.}
    \label{fig:XofKofL}
\end{figure}

Next, we set
\begin{equation}
    E_{\koflbf} :=
    \bigcup\left\{\left([z]\cup\,^{[z]}\!X_{\K}\cup\,_{[z]}\!X_{\K}\right)^2 \;\Big\vert \;
    [z]\in X_\Lalg/E_\Lalg\right\}.
    \label{EKofLDef}
\end{equation}
A block in $X_{\koflbf}/E_{\koflbf}$ consists
of a block $[z]$ in $X_{\Lalg}/E_{\Lalg}$ along with
$^{[z]}\!\mathbf{X}_{\K}$ and $_{[z]}\!\mathbf{X}_{\K}$, the 
copies of $\mathbf{X}_{\K}$ associated with it.
Using the fact that $E_{\Lalg}$ is an equivalence relation on
$X_{\Lalg}$ we can show that $E_{\koflbf}$ is an 
equivalence relation on $X_{\koflbf}$.
Moreover, since ${\leqslant_{X_{\Lalg}}} \subseteq E_{\Lalg}$,
it follows that ${\leqslant_{\koflbf}} \subseteq E_{\koflbf}$.
Recall that $E_{\koflbf}$ can be partially ordered as follows, 
for $(u_1,u_2),(v_1,v_2)\in E_{\koflbf}$,
\begin{align}
	&(u_1,u_2)\preccurlyeq_{E_{\K[\Lalg]}} (v_1,v_2)\qquad 
	\iff \qquad v_1\leqslant_{X_{\koflbf}}u_1\quad\text{and}\quad
	u_2\leqslant_{X_{\koflbf}}v_2\label{EIneq}	
\end{align}

Define $\alpha_{\koflbf}:X_{\koflbf}  \to X_{\koflbf}$ by 
\begin{numcases}{\alpha_{\koflbf}(u)=}
     \alpha_{\Lalg}(u)& if $u\in X_{\Lalg}$\nonumber\\
     \left(\alpha_{\K}(w)\right)^{[z]}& if $u=w^{[z]}$ for some $w\in X_{\K}$ and some $[z]$ in $X_{\Lalg}/E_{\Lalg}$\nonumber\\
     \left(\alpha_{\K}(w)\right)_{[z]}& if $u=w_{[z]}$ for some $w\in X_{\K}$ and some $[z]$ in $X_{\Lalg}/E_{\Lalg}$,\label{Def:AlphaKofL}
\end{numcases}  
and $\beta_{\koflbf}: X_{\koflbf} \to X_{\koflbf}$ by 
\begin{numcases}{\beta_{\koflbf}(u)=}
     \beta_{\Lalg}(u)& if $u\in X_{\Lalg}$\nonumber\\
     \left(\beta_{\K}(w)\right)_{[z]}& if $u=w^{[z]}$ for some $w\in X_{\K}$ and some $[z]$ in $X_{\Lalg}/E_{\Lalg}$\nonumber\\
     \left(\beta_{\K}(w)\right)^{[z]}& if $u=w_{[z]}$ for some $w\in X_{\K}$ and some $[z]$ in $X_{\Lalg}/E_{\Lalg}$,\label{Def:BetaKofL}
\end{numcases}
for $u\in X_{\koflbf}$.\\

\begin{example}\label{Example:EquivAlphaBetaXofKofL}
Let $\K$ and $\Lalg$ be as described in Example~\ref{Example:XofKofL} with $\mathbf{X}_{\koflbf}$
the poset depicted in Figure~\ref{fig:XofKofL}.
Then Figure~\ref{Fig:EquivAlphaBetaXofKofL} depicts
the equivalence classes of $E_{\koflbf}$ as well as 
the maps $\alpha_{\koflbf}$ and $\beta_{\koflbf}$
described in {\upshape(\ref{EKofLDef})},{\upshape (\ref{Def:AlphaKofL})}
and {\upshape(\ref{Def:BetaKofL})}, respectively.\\
\end{example}

\begin{figure}[ht]
    \centering
    \begin{tikzpicture}[pics/sample/.style={code={\draw[#1] (0,0) --(0.6,0) ;}},
    Dotted/.style={
    dash pattern=on 0.1\pgflinewidth off #1\pgflinewidth,line cap=round,
    shorten >=#1\pgflinewidth/2,shorten <=#1\pgflinewidth/2},
    Dotted/.default=3]
        \begin{scope}[box/.style = {draw,dashdotdotted,inner sep=20pt,rounded corners=5pt,thick}]
\node[draw,circle,inner sep=1.5pt] (y) at (-1,1) {};
\node[draw,circle,inner sep=1.5pt] (z) at (1,1) {};
\node[draw,circle,inner sep=1.5pt] (x) at (0,0) {};
\node[draw,circle,inner sep=1.5pt] (v) at (-1,-1) {};
\node[draw,circle,inner sep=1.5pt] (w) at (1,-1) {};
\draw[order] (y)--(x)--(v);
\draw[order] (z)--(x)--(w);
\node[label,anchor=south,yshift=-1pt] at (y) {$x^{[u]}$};
\node[label,anchor=south,yshift=-2pt, xshift=5pt] at (z) {$y^{[u]}$};
\node[label,anchor=north] at (v) {$x_{[u]}$};
\node[label,anchor=north, xshift=5pt] at (w) {$y_{[u]}$};
\node[label,anchor=east] at (x) {$u$};
\path (y) edge [->, bend left=23, dashed, thick] node {} (z);
\path (z) edge [->, bend left=23, dashed, thick] node {} (y);
\path (v) edge [->, bend left=23, dashed, thick] node {} (w);
\path (w) edge [->, bend left=23, dashed, thick] node {} (v);
\draw [->,dashed, thick] (x) edge[loop above,looseness=40]node{} (x);
\path (y) edge [->, bend left=23, dotted, thick] node {} (v);
\path (z) edge [->, bend left=23, dotted, thick] node {} (w);
\path (v) edge [->, bend left=23, dotted, thick] node {} (y);
\path (w) edge [->, bend left=23, dotted, thick] node {} (z);
\draw [->,dotted, thick] (x) edge[loop right,looseness=40]node{} (x);
\node[box,fit=(y)(z)(v)(w)(x),] {};
\end{scope}

\begin{scope}[xshift=4cm, box/.style = {draw,dashdotdotted,inner sep=20pt,rounded corners=5pt,thick}]
\node[draw,circle,inner sep=1.5pt] (y) at (-1,1) {};
\node[draw,circle,inner sep=1.5pt] (z) at (1,1) {};
\node[draw,circle,inner sep=1.5pt] (x) at (0,0) {};
\node[draw,circle,inner sep=1.5pt] (v) at (-1,-1) {};
\node[draw,circle,inner sep=1.5pt] (w) at (1,-1) {};
\draw[order] (y)--(x)--(v);
\draw[order] (z)--(x)--(w);
\node[label,anchor=south,yshift=-1pt] at (y) {$x^{[v]}$};
\node[label,anchor=south,yshift=-2pt, xshift=5pt] at (z) {$y^{[v]}$};
\node[label,anchor=north] at (v) {$x_{[v]}$};
\node[label,anchor=north, xshift=5pt] at (w) {$y_{[v]}$};
\node[label,anchor=west] at (x) {$v$};

\path (y) edge [->, bend left=23, dashed, thick] node {} (z);
\path (z) edge [->, bend left=23, dashed, thick] node {} (y);
\path (v) edge [->, bend left=23, dashed, thick] node {} (w);
\path (w) edge [->, bend left=23, dashed, thick] node {} (v);
\draw [->,dashed, thick] (x) edge[loop above,looseness=40]node{} (x);
\path (y) edge [->, bend left=23, dotted, thick] node {} (v);
\path (z) edge [->, bend left=23, dotted, thick] node {} (w);
\path (v) edge [->, bend left=23, dotted, thick] node {} (y);
\path (w) edge [->, bend left=23, dotted, thick] node {} (z);
\draw [->,dotted, thick] (x) edge[loop left,looseness=40]node{} (x);
\node[box,fit=(y)(z)(v)(w)(x),] {};
\path (5.5,1) 
 node[matrix,anchor=north east,draw,nodes={anchor=center},inner sep=2pt, thick]  {
  \pic{sample=solid}; & \node{$\leqslant_{X_{\koflbf}}$}; \\
  \pic{sample=dashed}; & \node{$\alpha_{\koflbf}$}; \\
  \pic{sample=dotted}; & \node{$\beta_{\koflbf}$}; \\
  \pic{sample=dashdotdotted}; & \node{$E_{\koflbf}$ blocks}; \\
 };
\end{scope}
    \end{tikzpicture}
     \caption{$E_{\koflbf}$, $\alpha_{\koflbf}$ and
    $\beta_{\koflbf}$ on $\mathbf{X}_{\koflbf}$ in Example~\ref{Example:EquivAlphaBetaXofKofL}.}
    \label{Fig:EquivAlphaBetaXofKofL}
\end{figure}

\begin{lemma}\label{Lem:AlphaBeta}
\begin{enumerate}[label=(\roman*)]
		\item The map $\alpha_{\koflbf}: X_{\koflbf} \to X_{\koflbf}$ is 
		an order automorphism of $\mathbf{X}_{\koflbf}$ such that
		$\alpha_{\koflbf} \subseteq E_{\koflbf}$.\label{Lem:AlphaBeta:Item1}
		\item The map $\beta_{\koflbf}: X_{\koflbf} \to X_{\koflbf}$ is a 
		self-inverse dual order automorphism of $\mathbf{X}_{\koflbf}$ 
		such that $\beta_{\koflbf} \subseteq E_{\koflbf}$.\label{Lem:AlphaBeta:Item2}
		\item $\alpha_{\koflbf}\mathbin{;}\beta_{\koflbf}\mathbin{;}\alpha_{\koflbf} 
		= \beta_{\koflbf}$\label{Lem:AlphaBeta:Item3}
	\end{enumerate}
\end{lemma}
\begin{proof}
	We prove items \ref{Lem:AlphaBeta:Item2} and \ref{Lem:AlphaBeta:Item3}.
	The proof of item \ref{Lem:AlphaBeta:Item1} is similar to that of
	\ref{Lem:AlphaBeta:Item2}. 
	
	To prove \ref{Lem:AlphaBeta:Item2} we first show that 
        $\beta_{\koflbf}\subseteq E_{\koflbf}$.  Let $u,v\in X_{\koflbf}$ such
        that $\beta_{\koflbf}(u)=v$.  There are three cases to consider:

        \vskip 0.2cm
        \noindent
	\underline{Case 1:}
        If $u\in X_{\Lalg}$, then $v\in X_{\Lalg}$ and
        $v=\beta_{\koflbf}(u)=\beta_{\Lalg}(u)$, and so $(u, v) \in   
        \beta_{\Lalg}\subseteq E_{\Lalg}$. Hence, $u \in [u]$ and $v \in [u]$, which means $(u, v) \in E_{\koflbf}$.

        \vskip 0.2cm
        \noindent
	\underline{Case 2:}
        If $u\in {}^{[z]}\!\!X_{\K}$ for some $[z]\in X_{\Lalg}/E_{\Lalg}$, then
        $v\in {}_{[z]}\!X_{\K}$ and there exist elements $w_1,w_2\in X_{\K}$ such that 
        $u=(w_1)^{[z]}$ and $v=(w_2)_{[z]}$. 
        Since ${}^{[z]}\!X_{\K}\times {}_{[z]}\!X_{\K}\subseteq E_{\koflbf}$
        by~(\ref{EKofLDef}), it follows that $(u,v)\in E_{\koflbf}.$

        \vskip 0.2cm
        \noindent
	\underline{Case 3:}
        If $u\in {}_{[z]}\!X_{\K}$ for some $[z]\in X_{\Lalg}/E_{\Lalg}$, then
        $v\in {}^{[z]}\!X_{\K}$ and there exist elements $w_1,w_2\in X_{\K}$ such that 
        $u=(w_1)_{[z]}$ and $v=(w_2)^{[z]}$. 
        Since ${}_{[z]}\!X_{\K}\times {}^{[z]}\!\!X_{\K}\subseteq E_{\koflbf}$
        by~(\ref{EKofLDef}), it follows that $(u,v)\in E_{\koflbf}.$

        \vskip 0.2cm
	To see that $\beta_{\koflbf}$ is a dual order 
        automorphism, let $u,v\in X_{\koflbf}$.
        There are six cases to consider.

        \vskip 0.2cm
	\noindent
	\underline{Case 1:} If $u,v\in X_{\Lalg}$, then
        \[u\leqslant_{X_{\koflbf}} v\iff u\leqslant_{X_{\Lalg}}v\iff
        \beta_{\Lalg}(v)\leqslant_{X_{\Lalg}}\beta_{\Lalg}(u) \iff
        \beta_{\koflbf}(v)\leqslant_{X_{\koflbf}}\beta_{\koflbf}(u).\]
   
	\noindent
	\underline{Case 2:} 
        If $u,v\in {}^{[z]}\!X_{\K}$ for some
        $[z]\in X_{\Lalg}/E_{\Lalg}$, then there exist 
        $w_1,w_2\in X_{\K}$ such that 
        $u=(w_1)^{[z]}$ and $v=(w_2)^{[z]}$. Hence,
        \begin{align*}          
        u\leqslant_{X_{\koflbf}}v  \iff
        w_1\leqslant_{X_{\K}} w_2 & \iff
        \beta_{\K}(w_2)\leqslant_{X_{\K}}\beta_{\K}(w_1) \\
        &\iff
         \left(\beta_{\K}(w_2)\right)_{[z]}\leqslant_{X_{\koflbf}}\left(\beta_{\K}(w_1)\right)_{[z]}\\
         & \iff
        \beta_{\koflbf}(v)\leqslant_{X_{\koflbf}}
        \beta_{\koflbf}(u). 
        \end{align*}

	\noindent
	\underline{Case 3:}
        If $u,v\in {}_{[z]}\!X_{\K}$ for some
        $[z]\in X_{\Lalg}/E_{\Lalg}$, the proof is 
        similar to the proof of the previous case
        described above.

        \vskip 0.2cm
	\noindent
	\underline{Case 4:} 
         If $u\in X_{\Lalg}$ and $v\in {}^{[u]}\!X_{\K}$,
        then $\left(u,\beta_{\Lalg}(u)\right)\in\beta_{\Lalg}
        \subseteq E_{\Lalg}$, i.e., $\beta_{\Lalg}(u)\in[u]$, and there exist $w\in X_{\K}$ such that $v=w^{[u]}$.
        It follows that
        \begin{align*}
        u\leqslant_{X_{\koflbf}} v
        \iff
        u\leqslant_{X_{\koflbf}}w^{[u]}
        & \iff
        \left(\beta_{\K}(w)\right)_{[u]}\leqslant_{X_{\koflbf}}\beta_{\Lalg}(u)\\       
        & \iff
        \beta_{\koflbf}(v)\leqslant_{X_{\koflbf}}
        \beta_{\koflbf}(u).
        \end{align*}

	\noindent
	\underline{Case 5:} 
        If $u\in {}_{[u]}\!X_{\K}$ and $v\in X_{\Lalg}$,
        the proof is similar to the proof of the previous case described above.

        \vskip 0.2cm
	\noindent
	\underline{Case 6:} 
        If $u\in {}_{[z]}\!X_{\K}$ and $v\in {}^{[z]}\!X_{\K}$
        for some $[z]\in X_{\Lalg}/E_{\Lalg}$, then there
        exist $w_1,w_2\in X_{\K}$ such that $u=(w_1)_{[z]}$
        and $v=(w_2)^{[z]}$. Thus,
        \begin{align*}
        u\leqslant_{X_{\koflbf}}v\iff
        (w_1)_{[z]}\leqslant_{X_{\koflbf}} (w_2)^{[z]} & \iff
        \left(\beta_{\K}(w_2)\right)_{[z]}\leqslant_{X_{\koflbf}}
        \left(\beta_{\K}(w_1)\right)^{[z]}\\
        & \iff
        \beta_{\koflbf}(v)\leqslant_{X_{\koflbf}}\beta_{\koflbf}(u).
        \end{align*}
        Note that for any other combination of elements $u,v\in X_{\koflbf}$ it
        will follow immediately from~(\ref{Def:IneqKofL}) 
        and~(\ref{Def:BetaKofL}) that $u\nleqslant_{X_{\koflbf}} v$
        and $\beta_{\koflbf}(v)\nleqslant_{X_{\koflbf}} \beta_{\koflbf}(u)$.


    \vskip 0.2cm
        Next we show that $\beta_{\koflbf}$ is self-inverse. Let $u \in X_{\koflbf}$. There are three cases.

        \vskip 0.2cm
        \noindent
	\underline{Case 1:} If $u \in X_\Lalg$, then $\beta_{\koflbf}(u) = \beta_\Lalg(u) \in X_\Lalg$ and so $\beta_{\koflbf}\left(\beta_{\koflbf}(u)\right) = \beta_{\Lalg}\left(\beta_\Lalg(u)\right))$. Hence, since  $\beta_\Lalg$ is self-inverse, we get $\beta_{\koflbf}\left(\beta_{\koflbf}(u)\right) = u$. 

    \vskip 0.2cm 
    \noindent
    \underline{Case 2:} If $u 
    \in {}_{[z]}\!X_{\K}$ for some
        $[z]\in X_{\Lalg}/E_{\Lalg}$, then there exists $w\in X_{\K}$
        such that $u=w_{[z]}$. Hence, 
        \begin{align*}
        \beta_{\koflbf}\left(\beta_{\koflbf}(u)\right) = \beta_{\koflbf}\left(\beta_{\koflbf}\left(w_{[z]}\right)\right) & = \beta_{\koflbf}\left(\left(\beta_\K(w)\right)^{[z]}\right)\\
        & = \left(\beta_\K\left(\beta_\K(w)\right)\right)_{[z]} = w_{[z]}.
        \end{align*}
        The last equality follows from the fact that $\beta_\K$ is self-inverse. 

    \vskip 0.2cm 
    \noindent
    \underline{Case 3:} If $u\in {}^{[z]}\!X_{\K}$ for some
        $[z]\in X_{\Lalg}/E_{\Lalg}$, the proof is similar to the previous case.

    \vskip 0.2cm

        Finally we show that $\beta_{\koflbf}$
        is surjective.  Let $u \in X_{\koflbf}$. We have to find some $v \in X_{\koflbf}$ such that $\beta_{\koflbf}(v) = u$. 
        There are three cases to consider:
        \vskip 0.2cm
        \noindent
	\underline{Case 1:} If $u 
    \in {}_{[z]}\!X_{\K}$ for some
        $[z]\in X_{\Lalg}/E_{\Lalg}$, then there exists $w_1\in X_{\K}$
        such that $u=(w_1)_{[z]}$. 
        Since $\beta_{\K}$ is surjective, there is some $w_2 \in X_{\K}$ such that $\beta_{\K}(w_2)=w_1$
        and hence $u=\left(w_1\right)_{[z]}= \left(\beta_{\K}(w_2)\right)_{[z]} = \beta_{\K[\Lalg]}\left(\left(w_2\right)^{[z]}\right)$. 

        \vskip 0.2cm
        \noindent
	\underline{Case 2:} If $u\in {}^{[z]}\!X_{\K}$, the proof is similar to the previous case.

        \vskip 0.2cm
        \noindent
	\underline{Case 3:}
        If $u\in X_{\Lalg}$, then the surjectivity of $\beta_\Lalg$ implies that there is some $v \in X_\Lalg$ such that 
        $u = \beta_{\Lalg}(v)=\beta_{\koflbf}(v)$.


        To prove item  \ref{Lem:AlphaBeta:Item3} let $u \in X_{\koflbf}$. We consider three cases.

        \vskip 0.2cm
        \noindent
	\underline{Case 1:} 
        If $u\in X_{\Lalg}$, then, since $\alpha_{\Lalg}\mathbin{;}\beta_{\Lalg}\mathbin{;}\alpha_{\Lalg}= \beta_{\Lalg}$,
        we have that 
        \[\alpha_{\koflbf}\left(\beta_{\koflbf}
        \left(\alpha_{\koflbf}\left(u\right)\right)\right)       =\alpha_{\Lalg}\left(\beta_{\Lalg}\left(\alpha_{\Lalg}\left(u\right)\right)\right)
        =\beta_{\Lalg}(u)=\beta_{\koflbf}(u).\]

        \noindent
	\underline{Case 2:} 
        If $u\in {}^{[z]}\!X_{\K}$ for some $[z]\in X_{\Lalg}/E_{\Lalg}$, then $u=w^{[z]}$ for some $w\in X_{\K}$.
        Since $\alpha_{\K}\mathbin{;}\beta_{\K}\mathbin{;}\alpha_{\K}= \beta_{\K}$, we have that 
        \begin{align*}                           \alpha_{\koflbf}\left(\beta_{\koflbf}\left(\alpha_{\koflbf}\left(u\right)\right)\right)
        &=\alpha_{\koflbf}\left(\beta_{\koflbf}\left(\left(\alpha_{\K}(w)\right)^{[z]}\right)\right)\\
        &=\alpha_{\koflbf}\left(\left(\beta_{\K}\left(\alpha_{\K}(w)\right)\right)_{[z]}\right)\\
        &=\left(\alpha_{\K}\left(\beta_{\K}\left(\alpha_{\K}(w)\right)\right)\right)_{[z]}\\
        &=\left(\beta_{\K}(w)\right)_{[z]}=\beta_{\koflbf}(u).
        \end{align*}

        \noindent
\underline{Case 3:}
        If $u\in {}_{[z]}\!X_{\K}$ for some $[z]\in X_{\Lalg}/E_{\Lalg}$ the proof is similar to the previous case.\\
        \end{proof}

Lemma~\ref{Lem:AlphaBeta} shows that $\alpha_{\K[\Lalg]}$ and $\beta_{\K[\Lalg]}$ satisfy the conditions of Theorem~\ref{thm:Dq(E)}, so we obtain the following result.\\

\begin{theorem}\label{thm:XKL-gives-DqRA}
The algebra $\left\langle \mathsf{Up}\left(E_{\koflbf}, \preccurlyeq_{E_{\koflbf}}\right), \cap, \cup, \mathbin{;}, {\sim}, {-}, {\neg}, \leqslant_{X_{\K[\Lalg]}}, 
0_{X_{\K[\Lalg]}}
\right\rangle$  is a DqRA.\\
\end{theorem}

Next we define a map $\psi:K\backslash\{a_0\}\cup L\to \mathcal{P}\left( (X_{\K[\Lalg]})^2\right)$ that will show $\mathbf{K[L]}$ to be representable.  For the sake of readability we let $R\subseteq X_{\koflbf}
\times X_{\koflbf}$ denote the following binary relation:
\begin{align}  
    R:= &\,\bigcup \left\{{}^{[z]}{\leqslant_{X_{\K}}} \;\Big\vert \; [z]\in  X_{\Lalg}/E_{\Lalg}\right\}\cup
 \bigcup \left\{{}_{[z]}{\leqslant_{X_{\K}}} \;\Big\vert \; [z]\in  X_{\Lalg}/E_{\Lalg}\right\}\cup\nonumber\\
&\,\bigcup \left\{[z]\times {}^{[z]}\!X_{\K} \;\Big\vert \; [z]\in  X_{\Lalg}/E_{\Lalg}\right\}\cup
        \bigcup \left\{{}_{[z]}\!X_{\K}\times [z] \;\Big\vert \; [z]\in  X_{\Lalg}/E_{\Lalg}\right\}\cup\nonumber\\
        &\,\bigcup \left\{{}_{[z]}\!X_{\K}\times {}^{[z]}\!X_{\K} \;\Big\vert \; [z]\in  X_{\Lalg}/E_{\Lalg}\right\}.\label{Equation:RelationRDef}
    \end{align}
In fact, $R = {{\leqslant}_{X_{\mathbf{K[L]}}}} {\setminus} {\leqslant_{X_{\mathbf{L}}}}$. Now, for $m\in K\backslash\{a_0\}\cup L$, define 
\begin{numcases}{\psi(m)=}
	E_{\K[\Lalg]}&if $m=a_1$	\nonumber\\	
	R\cup
        \varphi_{\Lalg}(m),&
	if $m\in L$,\label{PsiDef}\\	
	\varnothing,
	&if $m=a_{-1}$ \nonumber	
\end{numcases}

Before setting out to prove that this construction provides a
representation for the nested sum of $\mathbf{K}$
and $\mathbf{L}$, we consider an example that demonstrates how
the representation is constructed.\\

\begin{example} \label{Example:repsKofL}
Let $\mathbf{K}=\mathbf{S}_3$ and $\mathbf{L}=\mathbf{L}_1$ {\upshape(}the first diamond in Figure~\ref{fig:qRAs-examples}{\upshape)}. 
Then, as per Example~\ref{Example:KofLConstruction}, the six-element algebra from Figure~\ref{Figure:GOS1} is the nested sum of $\mathbf{K}$ and $\mathbf{L}$.
The posets used to represent 
$\mathbf{K}$ and $\mathbf{L}$, as well
as their representations, are given in Figures~\ref{fig:S3Representation} 
and~\ref{fig:DiamondRepresentation},
respectively {\upshape(}their representations are also depicted at the top of Figure~\ref{Fig:KofLRepresentation}{\upshape)}.  The poset $\mathbf{X}_{\koflbf}$ is constructed
from these posets and depicted in
Figures~\ref{fig:XofKofL} and~\ref{Fig:EquivAlphaBetaXofKofL}.
Recall that $R$ is defined as in~(\ref{Equation:RelationRDef}). 
Now, using the labels for elements of
$\koflbf$ from Figure~\ref{Figure:GOS1}, we have that
\begin{align*}
    \psi(\bot)&=\varnothing\\    \psi(0)&=R={\leqslant_{X_{\koflbf}}}{\setminus}{\leqslant_{X_\Lalg}}\\
    \psi(a)&=R\cup \{(u,u)\}={\leqslant_{X_{\koflbf}}}{\setminus}\{(v,v)\}\\
    \psi(b)&=R\cup \{(v,v)\}={\leqslant_{X_{\koflbf}}}{\setminus}\{(u,u)\}, \text{ and }\\
    \psi(1)&=R\,\cup \leqslant_{X_\Lalg}\, = {\leqslant_{X_{\koflbf}}}\\
    \psi(\top)&=E_{\koflbf}=
    \bigcup_{z\in\{u, v\}}\left([z]\cup\{x^{[z]}, y^{[z]}\}\cup \{x_{[z]}, y_{[z]}\}\right)^2
\end{align*}
The set of
binary relations $\{\psi(\bot),\psi(0),\psi(a),\psi(b),\psi(1),\psi(\top)\}$ forms
a subuniverse of  
$\langle \mathsf{Up}( E_{\koflbf},\preccurlyeq_{E_{\K[\Lalg]}}),\cap, \cup, \mathbin{;}, {\sim}, {-}, \neg, {\leqslant_{X_{\koflbf}}}, R\rangle$
and hence $\koflbf$ is representable, as depicted in Figure~\ref{Fig:KofLRepresentation}.\\
\end{example}

\begin{figure}[ht]
\centering
\begin{tikzpicture}[scale=0.85]
	
\begin{scope}[xshift=-3cm]
\node[draw,circle,inner sep=1.5pt,fill] (bot) at (0,0) {};
\node[draw,circle,inner sep=1.5pt,fill] (id) at (0,1) {};
\node[draw,circle,inner sep=1.5pt,fill] (top) at (0,2) {};
\draw[order] (bot)--(id)--(top);
\node[label,anchor=east,xshift=1pt] at (top) {$\{x, y\}^2$};
\node[label,anchor=east,xshift=1pt] at (id) {$\{(x,x), (y, y)\}$};
\node[label,anchor=east,xshift=1pt] at (bot) {$\varnothing$};
\end{scope}

\begin{scope}[xshift=2.5cm]
\node[draw,circle,inner sep=1.5pt,fill] (0) at (0,0) {};
\node[draw,circle,inner sep=1.5pt,fill] (a) at (-1,1) {};
\node[draw,circle,inner sep=1.5pt,fill] (b) at (1,1) {};
\node[draw,circle,inner sep=1.5pt,fill] (1) at (0,2) {};
\draw[order] (0)--(a)--(1);
\draw[order] (0)--(b)--(1);
\node[label,anchor=south] at (1) {$\{(u,u), (v,v)\}$};
\node[label,anchor=east,xshift=1pt] at (a) {$\left\{\left(u, u\right)\right\}$};

\node[label,anchor=west,xshift=-1pt] at (b) {$\left\{\left(v, v\right)\right\}$};
\node[label,anchor=north] at (0) {$\varnothing$};
\end{scope}

\begin{scope}[yshift=-6cm]
\node[draw,circle,inner sep=1.5pt,fill] (a1) at (1,0) {};
\node[draw,circle,inner sep=1.5pt,fill] (a2) at (1,4) {};
\node[draw,circle,inner sep=1.5pt,fill] (b0) at (1,1) {};
\node[draw,circle,inner sep=1.5pt,fill] (b1) at (1,3) {};
\node[draw,circle,inner sep=1.5pt,fill] (b2) at (0,2) {};
\node[draw,circle,inner sep=1.5pt,fill] (b3) at (2,2) {};
\draw[order] (b0)--(b2);
\draw[order] (b0)--(b3);
\draw[order] (b2)--(b1);
\draw[order] (b3)--(b1);
\draw[order] (a1)--(b0);
\draw[order] (b1)--(a2);
\node[label,anchor=east,xshift=1pt] at (b2) {{\small ${\leqslant_{X_{\K[\Lalg]}}}{\setminus} \{(v, v)\}$}};
\node[label,anchor=east,yshift=-1pt] at (b0) {{\small $R={\leqslant_{X_{\K[\Lalg]}}}{\setminus} \leqslant_{X_\Lalg}$}};
\node[label,anchor=west] at (b1) {{\small $\leqslant_{X_{\K[\Lalg]}}$}};
\node[label,anchor=west,xshift=-1pt] at (b3) {{\small ${\leqslant_{X_{\K[\Lalg]}}}{\setminus} \{(u,u)\}$}};
\node[label,anchor=east,xshift=1pt] at (a1) {$\varnothing$};
\node[label,anchor=west,yshift=5pt] at (a2) {{\small $\displaystyle\bigcup_{z\in\{u, v\}}\left([z]\cup\{x^{[z]}, y^{[z]}\}\cup \{x_{[z]}, y_{[z]}\}\right)^2$}};
\end{scope}

\end{tikzpicture}
\caption{The representations of $\K$, $\Lalg$ and $\koflbf$ from Example~\ref{Example:repsKofL}.}
\label{Fig:KofLRepresentation}
\end{figure}

Having given an example of how the new representation will work, 
we now prove that $\psi$ does indeed map elements from 
$K\backslash\{a_0\}\cup L$ to elements in $\mathsf{Up}\left(E_{\K[\Lalg]}, \preccurlyeq_{E_{\K[\Lalg]}}\right)$.

\vskip 0.4cm 

\begin{lemma}\label{Lem:PsiUpsets}
	For all $m\in K\backslash\{a_0\}\cup L$,
	the set $\psi(m)$ as defined in {\upshape (\ref{PsiDef})}
	is an up-set of $\left(E_{\koflbf}, \preccurlyeq_{E_{\koflbf}}\right)$.
\end{lemma}
\begin{proof}
	It is immediate that $\varnothing$ and $E_{\koflbf}$
	are up-sets of $\left(E_{\koflbf}, \preccurlyeq_{E_{\koflbf}}\right)$.
	Thus, we only need to prove that $\psi(m)$ is an up-set of
	$\left(E_{\koflbf}, \preccurlyeq _{E_{\koflbf}}\right)$ for $m\in L$.
	Let $\left(u_1,v_1\right) \in \psi(m)$
	and assume 
	$\left(u_1,v_1\right)
	\preccurlyeq_{E_{\koflbf}} 
	\left(u_2,v_2\right)$.
        That is, by~(\ref{EIneq}), we have that
        $u_2\leqslant_{X_{\koflbf}} u_1$ and $v_1\leqslant_{X_{\koflbf}} v_2$.
	There are six cases to consider.

        \vskip 0.2cm
        \noindent
	\underline{Case 1:} If $(u_1,v_1)\in {}^{[z]}{\leqslant_{X_{\K}}}$ for some
        $[z]\in X_{\Lalg}/E_{\Lalg}$, then there exists $w_1,w_2\in X_{\K}$
        such that $u_1=\left(w_1\right)^{[z]}$, $v_1=\left(w_2\right)^{[z]}$ and $w_1\leqslant_{X_{\K}} w_2$.
        Moreover, $v_1\leqslant_{X_{\koflbf}} v_2$ implies that
        $v_2\in {}^{[z]}\!X_{\K}$ and there exists $w_3\in X_{\K}$
        such that $v_2 = \left(w_3\right)^{[z]}$ and $w_2\leqslant_{X_{\K}} w_3$.  On the other hand,
        $u_2\leqslant_{X_{\koflbf}} u_1$ gives rise to three subcases:
        \begin{itemize}
            \item If $u_2\in[z]\subseteq X_{\Lalg}$, 
            then $(u_2,v_2)\in [z]\times {}^{[z]}\!X_{\K}\subseteq \psi(m)$.
            \item If $u_2\in {}_{[z]}\!X_{\K}$, then $(u_2,v_2)\in {}_{[z]}\!X_{\K}\times
            {}^{[z]}\!X_{\K}\subseteq \psi(m)$.
            \item If $u_2\in {}^{[z]}\!X_{\K}$, then there exists $w\in X_{\K}$
            such that $u_2=w^{[z]}$.  Since $u_2\leqslant_{X_{\koflbf}} u_1$, it
            follows that $w\leqslant_{X_{\K}} w_1$ and we have that
            $w\leqslant_{X_{\K}} w_1\leqslant_{X_{\K}} w_2\leqslant_{X_{\K}} w_3$.
            This in turn implies that $(u_2,v_2)\in {{}^{[z]}{\leqslant_{X_{\K}}}}\subseteq \psi(m)$.
        \end{itemize}

        \vskip 0.2cm
        \noindent
        \underline{Case 2:}
        If $(u_1,v_1)\in {}_{[z]}{\leqslant_{X_\K}}$ for some
        $[z]\in X_{\Lalg}/E_{\Lalg}$, then there exists $w_1,w_2\in X_{\K}$
        such that $u_1=\left(w_1\right)_{[z]}$, $v_1=\left(w_2\right)_{[z]}$ and $w_1\leqslant_{X_\K} w_2$.
        Moreover, $u_2\leqslant_{X_{\koflbf}} u_1$ implies that $u_2\in {}_{[z]}\!X_{\K}$
        and there exists $w\in X_{\K}$ such that $u_2=w_{[z]}$ and $w\leqslant_{X_{\K}} w_1$.
        On the other hand, $v_1\leqslant_{X_{\koflbf}}v_2$ gives rise to three subcases:
        \begin{itemize}
            \item If $v_2\in [z]\subseteq X_{\Lalg}$, 
            then $(u_2,v_2)\in {}_{[z]}\!X_{\koflbf}\times [z]\subseteq \psi(m)$.
            \item If $v_2\in {}^{[z]}\!X_{\K}$, then $(u_2,v_2)\in {}_{[z]}\!X_{\K}\times
            {}^{[z]}\!X_{\K}\subseteq \psi(m)$.
            \item If $v_2\in {}_{[z]}\!X_{\K}$, then there exists $w_3\in X_{\K}$
            such that $v_2=(w_3)_{[z]}$.  Since we have $v_1\leqslant_{X_{\koflbf}} v_2$,
            it follows that $w_2\leqslant_{X_{\K}} w_3$ and we get
            $w\leqslant_{X_{\K}} w_1\leqslant_{X_{\K}} w_2\leqslant_{X_{\K}} w_3$.
            This in turn implies that $(u_2,v_2)\in {}_{[z]}{\leqslant_{X_{\K}}}\subseteq \psi(m)$.
        \end{itemize}

        \vskip 0.2cm
        \noindent
        \underline{Case 3:}
        If $(u_1,v_1)\in [z]\times {}^{[z]}\!X_{\K}$ for some
        $[z]\in X_{\Lalg}/E_{\Lalg}$, then $v_1\leqslant_{X_{\koflbf}} v_2$
        implies that $v_2\in {}^{[z]}\!X_{\K}$.  On the other hand,
        $u_2\leqslant_{X_{\koflbf}} u_1$ gives rise to two subcases:
        \begin{itemize}
            \item If $u_2\in [z]$, then $(u_2,v_2)\in [z]\times {}^{[z]}\!X_{\K}\subseteq \psi(m)$.
            \item If $u_2\in {}_{[z]}\!X_{\K}$, then $(u_2,v_2)\in {}_{[z]}\!X_{\K}\times
            {}^{[z]}\!X_{\K}\subseteq \psi(m)$.
        \end{itemize}

        \vskip 0.2cm
        \noindent
        \underline{Case 4:}
        If $(u_1,v_1)\in {}_{[z]}\!X_{\K}\times [z]$ for some
        $[z]\in X_{\Lalg}/E_{\Lalg}$, then $u_2\leqslant_{X_{\koflbf}} u_1$
        implies that $u_2\in {}_{[z]}\!X_{\K}$.  On the other hand,
        $v_1\leqslant_{X_{\koflbf}} v_2$ gives rise to two subcases:
        \begin{itemize}
            \item If $v_2\in [z]$, then $(u_2,v_2)\in {}_{[z]}\!X_{\K}\times [z]\subseteq \psi(m)$.
            \item If $v_2\in {}^{[z]}\!X_{\K}$, then $(u_2,v_2)\in {}_{[z]}\!X_{\K}\times
            {}^{[z]}\!X_{\K}\subseteq \psi(m)$.
        \end{itemize}

        \vskip 0.2cm
        \noindent
        \underline{Case 5:}
        If $(u_1,v_1)\in {}_{[z]}\!X_{\K}\times {}^{[z]}\!X_{\K}$ for some
        $[z]\in X_{\Lalg}/E_{\Lalg}$, then $u_2\leqslant_{X_{\koflbf}} u_1$
        implies that $u_2\in {}_{[z]}\!X_{\K}$ and $v_1\leqslant_{X_{\koflbf}} v_2$
        implies that $v_2\in {}^{[z]}\!X_{\K}$. Therefore, it follows that
        $(u_2,v_2)\in {}_{[z]}\!X_{\K}\times
            {}^{[z]}\!X_{\K}\subseteq \psi(m)$.

        \vskip 0.2cm
        \noindent
        \underline{Case 6:}
        If $(u_1,v_1)\in \varphi_{\Lalg}(m)$, then $u_1,v_1\in X_{\Lalg}$ and $\left(u_1, v_1\right) \in E_{\Lalg}$. The inequalities $u_2\leqslant_{X_{\koflbf}} u_1$ and
        $v_1\leqslant_{X_{\koflbf}} v_2$ combined give rise to four
        subcases:
        \begin{itemize}
            \item If $u_2\in {}_{[u_1]}\!X_{\K}$ and $v_2\in {}^{[u_1]}\!X_{\K}$,
            then $(u_2,v_2)\in {}_{[u_1]}\!X_{\K}\times{}^{[u_1]}\!X_{\K}\subseteq \psi(m)$.
            \item If $u_2\in {}_{[u_1]}\!X_{\K}$ and $v_2\in [u_1]$, then
             $(u_2,v_2)\in {}_{[u_1]}\!X_{\K}\times [u_1]\subseteq \psi(m)$.
            \item If $u_2\in [u_1]$ and $v_2\in {}^{[u_1]}\!X_{\K}$, then 
            $(u_2,v_2)\in [u_1]\times ^{[u_1]}\!X_{\K}\subseteq \psi(m)$.
            \item If $u_2\in [u_1]$ and $v_2\in [u_1]$, 
            then $u_2\leqslant_{X_{\Lalg}} u_1$ and
        $v_1\leqslant_{X_{\Lalg}} v_2$, which means $\left(u_1, v_1\right) \preccurlyeq_{E_{\Lalg}}\left(u_2, v_2\right)$.
            But $\varphi_{\Lalg}(m)$ is an upset of $(E_{\Lalg},\preccurlyeq_{E_\Lalg})$, so $(u_1,v_1)\in \varphi_{\Lalg}(m)$ implies
            $(u_2,v_2)\in \varphi_{\Lalg}(m)\subseteq \psi(m)$.\qedhere
        \end{itemize}
        \end{proof}

We now prove a series of lemmas that, when combined, show that the 
map $\psi$ defined in (\ref{PsiDef}) is a DqRA embedding.

\vskip 0.4cm 

\begin{lemma} \label{Lem:PsiInject}
The map $\psi:K\backslash\{a_0\}\cup L\to \mathsf{Up}\left(E_{\koflbf}, \preccurlyeq_{E_{\koflbf}}\right)$ is
injective and $\psi(1_{\koflbf})= {\leqslant_{X_{\koflbf}}}$.
\end{lemma}
\begin{proof}
	
	Since $\varphi_{\Lalg}(1_{\Lalg})={\leqslant_{X_\Lalg}}$, 
	it follows directly from the definitions of $\psi(m)$ and 
        $\leqslant_{X_{\koflbf}}$ in (\ref{PsiDef}) and (\ref{Def:IneqKofL}), respectively
        that $\psi(1_{\koflbf})= {\leqslant_{X_{\koflbf}}}$.
	
	To show that $\psi$ is injective, let $m_1,m_2\in K{\setminus}\{a_0\}\cup L$
	such that $m_1\neq m_2$. Without loss of generality we can consider four
	cases.	

        \vskip 0.2cm
	\noindent\underline{Case 1}: If $m_1,m_2\in L$, then $\varphi_{\Lalg}(m_1)\neq\varphi_{\Lalg}(m_2)$ since $\varphi_{\Lalg}$ is injective
	and therefore we get $\psi(m_1)\neq\psi(m_2)$.
	
        \vskip 0.2cm
	\noindent\underline{Case 2}: If $m_1=a_{-1}$ and $m_2\in L$, then $\psi(m_1)=\varnothing$.
        On the other hand, if we let $w\in X_{\K}$ and $z\in X_{\Lalg}$, then $(w_{[z]},z)\in\psi(m_2)$ which shows $\psi(m_2)\neq\varnothing$ and hence $\psi(m_1)\neq\psi(m_2)$.

        \vskip 0.2cm
	\noindent\underline{Case 3}:  If $m_1=a_{1}$ and $m_2\in L$, then $\psi(m_1)=E_{\koflbf}$.
        Let $w\in X_{\K}$ and $z\in X_{\Lalg}$, then $(z,w_{[z]})\in E_{\koflbf}=\psi(m_1)$,
        but by (\ref{PsiDef}) we have $\left(z,w_{[z]}\right)\notin\psi(m_2)$.  Hence, it follows that $\psi(m_1)\neq\psi(m_2)$. 
 
	\vskip 0.2cm
	\noindent\underline{Case 4}: If $m_1,m_2\in K{\setminus}\{a_0\}$, then we may         assume, without loss of generality, that $m_1=a_1$ and $m_2=a_{-1}$. Then 
        $\psi(m_1)\neq \varnothing$ since $E_{\koflbf}\neq\varnothing$ and therefore we get 
	$\psi(m_1)\neq\psi(m_2)$.		
\end{proof}

In order to prove that $\psi$ preserves the linear negations and 
the involution, we will need the following auxiliary result.

\vskip 0.4cm 

\begin{lemma}\label{Lem:EComp}
	Let $E_{\koflbf}$ be the relation defined in {\upshape (\ref{EKofLDef})} and 
	$\alpha_{\koflbf}$ and $\beta_{\koflbf}$ the maps defined in 
    {\upshape(\ref{Def:AlphaKofL})} and 
{\upshape (\ref{Def:BetaKofL})}, respectively. Then, for 
	$\gamma\in\{\alpha_{\koflbf},\beta_{\koflbf}\}$, we have
	\[
		\gamma\mathbin{;}\varnothing=\varnothing,\quad 
		\varnothing\mathbin{;}\gamma=\varnothing,\quad
		\gamma\mathbin{;}E_{\koflbf}=E_{\koflbf},\quad\text{and}\quad
		E_{\koflbf}\mathbin{;}\gamma=E_{\koflbf}.
	\]
\end{lemma}
\begin{proof}
It follows from $\gamma\subseteq E_{\koflbf}$ and the transitivity of $E_{\koflbf}$ that
$\gamma\mathbin{;}E_{\koflbf}\subseteq E_{\koflbf}$.
For the inclusion in the other direction, let $(u,v)\in E_{\koflbf}$. Then $(v,u)\in E_{\koflbf}$ by symmetry and, since 
$(u,\gamma(u))\in\gamma\subseteq E_{\koflbf}$, we have that $(v,\gamma(u))\in E_{\koflbf}$ by transitivity.
This means $(\gamma(u), v) \in E_{\koflbf}$, and therefore $(u, v) \in \gamma \mathbin{;} E_{\koflbf}$. 

The first two equalities are trivially true and the final equality can be proved similarly to the above.
\end{proof}

\begin{lemma}\label{Lem:PresLinNeg}
	The map $\psi:K{\setminus}\{a_0\}\cup L\to \mathsf{Up}\left(E_{\koflbf}, \preccurlyeq_{E_{\koflbf}}\right)$
	preserves the linear negations $\sim_{\koflbf}$ and ${-}_{\koflbf}$, i.e., for 
	$m\in K\backslash\{a_0\}\cup L$,
	\[\psi(-_{\koflbf}m)=-\psi(m)\quad\text{and}\quad\psi({\sim}_{\koflbf} m)={\sim}\psi(m).\]
\end{lemma}
\begin{proof}
	We prove the claim for $-_{\koflbf}$.  The result for $\sim_{\koflbf}$ can be
	proved analogously.  If $m\in K{\setminus}\{a_0\}\cup L$,
        then $m\in K{\setminus}\{a_0\}$ or $m\in L$.

        \vskip 0.2cm
	\noindent\underline{Case 1}:
	If $m\in K{\setminus}\{a_0\}$, then $m=a_{-1}$ and $\psi(m)=\varnothing$, or $m=a_1$ and $\psi(m)=E_{\koflbf}$.
	Note that $E_{\koflbf}^c=\varnothing$ hence $(E_{\koflbf}^c)^\smile=\varnothing^\smile=\varnothing$. 
	Also $\varnothing^c=E_{\koflbf}$ hence
	$(\varnothing^c)^\smile=E_{\koflbf}^\smile=E_{\koflbf}$ since
	$E_{\koflbf}$ is symmetric.
	By Lemmas~\ref{Lemma:NegationBehaviour} and~\ref{Lem:EComp}
    we have	
\begin{align*}
\psi(-_{\koflbf}a_1)=\psi(-_{\K}a_1)=\psi(a_{-1})=\varnothing
		=\alpha_{\koflbf}\mathbin{;} \varnothing
		&=\alpha_{\koflbf}\mathbin{;}(E_{\koflbf}^c)^\smile\\&=-E_{\koflbf}=-\psi(a_1)
	\end{align*}
	and
	\begin{align*}
		\psi(-_{\koflbf}a_{-1})=\psi(-_{\K}a_{-1})=\psi(a_1)=E_{\koflbf}
		= \alpha_{\koflbf}\mathbin{;}E_{\koflbf}
		&=\alpha_{\koflbf}\mathbin{;}(\varnothing^c)^\smile\\
		& ={-}\varnothing=-\psi(a_{-1}).
	\end{align*}



        \vskip 0.2cm
	\noindent\underline{Case 2}:
	Let $m \in L$. To prove that 
        $\psi({-}_{\koflbf}m)=-\psi(m)$ we first
        show that $\psi({-}_{\koflbf}m)\subseteq-\psi(m)$.
        Let $(u,v)\in\psi({-}_{\koflbf} m)$.
        Then, by~(\ref{PsiDef}) there are six sub-cases to consider.
        \begin{itemize}
            \item
            Let $(u,v)\in {}^{[z]}{\leqslant_{X_{\K}}}$
            for some $[z]\in X_{\Lalg}/E_{\Lalg}$.
            Since $X_{\K}$ is a two-element antichain this
            implies that $u=v$ and there exists $w_1\in X_{\K}$
            such that $u=(w_1)^{[z]}=v$. Furthermore,
            $\alpha_{\koflbf}(u)=\left(\alpha_{\K}(w_1)\right)^{[z]}
            =\left(w_2\right)^{[z]}$ where $w_2\in X_{\K}$
            such that $w_2\neq w_1$. Then,
\begin{eqnarray*}
& &\,(w_1,w_2) \in E_\K \textnormal{ and } (w_1,w_2)\notin {\leqslant_{X_{\K}}}\\
& \Rightarrow &
\left(\left(w_1\right)^{[z]},
\left(w_2\right)^{[z]}\right) \in E_{\koflbf}
                \textnormal{ and } \left(\left(w_1\right)^{[z]},
\left(w_2\right)^{[z]}\right)\notin {}^{[z]}{\leqslant_{X_{\K}}}\\
& \Rightarrow & \left((w_1)^{[z]},
\left(\alpha_{\K}\left(w_1\right)\right)^{[z]}\right) \in E_{\koflbf} \textnormal{ and }  \left((w_1)^{[z]}, \left(\alpha_{\K}\left(w_1\right)\right)^{[z]}\right)\notin \psi(m)\\
                & \Rightarrow&
                \left(v, \alpha_{\koflbf}(u)\right)\in \psi(m)^c\\
                & \Rightarrow&
                \left( \alpha_{\koflbf}(u), v\right)\in \left(\psi(m)^c\right)^{\smallsmile}\\
& \Rightarrow &  \,(u,v)\in\alpha_{\koflbf}\mathbin{;}\left(\psi(m)^c\right)^{\smile} =-\psi(m).
            \end{eqnarray*}

            \item Let $(u,v)\in {}_{[z]}{\leqslant_{X_{\K}}}$
            for some $[z]\in X_{\Lalg}/E_{\Lalg}$. Then the proof
            that $(u,v)\in -\psi(m)$ is similar to the proof of 
            the previous subcase.

            \item  Let $(u,v)\in [z]\times {}^{[z]}\!X_{\K}$
            for some $[z]\in X_{\Lalg}/E_{\Lalg}$. Then
            $\alpha_{\koflbf}(u)=\alpha_{\Lalg}(u)$ and
            $(u,\alpha_{\Lalg}(u))\in \alpha_{\Lalg}\subseteq
            \alpha_{\koflbf}\subseteq E_{\koflbf}$.  That is,
            $\alpha_{\Lalg}(u)\in [z]$.  Hence,
            \begin{align*}
                \left(v,\alpha_{\Lalg}(u)\right) \in E_{\koflbf} \textnormal{ and } \left(v,\alpha_{\Lalg}(u)\right)\notin \psi(m)
                &\Rightarrow
                \left(v,\alpha_{\koflbf}(u)\right)\in\psi(m)^c\\
                &\Rightarrow
                \left(\alpha_{\koflbf}(u),v\right)\in\left(\psi(m)^c\right)^{\smile}\\
                &\Rightarrow
                (u,v)\in\alpha_{\koflbf}\mathbin{;}\left(\psi(m)^c\right)^{\smile}
                =-\psi(m).
            \end{align*}

            \item  Let $(u,v)\in {}_{[z]}\!X_{\K}\times [z]$
            for some $[z]\in X_{\Lalg}/E_{\Lalg}$. Then there
            exists $w\in X_{\K}$ such that $u=w_{[z]}$ and
            $\alpha_{\koflbf}(u)=\left(\alpha_{\K}(w)\right)_{[z]}$.
            Hence, 
            \begin{eqnarray*}
                & &(v,\left(\alpha_{\K}(w)\right)_{[z]}) \in E_{\koflbf} \textnormal{ and } (v,\left(\alpha_{\K}(w)\right)_{[z]})\notin
                \psi(m)\\
                & \Rightarrow &
                (v,\alpha_{\koflbf}(u))\in \psi(m)^c\\
                & \Rightarrow &
                (\alpha_{\koflbf}(u),v)\in \left(\psi(m)^c\right)^{\smile}\\
                & \Rightarrow &
                (u,v)\in\alpha_{\koflbf}\mathbin{;}\left(\psi(m)^c\right)^{\smile}
                =-\psi(m).
            \end{eqnarray*}

            \item Let $(u,v)\in {}_{[z]}\!X_{\K}\times {}^{[z]}\!X_{\K}$
            for some $[z]\in X_{\Lalg}/E_{\Lalg}$.
            The proof
            that $(u,v)\in -\psi(m)$ is similar to the proof of 
            the previous subcase.

            \item Let $(u,v)\in \varphi_{\Lalg}({-}_{\koflbf} m)
            =\varphi_{\Lalg}({-}_{\Lalg}m)$.  Since $\varphi_{\Lalg}$
            is an embedding that preserves the linear negations, 
            it follows that 
            \begin{eqnarray*}
            & & (u,v)\in{-}\varphi_{\Lalg}(m)
            =\alpha_{\Lalg}\mathbin{;}\left(\varphi_{\Lalg}(m)^c\right)^{\smile} \\
            & \Rightarrow &
            (\alpha_{\Lalg}(u),v)\in \left(\varphi_{\Lalg}(m)^c\right)^{\smile} \\
            & \Rightarrow &
            (v,\alpha_{\Lalg}(u))\in\varphi_{\Lalg}(m)^c\\
            & \Rightarrow &
            (v,\alpha_{\Lalg}(u)) \in E_{\Lalg} \textnormal{ and }(v,\alpha_{\Lalg}(u))\notin\varphi_{\Lalg}(m) \\
            & \Rightarrow &
            (v,\alpha_{\Lalg}(u)) \in E_{\koflbf} \textnormal{ and } (v,\alpha_{\koflbf}(u))\notin\psi(m)\\
            & \Rightarrow &
            (v,\alpha_{\koflbf}(u))\in \psi(m)^c \\
            & \Rightarrow &
            (\alpha_{\koflbf}(u),v)\in\left(\psi(m)^c\right)^{\smile}\\
            & \Rightarrow &
            (u,v)\in\alpha_{\koflbf}\mathbin{;}\left(\psi(m)^c\right)^{\smile}
            =-\psi(m).
            \end{eqnarray*}
        \end{itemize}
        
	For the inclusion in the other direction, let 
	$(u,v)\in-\psi(m)=\alpha_{\koflbf}\mathbin{;}\left(\psi(m)^c\right)^{\smile}$.
    Then,
        \begin{align*}
        (u,v)\in \alpha_{\koflbf}\mathbin{;}\left(\psi(m)^c\right)^\smile 
        &\Rightarrow
        (\alpha_{\koflbf}(u),v)\in \left(\psi(m)^c\right)^{\smile}\\
         &\Rightarrow
         (v,\alpha_{\koflbf}(u))\in \psi(m)^c\\
         &\Rightarrow
         \left(v,\alpha_{\koflbf}(u)\right) \in E_{\koflbf} \textnormal{ and } \left(v,\alpha_{\koflbf}(u)\right)\notin
         \psi(m).
    \end{align*} 
    Since $\alpha_{\koflbf}\subseteq
    E_{\koflbf}$ and $E_{\koflbf}$ is symmetric and transitive, 
    we get $(u, v) \in E_{\koflbf}$. Hence, 
    there exists $[z]\in X_{\Lalg}/E_{\Lalg}$
    such that $(u,v)\in \left([z]\cup {}^{[z]}\!X_{\K}\cup
    {}_{[z]}\!X_{\K}\right)^2$ by~(\ref{EKofLDef}).
    
    Suppose $(u,v)\in {}^{[z]}\!X_{\K}\times [z]$.
    Then $v\in[z]$ and there exists $w\in X_{\K}$ such
    that $u=w^{[z]}$ and $\alpha_{\koflbf}(u) =\left(\alpha_{\K}(w)\right)^{[z]}\in{}^{[z]}\!X_{\K}$.  But then
    $\left(v,\alpha_{\koflbf}(u)\right)\notin
         \psi(m)$ means $\left(v,\left(\alpha_{\K}(w)\right)^{[z]}\right)\notin \psi(m)$,
    which contradicts the fact that $[z]\times {}^{[z]}\!X_{\K}
    \subseteq \psi(m)$ for any $m\in L$ as per
    the definition in~(\ref{PsiDef}).  Hence, $-\psi(m)$
    does not contain any elements of ${}^{[z]}\!X_{\K}\times [z]$.
    Similarly, we can show that $-\psi(m)$ also
    does not contain any elements of ${}^{[z]}\!X_{\K}\times {}_{[z]}\!X_{\K}$ nor of $[z]\times {}_{[z]}\!X_{\K}$.

    If $(u,v)\in [z]\times {}^{[z]}\!X_{\K}$
    or $(u,v)\in {}_{[z]}\!X_{\K}\times [z]$
    or $(u,v)\in {}_{[z]}\!X_{\K}\times {}^{[z]}\!X_{\K}$,
    it follows immediately from the definition of
    $\psi$ in~(\ref{PsiDef}) that $(u,v)\in\psi({-}_{\koflbf}m)$.

    There are three subcases left to consider.
    \begin{itemize}
        \item Let $(u,v)\in {}^{[z]}\!X_{\K}\times {}^{[z]}\!X_{\K}$.
        Then there exist $w_1,w_2\in X_{\K}$ such that 
        $u=(w_1)^{[z]}$ and $v=(w_2)^{[z]}$.  Since $X_{\K}$
        is a two-element antichain, either $w_1\neq w_2$ and
        $\alpha_{\koflbf}(u)=\left(\alpha_{\K}(w_1)\right)^{[z]}=(w_2)^{[z]}=v$, or $w_1=w_2$.
        If $w_1\neq w_2$, then $\left(v,\alpha_{\koflbf}(u)\right)\notin
         \psi(m)$ means that $\left(v, \left(\alpha_{\K}(w_1)\right)^{[z]}\right)=((w_2)^{[z]},(w_2)^{[z]})\notin 
         \psi(m)$, which contradicts that ${{}^{[z]}{\leqslant_{X_{\K}}}}
         \subseteq \psi(m)$ for all $m\in L$.
         Therefore, it must be the case that $w_1=w_2$ and 
         $u=v$, and hence 
         $(u,v)\in {{}^{[z]}{\leqslant_{X_{\K}}}}\subseteq \psi({-}_{\koflbf}m)$.

         \item Let $(u,v)\in {}_{[z]}\!X_{\K}\times {}_{[z]}\!X_{\K}$.
         We can use a similar argument to the one in the
         previous subcase to conclude that $u=v$ and hence that
         $(u,v)\in{{}_{[z]}{\leqslant_{X_{\K}}}}\subseteq \psi({-}_{\koflbf}m)$.

         \item Let $(u,v)\in [z]\times [z]$. Then, $\left(v,\alpha_{\Lalg}(u)\right) \in E_\Lalg$, and 
            $\left(v,\alpha_{\koflbf}(u)\right)\notin
         \psi(m)$ 
         implies
         $(v,\alpha_{\Lalg}(u))\notin \varphi_{\Lalg}(m)$.
         Hence, 
         $(v,\alpha_{\Lalg}(u))\in\varphi_{\Lalg}(m)^c$,
         and so $(\alpha_{\Lalg}(u),v)\in \left(\varphi_{\Lalg}(m)^c\right)^{\smile}$,
         which gives $(u,v)\in\alpha_{\Lalg}\mathbin{;}\left(\varphi_{\Lalg}(m)^c\right)^{\smile}$.
         That is, $(u,v)\in {-}\varphi_{\Lalg}(m)= \varphi_{\Lalg}({-}_{\Lalg}m)=\varphi_{\Lalg}({-}_{\koflbf}m)
         \subseteq\psi({-}_{\koflbf} m)$ 
         by Lemma~\ref{Lemma:NegationBehaviour} and since $\varphi_{\Lalg}$ is an embedding that preserves the
         linear negations. \qedhere
         
    \end{itemize}
    	
\end{proof}

\begin{lemma}\label{Lem:PresSelfInvol}
	The map $\psi:K{\setminus}\{a_0\}\cup L\to \mathsf{Up}\left(E_{\koflbf}, \preccurlyeq_{E_{\koflbf}}\right)$
	preserves the 
involution
$\neg_{\koflbf}$, i.e., for 
	$m\in K{\setminus}\{a_0\}\cup L$, $\psi(\neg_{\koflbf}m)=\neg\psi(m)$.
\end{lemma}
\begin{proof}
Recall from Theorem~\ref{thm:Dq(E)} that
$\neg\psi(m)=\alpha_{\koflbf}\mathbin{;}\beta_{\koflbf}\mathbin{;} \psi(m)^c\mathbin{;}\beta_{\koflbf}$. 
Now let $m\in K{\setminus}\{a_0\}\cup L$.  

     \vskip 0.2cm
	\noindent\underline{Case 1}:
	If $m\in K\backslash\{a_0\}$ there are two subcases to consider:
	\begin{itemize}
	   \item If $m=a_{-1}$, then $\psi(m)=\varnothing$
        and $\neg_{\koflbf}m=\neg_{\K}m=a_1$.  Thus, $\psi(\neg_{\koflbf}m)=E_{\koflbf}$.
		By Lemma~\ref{Lem:EComp},
\begin{align*}	\neg\psi(m)=\alpha_{\koflbf}\mathbin{;}\beta_{\koflbf}\mathbin{;}\varnothing^c\mathbin{;}\beta_{\koflbf}
=\alpha_{\koflbf}\mathbin{;}\beta_{\koflbf}\mathbin{;}E_{\koflbf}\mathbin{;}\beta_{\koflbf}
&=E_{\koflbf}\\&=\psi(\neg_{\koflbf}m).
\end{align*}
			
		\item If $m=a_1$, then $\psi(m)=E_{\koflbf}$.  Also, 
        $\neg_{\K[\Lalg]}m=\neg_\K m=a_{-1}$  
		and $\psi(\neg_{\koflbf}m)=\varnothing$.
		Using Lemma~\ref{Lem:EComp} again we have
		\begin{align*}
			\neg\psi(m)&=\alpha_{\koflbf}\mathbin{;}\beta_{\koflbf}\mathbin{;}E_{\koflbf}^c\mathbin{;}\beta_{\koflbf}
			=\alpha_{\koflbf}\mathbin{;}\beta_{\koflbf}\mathbin{;}\varnothing\mathbin{;}\beta_{\koflbf}
			=\varnothing=\psi(\neg_{\koflbf}m).
		\end{align*}
    \end{itemize}

    \vskip 0.2cm
	\noindent\underline{Case 2}:
	Next assume that $m\in L$.  Then
    $\neg_{\koflbf}m=\neg_{\Lalg}m$ by definition.
    We first show that
    $\psi(\neg_{\koflbf}m)\subseteq\neg\psi(m)$.
    Let $(u,v)\in \psi(\neg_{\koflbf}m)=\psi(\neg_{\Lalg}m)$. We consider the six cases that arise
    from~(\ref{PsiDef}).
    \begin{itemize}
    \item If $(u,v)\in {}^{[z]}{\leqslant_{X_{\K}}}$
    for some $[z]\in X_{\Lalg}/E_{\Lalg}$, then
    there exists $w_1\in X_{\K}$ such that $u=(w_1)^{[z]}=v$
    since $\mathbf{X}_{\K}$ is a two-element antichain.
    Hence $\alpha_{\koflbf}(u)=\left(\alpha_{\K}(w_1)\right)^{[z]}=(w_2)^{[z]}$ where $w_2\in X_{\K}$
    such that $w_1\neq w_2$.  Furthermore, 
    $\beta_{\koflbf}(\alpha_{\koflbf}(u))=\beta_{\koflbf}\left((w_2)^{[z]}\right)
    =\left(\beta_{\K}(w_2)\right)_{[z]}=(w_2)_{[z]}$, and so
     $(u,(w_2)_{[z]})\in \alpha_{\koflbf}\mathbin{;}\beta_{\koflbf}$. Next note that since $(w_2,w_1)\notin
    {\leqslant_{X_{\K}}}$, it follows from~(\ref{PsiDef})
    that $\left((w_2)_{[z]},(w_1)_{[z]}\right)\notin
    \psi(m)$ and hence, since $\left((w_2)_{[z]},(w_1)_{[z]}\right) \in E_{\koflbf}$ we get
    $\left((w_2)_{[z]},(w_1)_{[z]}\right)\in\psi(m)^c$.
    Combined with the above, we have that
    $\left(u,(w_1)_{[z]}\right)\in \alpha_{\koflbf}\mathbin{;}\beta_{\koflbf}\mathbin{;}\psi(m)^c$.  Finally,
    since we have $\left((w_1)_{[z]},(w_1)^{[z]}\right)\in\beta_{\koflbf}$
    by~(\ref{Def:BetaKofL}), it follows that
    $(u,(w_1)^{[z]})=(u,v)\in \alpha_{\koflbf}\mathbin{;}\beta_{\koflbf}\mathbin{;}\psi(m)^c\mathbin{;}\beta_{\koflbf}=\neg\psi(m)$.

    \vskip 0.1cm
    \item If $(u,v)\in {}_{[z]}{\leqslant_{X_{\K}}}$
    for some $[z]\in X_{\Lalg}/E_{\Lalg}$, then
    the proof is similar to the previous case.

    \vskip 0.1cm
    \item If $(u,v)\in [z]\times {}^{[z]}\!X_{\K}$
    for some $[z]\in X_{\Lalg}/E_{\Lalg}$,
    then there exists $w\in X_{\K}$ such that
    $v=w^{[z]}$.  Also,
    $\alpha_{\koflbf}(u)=\alpha_{\Lalg}(u)$
    and $\beta_{\koflbf}\left(\alpha_{\Lalg}(u)\right)
    =\beta_{\Lalg}\left(\alpha_{\Lalg}(u)\right)$
    which implies that 
    $\left(u, \beta_{\Lalg}\left(\alpha_{\Lalg}(u)\right)\right)\in\alpha_{\koflbf}\mathbin{;}\beta_{\koflbf}$.
    Next, by~(\ref{PsiDef}),
    $\left(\beta_{\Lalg}\left(\alpha_{\Lalg}(u)\right),w_{[z]}\right)\notin \psi(m)$, so
    $\left(\beta_{\Lalg}\left(\alpha_{\Lalg}(u)\right),w_{[z]}\right)\in \psi(m)^c$.
    Consequently, $\left(u, w_{[z]}\right)\in \alpha_{\koflbf}\mathbin{;}\beta_{\koflbf}\mathbin{;}\psi(m)^c$.
    Since $\left(w_{[z]},w^{[z]}\right)\in
    \beta_{\koflbf}$, it follows that
    $\left(u,w^{[z]}\right)=(u,v)\in \alpha_{\koflbf}\mathbin{;}\beta_{\koflbf}\mathbin{;}\psi(m)^c\mathbin{;}\beta_{\koflbf}=\neg\psi(m)$.

    \vskip 0.1cm
    \item If $(u,v)\in {}_{[z]}X_{\K}\times [z]$
    for some $[z]\in X_{\Lalg}/E_{\Lalg}$,
    then there exists $w_1\in X_{\K}$ such that
    $u=(w_1)_{[z]}$. Hence, $\alpha_{\koflbf}(u)=
    \left(\alpha_{\K}(w_1)\right)_{[z]}=(w_2)_{[z]}$
    where $w_2\in X_{\K}$ such that $w_1\neq w_2$.
    Furthermore, $\beta_{\koflbf}((w_2)_{[z]})=\left(\beta_{\K}\left(w_2\right)\right)^{[z]}=(w_2)^{[z]}$.
    Thus, $(u,(w_2)^{[z]})\in\alpha_{\koflbf}\mathbin{;}\beta_{\koflbf}$. Next, as $\beta_{\Lalg}(v)\in X_\Lalg$ since $v\in [z]$, we have 
    $\left((w_2)^{[z]},\beta_{\Lalg}(v)\right)\notin\psi(m)$ by~(\ref{PsiDef}).
    Therefore, $\left((w_2)^{[z]},\beta_{\Lalg}(v)\right)\in\psi(m)^c$ and 
    $\left(u,\beta_{\Lalg}(v)\right)\in
    \alpha_{\koflbf}\mathbin{;}\beta_{\koflbf}\mathbin{;}\psi(m)^c$.
    Finally, since $\left(\beta_{\Lalg}(v),v\right)
    \in\beta_{\Lalg}\subseteq\beta_{\koflbf}$, it follows
    that $(u,v)\in \alpha_{\koflbf}\mathbin{;}\beta_{\koflbf}\mathbin{;}\psi(m)^c\mathbin{;}\beta_{\koflbf}=\neg\psi(m)$.

    \vskip 0.1cm
    \item If $(u,v)\in {}_{[z]}X_{\K}\times {}^{[z]}{X_{\K}}$
    for some $[z]\in X_{\Lalg}/E_{\Lalg}$,
    then there exists $w_1,w_2\in X_{\K}$ such that
    $u=(w_1)_{[z]}$ and $v=(w_2)^{[z]}$. It follows that $\alpha_{\koflbf}(u)=
    \left(\alpha_{\K}
    (w_1)\right)_{[z]}=(w_3)_{[z]}$
    where $w_3\in X_{\K}$ such that $w_1\neq w_3$.
    Also, $\beta_{\koflbf}((w_3)_{[z]})=\left(\beta_{\K}\left(w_3\right)\right)^{[z]}=(w_3)^{[z]}$.
    Hence, $(u,(w_3)^{[z]})\in\alpha_{\koflbf}\mathbin{;}\beta_{\koflbf}$. Next, 
    $\left((w_3)^{[z]},(w_2)_{[z]}\right)\notin\psi(m)$ by~(\ref{PsiDef}).
    Thus, $\left((w_3)^{[z]},(w_2)_{[z]}\right)\in\psi(m)^c$ and 
    $\left(u,(w_2)_{[z]}\right)\in
    \alpha_{\koflbf}\mathbin{;}\beta_{\koflbf}\mathbin{;}\psi(m)^c$.
    Finally, since  we have $\beta_{\koflbf}\left((w_2)_{[z]}\right)
    =\left(\beta_{\K}(w_2)\right)^{[z]}=(w_2)^{[z]}$,
    it follows that
    $\left(u,(w_2)^{[z]}\right)=(u,v)\in \alpha_{\koflbf}\mathbin{;}\beta_{\koflbf}\mathbin{;}\psi(m)^c\mathbin{;}\beta_{\koflbf}=\neg\psi(m)$.
    
    \vskip 0.1cm
    \item If $(u,v)\in \varphi_{\Lalg}(\neg_{\Lalg}m)$, then,
    since $\varphi_{\Lalg}$ is an embedding, we have
    \begin{eqnarray*}
       & & (u,v)\in\varphi_{\Lalg}(\neg_{\Lalg}m)=
        \neg\varphi_{\Lalg}(m) =\alpha_{\Lalg}\mathbin{;}\beta_{\Lalg}\mathbin{;}\varphi_{\Lalg}(m)^c\mathbin{;}\beta_{\Lalg}\\
        &\Rightarrow & (\beta_\Lalg(\alpha_\Lalg(u)), \beta_\Lalg(v)) \in E_\Lalg \textnormal{ and }
        (\beta_\Lalg(\alpha_\Lalg(u)), \beta_\Lalg(v)) \notin \varphi_\Lalg(m) \\
        & \Rightarrow & (\beta_{\koflbf}(\alpha_{\koflbf}(u)), \beta_{\koflbf}(v)) \in E_{\koflbf}\textnormal{ and }
        (\beta_{\koflbf}(\alpha_{\koflbf}(u)), \beta_{\koflbf}(v)) \notin \psi(m) \\
        & \Rightarrow & (\beta_{\koflbf}(\alpha_{\koflbf}(u)), \beta_{\koflbf}(v)) \in \psi(m)^c\\
        & \Rightarrow & (u, v) \in \alpha_{\koflbf}\mathbin{;}\beta_{\koflbf}\mathbin{;}\psi(m)^c\mathbin{;}\beta_{\koflbf}=\neg\psi(m).
    \end{eqnarray*}
    \end{itemize}
        
    For the containment in the other direction, let 
	$(u,v)\in \neg\psi(m)$.  
	Then we have $(u,v)\in\alpha_{\koflbf}\mathbin{;}\beta_{\koflbf}\mathbin{;}
    \psi(m)^c\mathbin{;}\beta_{\koflbf}$, i.e., 
    $\left(\beta_{\koflbf}(\alpha_{\koflbf}(u)),\beta_{\koflbf}(v)\right)\in\psi(m)^c$, i.e.,
    $\left(\beta_{\koflbf}(\alpha_{\koflbf}(u)),\beta_{\koflbf}(v)\right) \in E_{\koflbf}$ and $\left(\beta_{\koflbf}(\alpha_{\koflbf}(u)),\beta_{\koflbf}(v)\right)\notin\psi(m)$. Since $\alpha_{\koflbf}, \beta_{\koflbf} \subseteq E_{\koflbf}$ and $E_{\koflbf}$ is transitive, we get $(u, v) \in E_{\koflbf}$. 
    
    Now suppose $u\in {}^{[z]}\!X_{\K}$
    and $v\in[z]$ for
    some $[z]\in X_{\Lalg}/E_{\Lalg}$.
    Then there exists $w\in X_{\K}$
    such that $u=w^{[z]}$ and $\beta_{\koflbf}\left(
    \alpha_{\koflbf}(u)\right)=\left(\beta_{\K}(\alpha_{\K}
    (w))\right)_{[z]}$. Also $\beta_{\koflbf}(v)
    =\beta_{\Lalg}(v)\in[z]$, since 
    $\beta_{\Lalg}\subseteq E_{\Lalg}$.  But then
    $\left(\beta_{\koflbf}(\alpha_{\koflbf}(u)),\beta_{\koflbf}(v)\right)\in {}_{[z]}\!X_{\K}\times [z]\subseteq\psi(m)$, contradicting the previous
    statement that this pair is not a member of $\psi(m)$.
    Hence, it cannot be the case that $u\in {}^{[z]}\!X_{\K}$
    and $v\in[z]$ for
    some $[z]\in X_{\Lalg}/E_{\Lalg}$.
    Similarly we can show that it can also not be the case
    that $u\in[z]$ and $v\in{}_{[z]}\!X_{\K}$ for
    some $[z]\in X_{\Lalg}/E_{\Lalg}$
    nor that $u\in{}^{[z]}\!X_{\K}$ and 
    $v\in {}_{[z]}\!X_{\K}$for
    some $[z]\in X_{\Lalg}/E_{\Lalg}$. 
    
    On the other hand, if $u\in {}_{[z]}\!X_{\K}$ and 
    $v\in [z]$ for some $[z]\in X_{\Lalg}/E_{\Lalg}$,
    then it is immediate that
    $(u,v)\in {}_{[z]}\!X_{\K}\times [z]
    \subseteq \psi(\neg_{\koflbf}m)$ by~(\ref{PsiDef}).
    Similarly, if $u\in {}_{[z]}\!X_{\K}$ and 
    $v\in  {}^{[z]}\!X_{\K}$, or if $u\in [z]$ and 
    $v\in  {}^{[z]}\!X_{\K}$, then
    $(u,v)\in {}_{[z]}\!X_{\K}\times {}^{[z]}\!X_{\K}\subseteq \psi(\neg_{\koflbf}m)$ or, respectively,
    $(u,v)\in [z]\times {}^{[z]}\!X_{\K}\subseteq \psi(\neg_{\koflbf}m)$.
    Thus we are left 
    with three subcases to consider.
    \begin{itemize}
        \item If $u,v\in X_{\Lalg}$, then $[u]=[v]$
        since 
        $(u,v) \in E_{\koflbf}$.
        Furthermore, $\beta_{\koflbf}(\alpha_{\koflbf}(u))
        =\beta_{\Lalg}(\alpha_{\Lalg}(u))\in [u]$ and $\beta_{\koflbf}(v)=\beta_{\Lalg}(v)\in [v]$ since $\beta_{\Lalg}\subseteq\beta_{\koflbf} \subseteq E_{\koflbf}$ and 
        $\alpha_{\Lalg}\subseteq\alpha_{\koflbf}  \subseteq E_{\koflbf}$. 
        Hence,
        \begin{eqnarray*}
            & & \left(\beta_{\koflbf}(\alpha_{\koflbf}(u)),\beta_{\koflbf}(v)\right)\notin\psi(m)\\
            & \Rightarrow &
            \left(\beta_{\Lalg}(\alpha_{\Lalg}(u)),\beta_{\Lalg}(v)\right)\notin\varphi_{\Lalg}(m)\\
            & \Rightarrow &
            \left(\beta_{\Lalg}(\alpha_{\Lalg}(u)),\beta_{\Lalg}(v)\right)\in\varphi_{\Lalg}(m)^c\\
            & \Rightarrow &
            (u,v)\in\alpha_{\Lalg}\mathbin{;}\beta_{\Lalg}\mathbin{;}
            \varphi_{\Lalg}(m)^c\mathbin{;}\beta_{\Lalg}\\
            & \Rightarrow &
            (u,v)\in\neg\varphi_{\Lalg}(m)=\varphi_{\Lalg}(\neg_{\Lalg}m)\subseteq\psi(\neg_{\koflbf}m),
        \end{eqnarray*}
since $\varphi_{\Lalg}$ is an embedding that preserves the 
        involution.

        \vskip 0.1cm
\item If $u,v\in {}_{[z]}\!X_{\K}$, then
there exist $w_1,w_2\in X_{\K}$ such that
        $u=(w_1)_{[z]}$ and $v=(w_2)_{[z]}$.
        If $\alpha_{\K}(w_1)=w_3$
        where $w_3\in X_{\K}$
        such that $w_1\neq w_3$, then $\beta_{\koflbf}(\alpha_{\koflbf}(u))=
        \beta_{\koflbf}(\alpha_{\K}(w_1)_{[z]})=(\beta_{\K}(w_3))^{[z]}=(w_3)^{[z]}$. 
        Also, $\beta_{\koflbf}(v)=(\beta_{\K}(w_2))^{[z]}=(w_2)^{[z]}$. Hence,
        $\left(\beta_{\koflbf}(\alpha_{\koflbf}(u)),\beta_{\koflbf}(v)\right) \notin \psi(m)$ implies $\left((w_3)^{[z]},(w_2)^{[z]}\right)\notin\psi(m)$, which means 
        $w_3\neq w_2$.  But then, since $X_{\K}$ only contains
        two elements, it must be the case that $w_1=w_2$ and $u=v$.
        Hence, $(u,v)\in {}_{[z]}{\leqslant_{X_{\K}}}\subseteq \psi(\neg_{\koflbf}m)$
        by~(\ref{PsiDef}).
        
        \vskip 0.1cm
        \item If $u,v\in {}^{[z]}\!X_{\K}$, then
        the proof is similar to the previous case. \qedhere
    \end{itemize}
\end{proof}

\begin{lemma}\label{Lem:PresLattice}
	The map $\psi:(K\backslash\{a_0\})\cup L\to \mathsf{Up}\left(E_{\koflbf}, \preccurlyeq_{E_{\koflbf}}\right)$
	preserves the lattice operations $\wedge_{\koflbf}$ and $\vee_{\koflbf}$, i.e., for 
	$m_1,m_2\in (K{\setminus}\{a_0\})\cup L$,
	\[\psi(m_1\wedge_{\koflbf} m_2)=\psi(m_1)\cap\psi(m_2)
	\quad\text{and}\quad
	\psi(m_1\vee_{\koflbf} m_2)=\psi(m_1)\cup\psi(m_2).\]
\end{lemma}
\begin{proof}
We use the definitions of $\wedge_{\koflbf}$ and $\vee_{\koflbf}$ as
described in~(\ref{MeetDefinition}) and~(\ref{JoinDef}). Their commutativity reduces the number of cases we need to consider. 

\vskip 0.2cm
\noindent
\underline{Case 1:} $m_1 = a_{-1}$ or $m_2=a_{-1}$. Without loss of generality, suppose $m_1=a_{-1}$. Then 
\begin{align*}
\psi(m_1\wedge_{\koflbf}m_2)&=\psi(a_{-1})=\varnothing
=\varnothing \cap \psi(m_2) = \psi(m_1) \cap \psi(m_2),\text{ and }\\
\psi(m_1\vee_{\koflbf}m_2)&=\psi(m_2)=\varnothing \cup \psi(m_2) = \psi(m_1) \cup \psi(m_2).
\end{align*}
\underline{Case 2:} $m_1=a_1$ or $m_2=a_1$. Without loss of generality, suppose $m_1=a_1$. Then 
\begin{align*}
\psi(m_1\wedge_{\koflbf}m_2)&=\psi(m_2)=E_{\koflbf} \cap \psi(m_2) 
=\psi(m_1)\cap\psi(m_2),\text{ and }\\
\psi(m_1\vee_{\koflbf}m_2)&=\psi(a_{1})=E_{\koflbf}=E_{\koflbf} \cup \psi(m_2)
=\psi(m_1)\cup\psi(m_2).
\end{align*}

	\noindent
\underline{Case 3:} Finally, if $m_1,m_2\in L$, then 
$m_1\wedge_{\koflbf} m_2=m_1\wedge_{\Lalg} m_2$ and $m_1\vee_{\koflbf} m_2=m_1\wedge_{\Lalg} m_2$. 

Recall the relation $R={\leqslant_{X_{\mathbf{K[L]}}}} {\setminus} {\leqslant_{\mathbf{X_L}}}$ defined in~(\ref{Equation:RelationRDef}). We have  $\psi(m)=R \cup \varphi_{\Lalg}(m)$ for any $m\in L$.   Then, since 
$\varphi_{\Lalg}$ preserves $\wedge_\Lalg$ and $\vee_\Lalg$, we have
    \begin{align*}
    \psi(m_1\wedge_{\koflbf} m_2)&=R\cup\varphi_{\Lalg}(m_1\wedge_{\Lalg} m_2)=R\cup \left(\varphi_{\Lalg}(m_1)
    \cap\varphi_{\Lalg}(m_2)\right)\\
    &=
    \left(R\cup\varphi_{\Lalg}(m_1)\right)\cap
    \left(R\cup\varphi_{\Lalg}(m_2)\right)
    =\psi(m_1)\cap\psi(m_2),
    \end{align*}
     and
    \begin{align*}
    \psi(m_1\vee_{\koflbf} m_2)&=R\cup\varphi_{\Lalg}(m_1\vee_{\Lalg} m_2)=R\cup \left(\varphi_{\Lalg}(m_1)
    \cup\varphi_{\Lalg}(m_2)\right)\\
    &=
    \left(R\cup\varphi_{\Lalg}(m_1)\right)\cup
    \left(R\cup\varphi_{\Lalg}(m_2)\right)
    =\psi(m_1)\cup\psi(m_2). \qedhere
    \end{align*}
\end{proof}

Lastly, we need to show that $\psi$ preserves the monoid operation. \\

\begin{lemma}\label{Lem:PresMonoid}
	The map $\psi:K\backslash\{a_0\}\cup L\to \mathsf{Up}\left(E_{\koflbf}, \preccurlyeq_{E_{\koflbf}}\right)$
	preserves the monoid operation, $\cdot_{\koflbf}$, i.e., for 
	$m_1,m_2\in K\backslash\{a_0\}\cup L$,
	\[\psi(m_1\cdot_{\koflbf} m_2)=\psi(m_1)\mathbin{;}\psi(m_2).\]
\end{lemma}

\begin{proof} 
Recall the definition of $\cdot_{\koflbf}$
    from~(\ref{OperationDef}). 

\vskip 0.2cm
\noindent
\underline{Case 1:} $m_1=a_{-1}$ or $m_2=a_{-1}$. Without loss of generality, assume $m_1=a_{-1}$. Then  
$$ \psi(m_1 \cdot_{\mathbf{K[L]}} m_2)=\psi(a_{-1})=\varnothing = \varnothing \mathbin{;} \psi(m_2) = \psi(a_{-1}) \mathbin{;} \psi(m_2) = \psi(m_1)\mathbin{;}\psi(m_2). $$

\noindent \underline{Case 2:} $m_1=a_1=m_2$. Then $m_1\cdot_{\koflbf} m_2=m_1\cdot_{\K} m_2=a_{1}$ and $\psi(m_1\cdot_{\koflbf}m_2)= \psi(a_{1})=E_{\koflbf}=E_{\koflbf}\mathbin{;}E_{\koflbf}=\psi(m_1)\mathbin{;}\psi(m_2)$, where the equality  $E_{\koflbf}=E_{\koflbf}\mathbin{;}E_{\koflbf}$ follows from $E_{\koflbf}$ being reflexive and transitive. 	\\

\noindent \underline{Case 3:} If $m_1,m_2\in L$, then $m_1\cdot_{\koflbf} m_2=m_1\cdot_{\Lalg} m_2$. We first show that $\psi(m_1\cdot_{\koflbf}m_2)\subseteq \psi(m_1)\mathbin{;}\psi(m_2)$.  Let
$(u,v)\in \psi(m_1\cdot_{\koflbf}m_2)$. 
By (\ref{PsiDef}), there are six subcases to
consider:
	
	\begin{itemize}
	    \item If $(u,v)\in {}^{[z]}{\leqslant_{X_{\K}}}$
     for some $[z]\in X_{\Lalg}/E_{\Lalg}$, then 
     $(u,v)\in\psi(m_1)$ and $(v,v)\in\psi(m_2)$
     which implies that $(u,v)\in\psi(m_1)\mathbin{;}\psi(m_2)$.

     \vskip 0.1cm   
     \item If $(u,v)\in {}_{[z]}{\leqslant_{X_{\K}}}$
     for some $[z]\in X_{\Lalg}/E_{\Lalg}$ the proof is 
     similar to the proof of the previous case.

     \vskip 0.1cm   
     \item If $(u,v)\in [z]\times {}^{[z]}\!X_{\K}$
     for some $[z]\in X_{\Lalg}/E_{\Lalg}$, then we have
     $(u,v)\in \psi(m_1)$ and $(v,v)\in {}^{[z]}{\leqslant_{X_{\K}}}\subseteq\psi(m_2)$, which means 
     $(u,v)\in\psi(m_1)\mathbin{;}\psi(m_2)$.

     \vskip 0.1cm   
     \item If $(u,v)\in {}_{[z]}\!X_{\K}\times {}^{[z]}\!X_{\K}$
     for some $[z]\in X_{\Lalg}/E_{\Lalg}$, the proof is 
     similar to the proof of the previous case.

     \vskip 0.1cm   
     \item If $(u,v)\in {}_{[z]}\!X_{\K}\times [z]$
     for some $[z]\in X_{\Lalg}/E_{\Lalg}$, then
     $(u,u)\in {}_{[z]}{\leqslant_{X_{\K}}}\subseteq
     \psi(m_1)$, $(u,v)\in\psi(m_2)$ and hence
     $(u,v)\in\psi(m_1)\mathbin{;}\psi(m_2)$.

     \item 
     If $(u,v)\in\varphi_{\Lalg}(m_1\cdot_{\Lalg}m_2)$,
     then since $\varphi_{\Lalg}$ 
     preserves $\cdot_{\Lalg}$, 
     $(u,v)\in \varphi_{\Lalg}(m_1)\mathbin{;}\varphi_{\Lalg}(m_2)$.
     Therefore, there exists some $u_1\in X_{\Lalg}$ such that
     $(u,u_1)\in\varphi_{\Lalg}(m_1)\subseteq\psi(m_1)$
     and $(u_1,v)\in \varphi_{\Lalg}(m_2)\subseteq\psi(m_2)$.
     It follows that $(u,v)\in\psi(m_1)\mathbin{;}\psi(m_2)$.
     \end{itemize}

    \vskip 0.2cm
    For the inclusion in the other direction, 
    let $(u,v)\in\psi(m_1)\mathbin{;}\psi(m_2)$.  Then there
    exists some $u_1\in X_{\koflbf}$ such that 
    $(u,u_1)\in\psi(m_1)$ and $(u_1,v)\in\psi(m_2)$.
    There are three subcases to consider:
    \begin{itemize}
        \item If $u_1\in {}_{[z]}\!X_{\K}$  for some $[z]\in X_{\Lalg}/E_{\Lalg}$, then
        $(u,u_1)\in\psi(m_1)$ implies $u\in {}_{[z]}\!X_{\K}$.  On the other hand, $(u_1,v)\in\psi(m_2)$
        implies $v\in {}_{[z]}\!X_{\K}$, or 
        $v\in [z]$, or $v\in {}^{[z]}\!X_{\K}$.
        These three cases in turn respectively imply that
        $(u,v)\in {}_{[z]}{\leqslant_{X_{\K}}}$, or 
        $(u,v)\in {}_{[z]}\!X_{\K}\times[z]$, or
        $(u,v)\in {}_{[z]}\!X_{\K}\times {}^{[z]}\!X_{\K}$.
        Hence, $(u,v)\in\psi(m_1\cdot_{\koflbf}m_2)$.

        \vskip 0.1cm    
        \item If $u_1\in {}^{[z]}\!X_{\K}$  for some $[z]\in X_{\Lalg}/E_{\Lalg}$,
        then
        $(u_1,v)\in\psi(m_2)$ implies that $v\in {}^{[z]}\!X_{\K}$.  On the other hand, $(u,u_1)\in\psi(m_1)$
        implies that $u\in {}_{[z]}\!X_{\K}$, or 
        $u\in [z]$, or $u\in {}^{[z]}\!X_{\K}$.
        Respectively, these possibilities imply that
        $(u,v)\in {}_{[z]}\!X_{\K}\times {}^{[z]}\!X_{\K}$, or 
        $(u,v)\in [z]\times {}^{[z]}\!X_{\K}$, or
        $(u,v)\in {}^{[z]}{\leqslant_{X_{\K}}}$.
        Hence, $(u,v)\in\psi(m_1\cdot_{\koflbf}m_2)$.
        
        \item If $u_1\in [z]$  for some $[z]\in X_{\Lalg}/E_{\Lalg}$, then there are four 
        possibilities:
        \begin{itemize}
            \item $u\in {}_{[z]}\!X_{\K}$,
            $v\in [z]$ and hence 
            $(u,v)\in {}_{[z]}\!X_{\K}\times[z]$.
            \item $u\in {}_{[z]}\!X_{\K}$,
            $v\in {}^{[z]}\!X_{\K}$ and hence
            $(u,v)\in {}_{[z]}\!X_{\K}\times{}^{[z]}\!X_{\K}$.
            \item $u\in [z]$,
            $v\in {}^{[z]}\!X_{\K}$ and hence
            $(u,v)\in [z]\times{}^{[z]}\!X_{\K}$.
            \item $u, v\in [z]$, 
           and therefore it must be the case that $(u,u_1) \in \varphi_{\Lalg}(m_1)$ and $(u_1, v) \in \varphi_{\Lalg}(m_2)$, which gives 
            $(u,v)\in\varphi_{\Lalg}(m_1)\mathbin{;}\varphi_{\Lalg}(m_2)=\varphi_{\Lalg}(m_1\cdot_{\Lalg}m_2)$.
        \end{itemize}
        For all of these possibilities it follows that 
        $(u,v)\in\psi(m_1\cdot_{\koflbf}m_2)$.
    \end{itemize}

    \vskip 0.2cm 
    
\noindent \underline{Case 4:} $m_1=a_1$ and $m_2 \in L$. If $m_1=a_{1}$, then $\psi(m_1)=E_{\koflbf}$. Since $\psi(m_2)\subseteq E_{\koflbf}$ 	it follows that   $\psi(m_1)\mathbin{;}\psi(m_2)=E_{\koflbf}\mathbin{;} \psi(m_2)\subseteq E_{\koflbf}=\psi(m_1)= \psi(m_1\cdot_{\koflbf} m_2)$. 

    	For the inclusion in the other direction, let $(u,v)\in E_{\koflbf}$.
        \begin{itemize}
            \item If $u\in X_{\koflbf}$ and $v\in {}_{[z]}X_{\K}$ for some
            $[z]\in X_{\Lalg}/E_{\Lalg}$, then $(v,v)\in\psi(m_2)$ and
            $(u,v)\in E_{\koflbf}\mathbin{;}\psi(m_2)$.

            \vskip 0.1cm
            \item If $u\in X_{\koflbf}$ and 
            $v\in {}^{[z]}\!X_{\K}$ for some
            $[z]\in X_{\Lalg}/E_{\Lalg}$, then $(v,v)\in\psi(m_2)$ and
            $(u,v)\in E_{\koflbf}\mathbin{;}\psi(m_2)$.

            \vskip 0.1cm
            \item If $u\in {}_{[v]}\!X_{\K}$ and $v\in X_{\Lalg}$, then $(u,u)\in E_{\koflbf}$ and
            $(u,v)\in \psi(m_2)$,  and so
            $(u,v)\in E_{\koflbf}\mathbin{;}\psi(m_2)$.

            \vskip 0.1cm
            \item If $u\in {}^{[v]}\!X_{\K}$ and $v\in X_{\Lalg}$, then there exists $w\in X_{\K}$
            such that $u=w^{[v]}$.  Thus,
            $(u,w_{[v]})\in E_{\koflbf}$ and
            $(w_{[v]},v)\in \psi(m_2)$, and
            hence 
            $(u,v)\in E_{\koflbf}\mathbin{;}\psi(m_2)$.

            \vskip 0.1cm
            \item If $u, v\in X_{\Lalg}$, then
            $[u]=[v]$ in $X_{\Lalg}/E_{\Lalg}$.
            Let $w\in X_{\K}\neq \varnothing$. Thus
            we have that $(u,w_{[u]})\in E_{\koflbf}$ and
            $(w_{[u]},v)\in \psi(m_2)$, and hence 
            $(u,v)\in E_{\koflbf}\mathbin{;}\psi(m_2)$. \\ 
        \end{itemize}

\noindent \underline{Case 5:} If $m_1\in L$ and $m_2=a_1$, then the proof that $\psi(m_1\cdot_{\koflbf} m_2)=\psi(m_2)= \psi(m_1)\mathbin{;}\psi(m_2)$ is similar to the proof of the previous case.	
\end{proof}

Combining Lemmas~\ref{Lem:PsiInject},~\ref{Lem:PresLinNeg},~\ref{Lem:PresSelfInvol},~\ref{Lem:PresLattice} and~\ref{Lem:PresMonoid} 
we have shown that the map 
$\psi$
is a DqRA embedding from $(K\backslash\{a_0\})\cup L$ into $\mathsf{Up}\left(E_{\koflbf}, \preccurlyeq_{E_{\koflbf}}\right)$. Regarding the finiteness of the representations, it is clear from the construction of $\mathbf{X}_{\mathbf{S}_3[\Lalg]}$ that if $\Lalg$ is finitely representable (i.e. $\mathbf{X}_\Lalg$ is finite), then $\mathbf{S}_3[\Lalg]$ is finitely representable.  
This completes the proof of Theorem~\ref{Thm:K[L]represenatble}.\\

\section{Finite representations of finite Sugihara chains}\label{sec:applications}


In this section,
we show how the results of Section~\ref{sec:KL-rep} lead to finite representations of all finite Sugihara chains, and give some examples.
A consequence of this is that the result in Theorem~\ref{Thm:K[L]represenatble} can be generalised to Corollary~\ref{cor:Sn[L]-rep}, which states that if $\Lalg$ is a representable distributive quasi relation algebra, then the nested sum of a finite odd Sugihara chain $\mathbf{S}_n$ and $\Lalg$ is representable.

Proposition~\ref{prop:K=Sn+L=Sm} gives us that $\mathbf{S}_3[\mathbf{S}_{n-2}]\approx \mathbf{S}_{n}$. 
Using Examples~\ref{ex:S2} and~\ref{ex:S3} and Theorem~\ref{Thm:K[L]represenatble} we can apply a simple inductive argument to obtain the result below. \\

\begin{theorem}\label{thm:fin-rep-Sn}
Every finite Sugihara chain is finitely representable. \\
\end{theorem}

The construction used in the proof of Theorem~\ref{Thm:K[L]represenatble} provides alternative representations to those obtained in~\cite{Mad2010} and~\cite{CR-Sugihara} where the representations for $\mathbf{S}_n$ were infinite for $n \geqslant 4$.\\

\begin{figure}[ht]
\centering
\begin{tikzpicture}[scale=1.5,pics/sample/.style={code={\draw[#1] (0,0) --(0.6,0) ;}},
Dotted/.style={
dash pattern=on 0.1\pgflinewidth off #1\pgflinewidth,line cap=round,
shorten >=#1\pgflinewidth/2,shorten <=#1\pgflinewidth/2},
Dotted/.default=3]
\begin{scope}[xshift=-5cm, box/.style = {draw,dashdotdotted,inner sep=20pt,rounded corners=5pt,thick}]
\node[draw,circle,inner sep=1.5pt] (x) at (-2.5,0.7) {};
\node[draw,circle,inner sep=1.5pt] (y) at (-0.5,0.7) {};
\path (x) edge [->, bend left=25, dashed, thick] node {} (y);
\path (y) edge [->, bend left=25, dashed, thick] node {} (x);
\draw [->,dotted, thick] (x) edge[loop left]node{} (x);
\draw [->,dotted, thick] (y) edge[loop right]node{} (y);
\node[label,anchor=north,xshift=-1pt] at (x) {$x$};
\node[label,anchor=north,xshift=-1pt] at (y) {$y$};
\path (1,3.5) 
 node[matrix,anchor=north east,draw,nodes={anchor=center},inner sep=2pt, thick]  {
 \pic{sample=solid}; & \node{$\leqslant$}; \\
  \pic{sample=dashed}; & \node{$\alpha$}; \\
  \pic{sample=dotted}; & \node{$\beta$}; \\
  \pic{sample=dashdotdotted}; & \node{$E$ blocks}; \\
 };
 \node[box,fit=(x)(y)] {};
\node[label,anchor=north,xshift=43pt,yshift=-25pt] at (x) {$\mathbf{X}_{\mathbf{S}_3}$};
\end{scope}

\begin{scope}[xshift=-1cm,box/.style = {draw,dashdotdotted,inner sep=20pt,rounded corners=5pt,thick}]
\node[draw,circle,inner sep=1.5pt] (u) at (-2.5,0.7) {};
\draw [->,dotted, thick] (u) edge[loop right, looseness=40]node{} (u);
\draw [->,dashed, thick] (u) edge[loop above, looseness=40]node{} (u);
\node[label,anchor=north,xshift=-1pt] at (u) {$u$};
\node[box,fit=(u)] {};
\node[label,anchor=north,xshift=8pt,yshift=-25pt] at (u) {$\mathbf{X}_{\mathbf{S}_2}$};
\end{scope}

\begin{scope}[xshift=-0.5cm, yshift=1.8cm,box/.style = {draw,dashdotdotted,inner sep=20pt,rounded corners=5pt,thick}]
\node[draw,circle,inner sep=1.5pt] (a) at (-1,1) {};
\node[draw,circle,inner sep=1.5pt] (b) at (1,1) {};
\node[draw,circle,inner sep=1.5pt] (c) at (0,0) {};
\node[draw,circle,inner sep=1.5pt] (d) at (-1,-1) {};
\node[draw,circle,inner sep=1.5pt] (e) at (1,-1) {};
\draw[order] (a)--(c)--(d);
\draw[order] (b)--(c)--(e);
\node[label,anchor=south,yshift=-1pt] at (a) {$x^{[u]}$};
\node[label,anchor=south,yshift=-2pt, xshift=5pt] at (b) {$y^{[u]}$};
\node[label,anchor=north] at (d) {$x_{[u]}$};
\node[label,anchor=north, xshift=5pt] at (e) {$y_{[u]}$};
\node[label,anchor=east] at (c) {$u$};
\path (a) edge [->, bend left=23, dashed, thick] node {} (b);
\path (b) edge [->, bend left=23, dashed, thick] node {} (a);
\path (d) edge [->, bend left=23, dashed, thick] node {} (e);
\path (e) edge [->, bend left=23, dashed, thick] node {} (d);
\draw [->,dashed, thick] (c) edge[loop above,looseness=40]node{} (c);
\path (a) edge [->, bend left=23, dotted, thick] node {} (d);
\path (b) edge [->, bend left=23, dotted, thick] node {} (e);
\path (d) edge [->, bend left=23, dotted, thick] node {} (a);
\path (e) edge [->, bend left=23, dotted, thick] node {} (b);
\draw [->,dotted, thick] (c) edge[loop right,looseness=40]node{} (c);
\node[box,fit=(a)(b)(d)(e)(c),] {};
\node[label,anchor=north,xshift=40pt,yshift=-25pt] at (d) {$\mathbf{X}_{\mathbf{S}_{3}[\mathbf{S}_2]}$};
\end{scope}

\begin{scope}[xshift=-6.6cm,yshift=-2cm]
\node[draw,circle,inner sep=1.5pt,fill] (m) at (0,-0.4) {};
\node[draw,circle,inner sep=1.5pt,fill] (n) at (0,0.3) {};
\node[draw,circle,inner sep=1.5pt,fill] (o) at (0,1) {};
\draw[order] (m)--(n)--(o);
\node[label,anchor=west] at (o) {$(X_{\mathbf{S}_3})^2$};
\node[label,anchor=west] at (n) {$\leqslant_{X_{\mathbf{S}_3}}$};
\node[label,anchor=west] at (m) {$\varnothing$};
\node[label,anchor=north,yshift=-7pt] at (m) {$\mathbf{S}_3$};
\end{scope}

\begin{scope}[xshift=-3.7cm,yshift=-2cm]
\node[draw,circle,inner sep=1.5pt,fill] (0) at (0,0.3) {};
\node[draw,circle,inner sep=1.5pt,fill] (1) at (0,1) {};
\draw[order] (0)--(1);
\node[label,anchor=west] at (1) {$(X_{\mathbf{S}_2})^2$};
\node[label,anchor=west] at (0) {$\varnothing$};
\node[label,anchor=north,yshift=-7pt] at (0) {$\mathbf{S}_2$};
\end{scope}

\begin{scope}[xshift=-1.6cm,yshift=-2cm]
\node[draw,circle,inner sep=1.5pt,fill] (p) at (0,-1.1) {};
\node[draw,circle,inner sep=1.5pt,fill] (q) at (0,-0.4) {};
\node[draw,circle,inner sep=1.5pt,fill] (r) at (0,0.3) {};
\node[draw,circle,inner sep=1.5pt,fill] (s) at (0,1) {};
\draw[order] (p)--(q)--(r)--(s);
\node[label,anchor=west] at (s) {$\left\{u,x^{[u]},y^{[u]},x_{[u]},y_{[u]}\right\}^2$};
\node[label,anchor=west] at (r) {$\leqslant_{X_{\mathbf{S}_{3}[\mathbf{S}_2]}}$};
\node[label,anchor=west] at (q) {$\leqslant_{X_{\mathbf{S}_{3}
[\mathbf{S}_2]}}{\setminus}\{(u,u)\}$};
\node[label,anchor=west] at (p) {$\varnothing$};
\node[label,anchor=north,yshift=-7pt] at (p) {$\mathbf{S}_3[\mathbf{S}_2]\approx\mathbf{S}_4$};
\end{scope}

\end{tikzpicture}
\caption{The posets used to represent $\mathbf{S}_3,
\mathbf{S}_2$ and $\mathbf{S}_3[\mathbf{S}_2]\approx\mathbf{S}_4$,
respectively, and their respective representations.}
\label{Fig:S3-S2}
\end{figure}

\begin{example}\label{ex:fin-rep-S4} Let $\mathbf{K}=\mathbf{S}_3$ and $\mathbf{L}=\mathbf{S}_2$. Then 
$\mathbf{K}[\mathbf{L}]\approx \mathbf{S}_4$ {\upshape(}by Proposition~\ref{prop:K=Sn+L=Sm}{\upshape)} and $\mathbf{K}[\mathbf{L}]$ is representable {\upshape(}by Theorem~\ref{Thm:K[L]represenatble}{\upshape)}. Recall that $\mathbf{S}_2$ is representable over a one-element poset and $\mathbf{S}_3$ over a two-element discrete poset with an order automorphism interchanging the elements of the poset. If $X_\K = \{x,y\}$ and $X_{\Lalg} = \{u\}$, then $\mathbf{S}_4$ is representable over $\mathbf{X}_{\koflbf} = (X_{\koflbf},\leqslant_{X_{\koflbf}})$ with $X_{\K[\Lalg]} =\{u, x^{[u]},y^{[u]}, x_{[u]}, y_{[u]}\}$, $$\leqslant_{X_{\K[\Lalg]}} = \textnormal{id}_{X_{\K[\Lalg]}} \cup \left[\{x_{[u]}, y_{[u]}\}\times \{u\}\right] \cup \left[\{u\}\times \{x^{[u]}, y^{[u]}\}\right]\cup \left[\{x_{[u]}, y_{[u]}\}\times \{x^{[u]}, y^{[u]}\}\right]$$ 
and $E_{\K[\Lalg]}= (X_{\K[\Lalg]})^2$. Here $\alpha_{\K[\Lalg]} = \{(u, u), (x_{[u]}, y_{[u]}), (y_{[u]}, x_{[u]}), (x^{[u]}, y^{[u]}), (y^{[u]}, x^{[u]})\}$ and $\beta_{\K[\Lalg]} = \{(u, u), (x_{[u]}, x^{[u]}), (x^{[u]}, x_{[u]}), (y_{[u]}, y^{[u]}), (y^{[u]},y_{[u]})\}$. The posets $\mathbf{X}_\K, \mathbf{X}_\Lalg$ and $\mathbf{X}_{\K[\Lalg]}$ are depicted in Figure~\ref{Fig:S3-S2}. It also shows the representations of $\mathbf{S}_3$, $\mathbf{S}_2$ and $\mathbf{S}_4$. 
\\
\end{example}

\begin{example}\label{ex:fin-rep-S5}
Let $\mathbf{K}=\mathbf{S}_3$ and $\mathbf{L}=\mathbf{S}_3$. Then 
$\mathbf{K}[\mathbf{L}]\approx \mathbf{S}_5$ {\upshape(}by Proposition~\ref{prop:K=Sn+L=Sm}{\upshape)} and $\mathbf{K}[\mathbf{L}]$ is representable {\upshape(}by Theorem~\ref{Thm:K[L]represenatble}{\upshape)}.  If $X_{\K}= \{x, y\}$ and $X_{\Lalg}= \{u, v\}$, then $\mathbf{S}_5$ is representable over $\mathbf{X}_{\koflbf} = (X_{\koflbf},\leqslant_{X_{\koflbf}})$ with $X_{\K[\Lalg]} = \{u, v, x^{[u]},y^{[u]}, x_{[u]}, y_{[u]}\}$,  
$$\leqslant_{X_{\K[\Lalg]}} = \textnormal{id}_{X_{\K[\Lalg]}} \cup \left[\{x_{[u]}, y_{[u]}\}\times \{u, v\}\right] \cup \left[\{u, v\}\times \{x^{[u]}, y^{[u]}\}\right] \cup  \left[\{x_{[u]}, y_{[u]}\}\times \{x^{[u]}, y^{[u]}\}\right],$$
$E_{\K[\Lalg]}= (X_{\K[\Lalg]})^2$, 
$\alpha_{\K[\Lalg]} = \{(u, v), (v, u), (x_{[u]}, y_{[u]}), (y_{[u]}, x_{[u]}), (x^{[u]}, y^{[u]}), (y^{[u]}, x^{[u]})\}$ and $\beta_{\K[\Lalg]} = \{(u, u), (v, v),  (x_{[u]}, x^{[u]}), (x^{[u]}, x_{[u]}), (y_{[u]}, y^{[u]}), (y^{[u]},y_{[u]})\}$. Figure~\ref{Fig:S3[S3]} shows the posets $\mathbf{X}_\K, \mathbf{X}_\Lalg$ and $\mathbf{X}_{\K[\Lalg]}$, and the representation of $\mathbf{S}_5$. \\
\end{example}

In Figure~\ref{fig:posets-for-S2-S7} we show how to extend the posets from Examples~\ref{ex:fin-rep-S4} and~\ref{ex:fin-rep-S5}
to give finite posets which can be used to represent $\mathbf{S}_6$ and $\mathbf{S}_7$.

\begin{figure}
\centering    
    \begin{tikzpicture}[scale=1.1,pics/sample/.style={code={\draw[#1] (0,0) --(0.6,0) ;}},
    Dotted/.style={
    dash pattern=on 0.1\pgflinewidth off #1\pgflinewidth,line cap=round,
    shorten >=#1\pgflinewidth/2,shorten <=#1\pgflinewidth/2},
    Dotted/.default=3]

\begin{scope}[xshift=-2cm,box/.style = {draw,dashdotdotted,inner sep=25pt,rounded corners=5pt,thick}]
\node[draw,circle,inner sep=1.5pt] (y) at (-1,1.5) {};
\node[draw,circle,inner sep=1.5pt] (z) at (1,1.5) {};
\node[draw,circle,inner sep=1.5pt] (w) at (-1,0) {};
\node[draw,circle,inner sep=1.5pt] (x) at (1,0) {};
\node[draw,circle,inner sep=1.5pt] (u) at (-1,-1.5) {};
\node[draw,circle,inner sep=1.5pt] (v) at (1,-1.5) {};
\draw[order] (u)--(w)--(y);
\draw[order] (v)--(x)--(z);
\draw[order] (u)--(x);
\draw[order] (u)--(w);
\draw[order] (v)--(x);
\draw[order] (v)--(w);
\draw[order] (w)--(z);
\draw[order] (w)--(y);
\draw[order] (x)--(y);
\draw[order] (x)--(z);
\node[label,anchor=east] at (w) {$u$};
\node[label,anchor=west] at (x) {$v$};
\node[label,anchor=east] at (y) {$x^{[u]}$};
\node[label,anchor=west] at (z) {$y^{[u]}$};
\node[label,anchor=east] at (u) {$x_{[u]}$};
\node[label,anchor=west] at (v) {$y_{[u]}$};
\path (y) edge [<->, bend left=23, dashed, thick] node {} (z);
\path (z) edge [->, bend left=23, dashed, thick] node {} (y);
\path (u) edge [->, bend left=23, dashed, thick] node {} (v);
\path (v) edge [->, bend left=23, dashed, thick] node {} (u);
\draw [->,dashed, thick] (w) edge[loop,looseness=20]node{} (w);
\draw [->,dashed, thick] (x) edge[loop,looseness=20]node{} (x);
\path (y) edge [->, bend left=40, dotted, thick] node {} (u);
\path (u) edge [->, bend left=40, dotted, thick] node {} (y);
\path (z) edge [->, bend left=40, dotted, thick] node {} (v);
\path (v) edge [->, bend left=40, dotted, thick] node {} (z);
\draw [->,dotted, thick] (w) edge[loop right,looseness=40]node{} (w);
\draw [->,dotted, thick] (x) edge[loop left,looseness=40]node{} (x);
\node[box,fit=(x)(y)(z)(w)(v)(u)] {};
\path (4.5,0.9) 
 node[matrix,anchor=north east,draw,nodes={anchor=center},inner sep=2pt, thick]  {
 \pic{sample=solid}; & \node{$\leqslant$}; \\
  \pic{sample=dashed}; & \node{$\alpha$}; \\
  \pic{sample=dotted}; & \node{$\beta$}; \\
  \pic{sample=dashdotdotted}; & \node{$E$ blocks}; \\
 };

\node[label,anchor=north,xshift=40pt,yshift=-28pt] at (u) {$\mathbf{X}_{\mathbf{S}_3[\mathbf{S}_3]}$};
\end{scope}

\begin{scope}[xshift=3.8cm]
\node[draw,circle,inner sep=1.5pt,fill] (h) at (0,-2) {};
\node[draw,circle,inner sep=1.5pt,fill] (i) at (0,-1) {};
\node[draw,circle,inner sep=1.5pt,fill] (j) at (0,0) {};
\node[draw,circle,inner sep=1.5pt,fill] (k) at (0,1) {};
\node[draw,circle,inner sep=1.5pt,fill] (l) at (0,2) {};
\draw[order] (h)--(i)--(j)--(k)--(l);
\node[label,anchor=west] at (h) {$\varnothing$};
\node[label,anchor=west] at (i) {$\leqslant_{X_{\mathbf{S}_3[\mathbf{S}_3]}}{\setminus}\,\{(u,u),(v,v)\}$};
\node[label,anchor=west] at (j) {$\leqslant_{X_{\mathbf{S}_3[\mathbf{S}_3]}}$};
\node[label,anchor=west] at (k) {$\leqslant_{X_{\mathbf{S}_3[\mathbf{S}_3]}}\cup\,\{(u,v),(v,u)\}$};
\node[label,anchor=west] at (l) {$(X_{\mathbf{S}_3[\mathbf{S}_3]})^2=
E_{\mathbf{S}_3[\mathbf{S}_3]}$};
\node[label,anchor=north,yshift=-10pt] at (h) {$\mathbf{S}_3[\mathbf{S}_3]\approx\mathbf{S}_5$};
\end{scope}

\end{tikzpicture}
\caption{The poset used to represent $\mathbf{S}_3[\mathbf{S}_3]\approx\mathbf{S}_5$
and its representation.}
\label{Fig:S3[S3]}
\end{figure}

\begin{figure}
\centering

\begin{tikzpicture}[scale=0.7]
\begin{scope}
\node[draw,circle,inner sep=1.5pt] (y) at (-1,1) {};
\node[draw,circle,inner sep=1.5pt] (z) at (1,1) {};
\node[draw,circle,inner sep=1.5pt] (x) at (0,0) {};
\node[draw,circle,inner sep=1.5pt] (v) at (-1,-1) {};
\node[draw,circle,inner sep=1.5pt] (w) at (1,-1) {};
\draw[order] (y)--(x)--(v);
\draw[order] (z)--(x)--(w);
	
\end{scope}

\begin{scope}[xshift=4cm]
\node[draw,circle,inner sep=1.5pt] (y) at (-1,1.5) {};
\node[draw,circle,inner sep=1.5pt] (z) at (1,1.5) {};
\node[draw,circle,inner sep=1.5pt] (w) at (-1,0) {};
\node[draw,circle,inner sep=1.5pt] (x) at (1,0) {};
\node[draw,circle,inner sep=1.5pt] (u) at (-1,-1.5) {};
\node[draw,circle,inner sep=1.5pt] (v) at (1,-1.5) {};
\draw[order] (u)--(w)--(y);
\draw[order] (v)--(x)--(z);
\draw[order] (u)--(x);
\draw[order] (u)--(w);
\draw[order] (v)--(x);
\draw[order] (v)--(w);
\draw[order] (w)--(z);
\draw[order] (w)--(y);
\draw[order] (x)--(y);
\draw[order] (x)--(z);

\end{scope}

\begin{scope}[xshift=8cm]
\node[draw,circle,inner sep=1.5pt] (r) at (-1,-2.5) {};
\node[draw,circle,inner sep=1.5pt] (s) at (1,-2.5) {};
\node[draw,circle,inner sep=1.5pt] (t) at (-1,-1) {};
\node[draw,circle,inner sep=1.5pt] (u) at (1,-1) {};
\node[draw,circle,inner sep=1.5pt] (v) at (0,0) {};
\node[draw,circle,inner sep=1.5pt] (w) at (1,1) {};
\node[draw,circle,inner sep=1.5pt] (x) at (-1,1) {};
\node[draw,circle,inner sep=1.5pt] (y) at (1,2.5) {};
\node[draw,circle,inner sep=1.5pt] (z) at (-1,2.5) {};
\draw[order] (r)--(t)--(v);
\draw[order] (s)--(u)--(v);
\draw[order] (v)--(w)--(y);
\draw[order] (v)--(x)--(z);
\draw[order] (r)--(u);
\draw[order] (s)--(t);
\draw[order] (w)--(z);
\draw[order] (x)--(y);

\end{scope}

\begin{scope}[xshift=12cm]
\node[draw,circle,inner sep=1.5pt] (q) at (-1,3) {};
\node[draw,circle,inner sep=1.5pt] (r) at (1,3) {};
\node[draw,circle,inner sep=1.5pt] (y) at (-1,1.5) {};
\node[draw,circle,inner sep=1.5pt] (z) at (1,1.5) {};
\node[draw,circle,inner sep=1.5pt] (w) at (-1,0) {};
\node[draw,circle,inner sep=1.5pt] (x) at (1,0) {};
\node[draw,circle,inner sep=1.5pt] (u) at (-1,-1.5) {};
\node[draw,circle,inner sep=1.5pt] (v) at (1,-1.5) {};
\node[draw,circle,inner sep=1.5pt] (s) at (-1,-3) {};
\node[draw,circle,inner sep=1.5pt] (t) at (1,-3) {};
\draw[order] (s)--(u)--(w)--(y)--(q);
\draw[order] (t)--(v)--(x)--(z)--(r);
\draw[order] (s)--(v);
\draw[order] (t)--(u);
\draw[order] (y)--(r);
\draw[order] (z)--(q);
\draw[order] (u)--(x);
\draw[order] (u)--(w);
\draw[order] (v)--(x);
\draw[order] (v)--(w);
\draw[order] (w)--(z);
\draw[order] (w)--(y);
\draw[order] (x)--(y);
\draw[order] (x)--(z);
	
\end{scope}
	
\end{tikzpicture}

\caption{Figure showing the posets used to represent $\mathbf{S}_4$ up to $\mathbf{S}_7$.}\label{fig:posets-for-S2-S7}
\end{figure}

Lastly, using Proposition~\ref{prop:K=Sn+L=Sm}
we get the isomorphic relationship 
$\mathbf{S}_n[\Lalg]\approx \mathbf{S}_3\left[ \mathbf{S}_{n-2}[\Lalg] \right]$. This isomorphism and Theorem~\ref{thm:fin-rep-Sn}  justify the corollary below. \\

\begin{corollary}\label{cor:Sn[L]-rep}
Let $\mathbf{K}=\mathbf{S}_n$ for $n \geqslant 3$ and $n$ odd. If\, $\mathbf{L}$ is a {\upshape(}finitely{\upshape)} representable DqRA, then   $\mathbf{S}_n[\mathbf{L}]$ is {\upshape(}finitely{\upshape)} representable. \\
\end{corollary}

\bmhead{Acknowledgements}
The authors are grateful to Peter Jipsen for drawing our attention to the $\K[\Lalg]$ construction for residuated lattices, and to Wesley Fussner for pointing out the recent use of the term `nested sum'.   
We also thank the referee for their helpful feedback which has improved the paper. 

\section*{Declarations}

\bmhead{Ethical approval} Not applicable. 

\bmhead{Funding} The first author acknowledges support from the National Research Foundation (NRF) of South Africa (grant 127266).

\bmhead{Availability of data and materials} Not applicable.



\end{document}